\documentclass[12pt]{article}
\usepackage{textcomp}%
\usepackage[dvipsnames]{xcolor}%
\usepackage{graphicx}%
\usepackage{multirow}%
\usepackage{amsmath,amssymb,amsfonts,amsthm,mathrsfs}%
\usepackage{ytableau}
\usepackage[font=footnotesize,labelfont=bf]{caption}
\usepackage{youngtab}
\usepackage{pgffor}
\usepackage{mathrsfs}%
\usepackage[title]{appendix}%
\usepackage[margin=3cm]{geometry}
\usepackage{mathdots}
\usepackage[colorlinks=true,linkcolor=blue,citecolor=blue,hyperfootnotes=false]{hyperref}
\usepackage{cleveref}%
\usepackage{authblk}
\usepackage{tikz}
\usepackage{ifthen}
\tikzstyle{Mytext} = [text width=0.5cm,text centered]
\usepackage{enumitem}
\usetikzlibrary{calc}
\usetikzlibrary{matrix,arrows,decorations.pathmorphing,decorations.markings,calc}
\usepackage{graphicx}
\usepackage{caption}
\usepackage{subcaption}
\usepackage{array}

\usepackage{ytableau}
\usepackage{parskip}
\usepackage{enumerate}

\newcommand{\ZZ}{ \mathbb{Z} }

\newcounter{x}
\newcounter{y}
\newcounter{z}
\newcounter{row}

\newcommand\xaxis{210}
\newcommand\yaxis{-30}
\newcommand\zaxis{90}

\newcommand\topside[3]{
  \fill[fill=YellowGreen, draw=black,shift={(\xaxis:#1)},shift={(\yaxis:#2)},
  shift={(\zaxis:#3)}] (0,0) -- (30:1) -- (0,1) --(150:1)--(0,0);
}

\newcommand\leftside[3]{
  \fill[fill=Orchid, draw=black,shift={(\xaxis:#1)},shift={(\yaxis:#2)},
  shift={(\zaxis:#3)}] (0,0) -- (0,-1) -- (210:1) --(150:1)--(0,0);
}

\newcommand\rightside[3]{
  \fill[fill=RawSienna, draw=black,shift={(\xaxis:#1)},shift={(\yaxis:#2)},
  shift={(\zaxis:#3)}] (0,0) -- (30:1) -- (-30:1) --(0,-1)--(0,0);
}

\newcommand\cube[3]{
  \topside{#1}{#2}{#3} \leftside{#1}{#2}{#3} \rightside{#1}{#2}{#3}
}

\newcommand\planepartition[1]{
    \setcounter{x}{-1}
    \foreach \a in {#1} {
        \addtocounter{x}{1}
        \setcounter{y}{-1}
        \foreach \b in \a {
            \addtocounter{y}{1}
            \setcounter{z}{-1}
            \ifthenelse{\b > 0}{
                \foreach \c in {1,...,\b} {
                    \addtocounter{z}{1}
                    \cube{\value{x}}{\value{y}}{\value{z}}
                }
            }{
                \topside{\value{x}}{\value{y}}{-1}
            }
        }
    }
}

\newcommand\highlightedreverseplanepartition[3]{
 \begin{tikzpicture}[scale = 0.8]
 \setcounter{x}{-1}
  \foreach \a in {#2} {
    \addtocounter{x}{1}
    \setcounter{y}{-1}

    \setcounter{row}{-1}
    \let\mylist\empty
    \foreach \b in \a {
        \addtocounter{row}{1}
        \ifx\mylist\empty
            \xdef\mylist{\b}%
        \else
            \xdef\mylist{\b,\mylist}%
        \fi
    }
    \foreach \b in \mylist {
      \addtocounter{y}{1}
      \setcounter{z}{-1}
      \ifthenelse{#1 > 0}{
        \foreach \c in {1,...,#1} {
            \addtocounter{z}{1}
            \rightside{\value{y}-\value{row}+1}{\value{x}}{\value{z}+1}
        }
      }{}
    }
  }
 
 \setcounter{x}{-1}
  \foreach \a in {#2} {
    \addtocounter{x}{1}
    \setcounter{y}{-1}

    \setcounter{row}{-1}
    \let\mylist\empty
    \foreach \b in \a {
        \addtocounter{row}{1}
        \ifx\mylist\empty
            \xdef\mylist{\b}%
        \else
            \xdef\mylist{\b,\mylist}%
        \fi
    }

    \setcounter{z}{-1}
    \foreach \c in {1,...,#1} {
        \addtocounter{z}{1}
        \leftside{\value{y}-\value{row}}{\value{x}}{\value{z}}
    }
    
    \foreach \b in \mylist {
      \addtocounter{y}{1}
      \setcounter{z}{-1}
      \ifthenelse{\b > 0}{
        \foreach \c in {1,...,\b} {
            \addtocounter{z}{1}
            \cube{\value{y}-\value{row}}{\value{x}}{\value{z}}
        }
      }{
        \topside{\value{y}-\value{row}}{\value{x}}{-1}
      }
      \
    }
  }

  \end{tikzpicture}
}

\newcommand\reverseplanepartition[2]{
 \setcounter{x}{-1}
  \foreach \a in {#2} {
    \addtocounter{x}{1}
    \setcounter{y}{-1}

    \setcounter{row}{-1}
    \let\mylist\empty
    \foreach \b in \a {
        \addtocounter{row}{1}
        \ifx\mylist\empty
            \xdef\mylist{\b}%
        \else
            \xdef\mylist{\b,\mylist}%
        \fi
    }
    \foreach \b in \mylist {
      \addtocounter{y}{1}
      \setcounter{z}{-1}
      \ifthenelse{#1 > 0}{
        \foreach \c in {1,...,#1} {
            \addtocounter{z}{1}
            \rightside{\value{y}-\value{row}+1}{\value{x}}{\value{z}+1}
        }
      }{}
    }
  }
 
 \setcounter{x}{-1}
  \foreach \a in {#2} {
    \addtocounter{x}{1}
    \setcounter{y}{-1}

    \setcounter{row}{-1}
    \let\mylist\empty
    \foreach \b in \a {
        \addtocounter{row}{1}
        \ifx\mylist\empty
            \xdef\mylist{\b}%
        \else
            \xdef\mylist{\b,\mylist}%
        \fi
    }

    \setcounter{z}{-1}
    \foreach \c in {1,...,#1} {
        \addtocounter{z}{1}
        \leftside{\value{y}-\value{row}}{\value{x}}{\value{z}}
    }
    
    \foreach \b in \mylist {
      \addtocounter{y}{1}
      \setcounter{z}{-1}
      \ifthenelse{\b > 0}{
        \foreach \c in {1,...,\b} {
            \addtocounter{z}{1}
            \cube{\value{y}-\value{row}}{\value{x}}{\value{z}}
        }
      }{
        \topside{\value{y}-\value{row}}{\value{x}}{-1}
      }
      \
    }
  }
}

\newcommand{\getXCoordinate}[1]{%
    \def\xcoordinate{}
    \def\ycoordinate{}
    \expandafter\splitCoordinate#1\relax
    \xcoordinate
}

\def\splitCoordinate#1,#2\relax{%
    \def\xcoordinate{#1}%
    \def\ycoordinate{#2}%
}

\newtheorem{theorem}{Theorem}[section]
\newtheorem{proposition}[theorem]{Proposition}%
\newtheorem{lemma}[theorem]{Lemma}
\newtheorem{cor}[theorem]{Corollary}
\newtheorem{example}[theorem]{Example}%
\newtheorem{remark}[theorem]{Remark}%
\newtheorem{definition}[theorem]{Definition}%
\newtheorem{result}{Result}

\title{Colored Vertex Models and Interacting Reverse Plane Partitions}

\author[1]{Jonah Guse\thanks{This work was done as part of an undergraduate reading course at the University of Wisconsin, Madison.}}

\author[1]{David Jiang\protect\footnotemark[1]}

\author[2]{David Keating\thanks{The third author was partially supported by the NSF RTG grant DMS 1937241.}}

\affil[1]{Department of Mathematics, University of Wisconsin,  Madison}
\affil[2]{Department of Mathematics, University of Illinois, Urbana-Champaign}

\date{}
\begin{document}

\maketitle

\abstract{We study the coupling of pairs of reverse plane partitions of the same shape by assigning a certain local interaction between the reverse plane partitions.We show that they are in bijection with a certain Yang-Baxter integrable colored vertex model. By utilizing the Yang-Baxter equation for this colored vertex model, we are able to compute the generating function for the interacting pairs of reverse plane partitions. We also give a bijection between the coupled pairs of reverse plane partitions with the interaction strength set to zero and a single reverse plane partition of the same shape.}

\setcounter{tocdepth}{2}
\tableofcontents

\section{Introduction}

In this paper, we use a multicolored Yang-Baxter integrable vertex model to construct a novel model of interacting pairs of reverse plane partitions. Our main result is a product formula for the generating function of interacting pairs of reverse plane partitions. 
\begin{result}[Theorem \ref{mainthm}]
For any fixed $\lambda$, the generating function for volume weighted pairs of interacting reverse plane partitions is given by
\begin{equation}
\sum_{\Lambda, \Lambda^\prime \in RPP(\lambda)} q^{|\Lambda| + |\Lambda'|}t^{g(\Lambda,\Lambda^\prime)} = \prod_{c \in \lambda} \frac{1}{(1-q^{h_{\lambda}(c)})(1-q^{h_{\lambda}(c)}t)}.
\end{equation}
\end{result}
Here $t$ is a positive real parameter governing the strength of the interaction between the pair of reverse plane partitions and $g(\Lambda,\Lambda^\prime)$ counts the number of occurrences of certain lozenge orientations when viewing the reverse plain partitions as tilings. Note that when $t=0$, the generating function reduces to that of a single reverse plane partition indicating that volume-weighted pairs of reverse plane partitions with interaction strength zero and fixed total volume are equinumerous with the volume-weighted single reverse plane partitions with the same fixed volumes. Our second main result is a bijective proof of this fact.
\begin{result}[Proposition \ref{prop:sliding}]
For any integer $n \ge 0$, there is a bijection between the pairs of reverse plane partitions $\Lambda, \Lambda' \in RPP(\lambda)$ satisfying $g(\Lambda,\Lambda') = 0$ and having total volume $|\Lambda| + |\Lambda'|=n$ and reverse plane partitions of shape $\lambda$ and volume $n$. 
\end{result}

The organization of this paper is as follows: In Section \ref{sec1} we give some preliminary material covering the basic definitions and result we will need on integer partitions, reverse plane partitions, and five-vertex models. In Section \ref{sec:vertexRPP} we show how the RPP are in bijection with path configurations of a certain five-vertex model. We then use the Yang-Baxter integrability of the vertex model to give an alternate proof of the well-known  product formula for the generating function of reverse plane partition. In Section \ref{sec:multicolored} we introduced the multicolored vertex models that generalize the five-vertex model we used in the previous sections. We show that these multicolored vertex models can be used to construct a model of interacting pairs of reverse plane partitions. We then then use the integrability of the vertex model to prove our main result Theorem \ref{mainthm}. Finally, in Section \ref{sec:tto0} we study the limit when the strength of the interaction between the reverse plane partitions goes to zero. We give a bijection between pairs of interacting reverse plane partitions when the interaction strength goes to zero and a single reverse plane partition of the same shape.

\section{Preliminaries}\label{sec1}

In this section we give some combinatorial background and define reverse plane partitions. We also introduce the Yang-Baxter integrable five vertex models which we will show are in bijection with the reverse plane partitions. For more background on partitions, plane partitions, and reverse plane partitions, see \cite{Andrews,Stanley}.
The five-vertex models with discuss in this section are closely related to Schur processes \cite{borodin2016lectures} and dimer models \cite{boutillier2017dimers}. For a more general background on vertex models see \cite{reshetikhin2010lectures}.

\subsection{Partitions}

\subsubsection{Integer partitions and Young Diagram}

To begin we define integer partitions.
\begin{definition}
An \emph{integer partition} of $n\in \ZZ_{>0}$ is a way of writing the positive integer $n$ as a sum of other positive integers, where the order of the summands does not matter. 
\end{definition}

We call each of the summands a \emph{part} of the partition and write $\lambda$ as a decreasing sequence of its parts. That is, $\lambda = (\lambda_1,\lambda_2,\ldots,\lambda_\ell)$ where $\lambda_1\ge \lambda_2 \ge \ldots \lambda_\ell >0$ and  $\lambda_1+ \lambda_2 + \ldots \lambda_\ell = n$. We will refer to the number of parts $\ell$ as the \emph{length} of $\lambda$. Note that we may extend $\lambda$ to a infinite sequence by setting $\lambda_i=0$ for $i> \ell$. In this case,  the length of $\lambda$ is the number of non-zero parts. We define the \emph{size} of a partition by $|\lambda| = \sum_{i} \lambda_i$.

A \emph{Young diagram} is a pictorial representation of an integer partition. Taking a certain partition $\lambda$, the Young diagram is generated by taking the largest part $\lambda_1$ of $\lambda$ and placing $\lambda_1$ many blocks on the first row, for the second largest part $\lambda_2$ placing $\lambda_2$ many blocks on the second row, and continuing until all parts of the partition have been represented with blocks. We draw our Young diagrams in the \emph{French convention} where the first row is on the bottom and subsequent rows are stacked on top.

\begin{example}$\\$
The Young diagrams
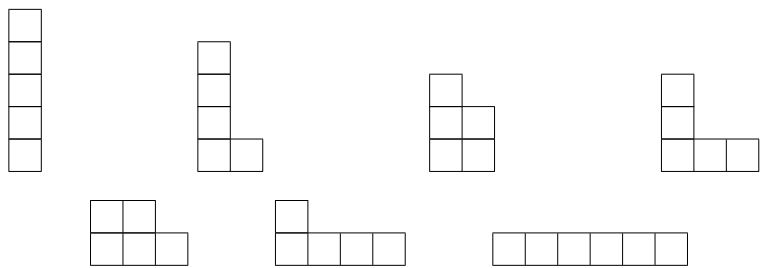
\begin{center}
    $\young(~,~,~,~,~)$\hspace{55pt}
    $\young(~,~,~,~~)$ \hspace{55pt}
    $\young(~,~~,~~)$ \hspace{55pt}
    $\young(~,~,~~~)$ \hspace{55pt} \\\vspace{10pt}
    $\young(~~,~~~)$ \hspace{25pt}
    $\young(~,~~~~)$ \hspace{25pt}
    $\young(~~~~~~)$
\end{center}
correspond to the partitions
$$5 = 1 + 1 + 1 + 1 + 1 \hspace{40pt}
    5 = 1 + 1 + 1 + 2 \hspace{40pt}
    5 = 1 + 2 + 2 \hspace{40pt}
    5 = 1 + 1 + 3$$$$
    5 = 2 + 3 \hspace{28pt}
    5 = 1 + 4\hspace{25pt}
    5 = 5$$.
\end{example}
\noindent We call a box in the Young diagram a \emph{cell}.

\subsubsection{Interlacing Partitions}
We say that two partitions $\mu$ and $\lambda$ \emph{interlace} if
\begin{equation}
\lambda_1\ge \mu_1 \ge \lambda_2 \ge \mu_2 \ge \ldots
\end{equation}
and write $\mu \preceq \lambda $.  Equivalently, $\mu \preceq \lambda $ if and only if one can get the Young diagram of $\lambda$ from the Young diagram of $\mu$ by adding at most one cell to each column.  We say that the two partitions differ by a \emph{horizontal strip}.\\\\

\begin{example}
The partitions $\mu=(3,2,1,1)$ and $\lambda= (5,2,2,1,1)$ satisfy $\mu \preceq \lambda$ since we can get $\lambda$ from $\mu$ by adding at most one cell to each column. Below we draw the Young diagram of $\mu$ in gray and the horizontal strip you need to add to get $\lambda$ in white.
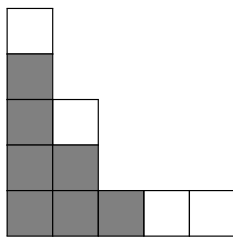
\[
\begin{ytableau}  
\empty \\
*(gray) \\
*(gray) & \empty \\
*(gray) &  *(gray)  \\
*(gray) &  *(gray) &  *(gray) & \empty & \empty \\
\end{ytableau}
\]
\end{example}

If we have two partitions $\mu$ and $\lambda$ such that the cells in the Young diagram of $\mu$ are a subset of the cells in the Young diagram of $\lambda$ then we can define the \emph{skew-partition} $\lambda/\mu$ as the Young diagram one gets by removing the cells of $\mu$ from the $\lambda$. The \emph{size} of a skew partition is denoted $|\lambda/\mu|$ and is given by $|\lambda/\mu| = |\lambda|-|\mu|$.

\subsubsection{Maya Diagram}
We can also represent partitions as certain sequences of particles and holes called \emph{Maya diagrams}. To understand Maya diagrams, it is helpful to first represent partitions in \emph{Russian convention}. Russian convention is a different perspective on Young diagrams, where we rotate the Young diagram 45 degrees to counter-clockwise.
\begin{example} We can represent the partition $8 = 4 + 3 + 1$ in the Russian convention as
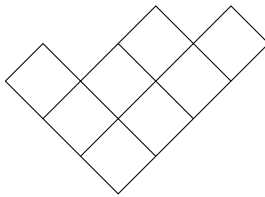
\[
\begin{tikzpicture}[baseline=(current bounding box.center)]
\begin{scope}[scale = 0.707, xscale=-1, rotate=135]
\draw (0,0)--(1,0)--(1,1)--(3,1)--(3,2)--(4,2)--(4,3)--(0,3)--(0,0);
\draw (0,1)--(1,1);
\draw (0,2)--(3,2);
\draw (1,1)--(1,3);
\draw (2,1)--(2,3);
\draw (3,2)--(3,3);
\end{scope}
\end{tikzpicture}
\]
which in French convention is \\
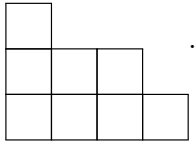
\[
\begin{ytableau}  
\empty \\
\empty & \empty & \empty \\
\empty &  \empty &  \empty & \empty \\
\end{ytableau}.
\]
\end{example}
\begin{definition}
The \emph{Maya diagrams} for for the partition $\lambda$ is obtained by viewing the Young diagram of $\lambda$ in Russian convention then, starting from the left, drawing a particle for every step diagonally down and a hole for every step diagonally up. Note that we extend the bounding lines of the Young diagram out infinitely far to the left and right.
\end{definition}
\begin{definition}
    There exists a unique point on the Maya diagram where the number of particles to the right of this point is equal to the number of holes to the left of this point. We call this point the \emph{center} of the Maya diagram.
\end{definition}

\begin{example} The Maya diagram of the partition $12 = 4 + 3 + 2 + 2 + 1$ is the string of particles and holes underneath the Young diagram below, with the center drawn in red:
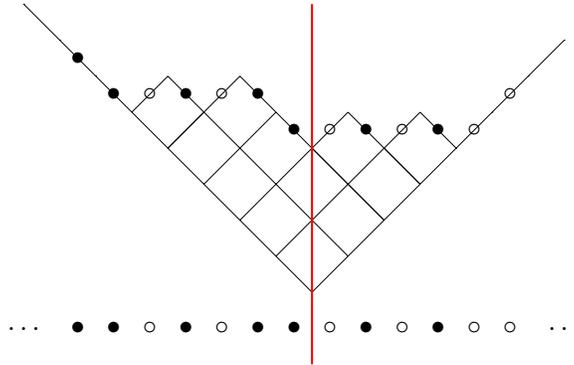
\[
\resizebox{8cm}{!}{
\begin{tikzpicture}[baseline=(current bounding box.center)]
\begin{scope}[rotate = 45, scale=1.414213]
\draw (0,0) grid (2,4); 
\draw (0,4) grid (1,5); 
\draw (2,0) grid (3,2); 
\draw (3,0) grid (4,1);
\node[scale=2] at (0.5,5) {$\circ$}; 
\node[scale=2] at (1,4.5) {$\bullet$}; \node[scale=2] at (1.5,4) {$\circ$}; 
\node[scale=2] at (2,3.5) {$\bullet$}; \node[scale=2] at (2,2.5) {$\bullet$}; \node[scale=2] at (2.5,2) {$\circ$}; 
\node[scale=2] at (3,1.5) {$\bullet$}; \node[scale=2] at (3.5,1) {$\circ$}; 
\node[scale=2] at (4,0.5) {$\bullet$}; 
\draw (4,0) grid (7,0); 
\draw (0,5) grid (0,8);
\node[scale=2] at (4.5,0) {$\circ$}; 
\node[scale=2] at (5.5,0) {$\circ$};  
\node[scale=2] at (0,5.5) {$\bullet$}; \node[scale=2] at (0,6.5) {$\bullet$}; 
\draw[ultra thick, red] (-1,-1)--(4,4);
\end{scope}
\node[scale=2] at (0.5,-1) {$\circ$};
\node[scale=2] at (1.5,-1) {$\bullet$}; 
\node[scale=2] at (2.5,-1) {$\circ$};
\node[scale=2] at (3.5,-1) {$\bullet$}; 
\node[scale=2] at (4.5,-1) {$\circ$};
\node[scale=2] at (5.5,-1) {$\circ$};
\node[scale=2] at (7,-1) {$\ldots$};
\node[scale=2] at (-0.5,-1) {$\bullet$};
\node[scale=2] at (-1.5,-1) {$\bullet$};
\node[scale=2] at (-2.5,-1) {$\circ$}; 
\node[scale=2] at (-3.5,-1) {$\bullet$};
\node[scale=2] at (-4.5,-1) {$\circ$}; 
\node[scale=2] at (-5.5,-1) {$\bullet$};
\node[scale=2] at (-6.5,-1) {$\bullet$};
\node[scale=2] at (-8,-1) {$\ldots$};
\end{tikzpicture} }
\]
\end{example}

\subsubsection{Hook lengths}
For a cell $c$ in row $i$ and column $j$ of the Young diagram of a partition $\lambda$. We define the \emph{arm} and \emph{leg} of $c$ to be
\begin{equation}
\begin{aligned}
\operatorname{arm}(c) &\;= \#\{ \text{cells to the right of $c$ in the same row} \} = \lambda_i-i \\
\operatorname{leg}(c) &\;= \#\{ \text{cells above $c$ in the same column} \} = \lambda'_j - j
\end{aligned}
\end{equation}
respectively. 

The \emph{hook length} of $c$, $h_{\lambda}(c)$ is given by
\begin{equation}
h_{\lambda}(c) =  \operatorname{arm}(c) + \operatorname{leg}(c) +1.
\end{equation}

\begin{example}
Consider the partition $\lambda = (4,3,1)$.
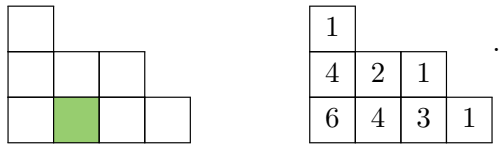
\[
\begin{ytableau}  
\empty \\
\empty & \empty & \empty \\
\empty &  *(YellowGreen) &  \empty & \empty \\
\end{ytableau}
\qquad \qquad
\begin{ytableau}  
1 \\
4 & 2 & 1 \\
6 &  4 &  3 & 1 \\
\end{ytableau}.
\]
The highlighted cell $c$ in the left diagram has $\operatorname{arm}(c)=2$, $\operatorname{leg}(c)=1$ and $h_{\lambda}(c) = 4$. On the right we fill each cell of $\lambda$ with the corresponding hook length.
\end{example}

\subsection{Plane Partitions}
Young diagrams can be thought of as a model for a one-dimensional partition as we can write them as a one-dimensional sequence $\lambda = (\lambda_1,\lambda_2,\ldots,\lambda_\ell)$. Expanding on this, one defines two-dimensional partitions known as \emph{plane partitions} (PP).
\begin{definition}
A \emph{plane partition} $\Pi$ of $n$ is a 2-D bi-infinite array of non-negative integers $(\Pi_{i,j})_{i,j=1}^\infty$ with sum $n$, such that rows are non-increasing from left to right and columns are non-increasing from bottom to top.
\end{definition}
We call the sum of the entries in the plane partition its \emph{volume}.
\begin{example}\label{PP1}
    Below are some plane partitions of volume $4$
    \begin{center}
    \begin{ytableau}
      \vdots & \vdots & \iddots \\
        0 & 0 & \ldots \\
        4 & 0 & \ldots 
    \end{ytableau} \hspace{25pt}
    \begin{ytableau}
       \vdots & \vdots & \ddots \\
        2 & 0 & \ldots \\
        2 & 0 & \ldots \\
        0 & 0 & \ldots 
    \end{ytableau} \hspace{25pt}
    \begin{ytableau}
    \vdots & \vdots & \iddots \\
        0 & 0 & \ldots \\
        1 & 0 & \ldots \\
        1 & 0 & \ldots \\
        2 & 0 & \ldots 
    \end{ytableau} \hspace{25pt}
    \begin{ytableau}
    \vdots & \vdots & \vdots & \iddots \\
        0 & 0 & 0 & \ldots \\
        1 & 0 & 0 & \ldots \\ 
        2 & 1 & 0 & \ldots
    \end{ytableau} \hspace{25pt}
    \begin{ytableau}
    \vdots & \vdots & \vdots & \vdots & \iddots \\
    	0 & 0 & 0 & 0 & \ldots \\
        2 & 1 & 1 & 0 & \ldots
    \end{ytableau} \hspace{25pt}
\end{center}
where one should extend the arrays infinitely far up and to the right with zeros.
\end{example}
 
An alternate way to represent the plane partition is as stacks of cubes.  To do so, on each entry of the array we stack the corresponding number of cubes, then view the array top-left. These stacks of cubes can also be thought of as lozenge tilings.
\begin{example}
Below are the respective lozenge tilings for the plane partitions in Example \ref{PP1}. We color code these lozenge tilings based on the orientation of the lozenges.\\
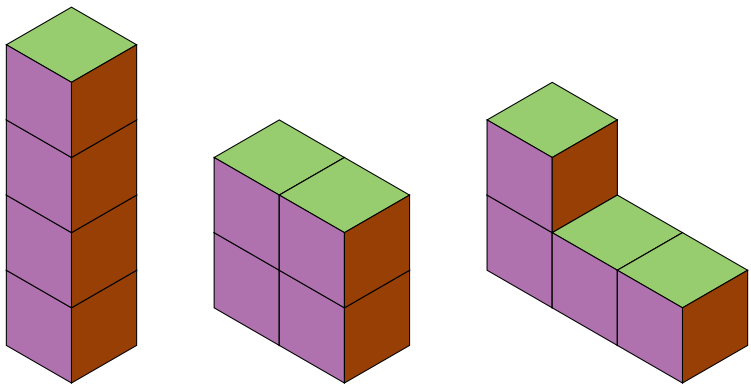
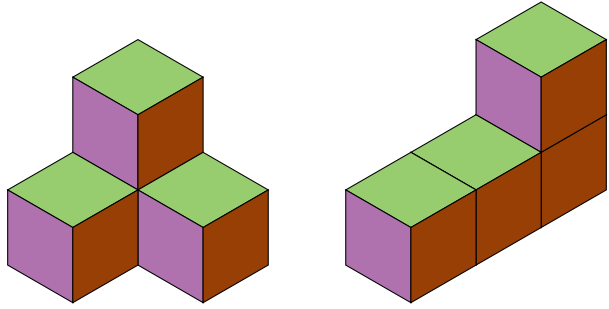
\begin{center}
    \begin{tikzpicture}
        \planepartition{{4}}
    \end{tikzpicture}\hspace{25pt}
    \begin{tikzpicture}
        \planepartition{{2,2}}
    \end{tikzpicture}\hspace{25pt}
    \begin{tikzpicture}
        \planepartition{{2,1,1}}
    \end{tikzpicture}\hspace{25pt}
    \begin{tikzpicture}
        \planepartition{{2,1}, {1}}
    \end{tikzpicture}\hspace{25pt}
    \begin{tikzpicture}
        \planepartition{{2},{1},{1}}
    \end{tikzpicture}
\end{center}
We omit drawing the stacks of height zero.
\end{example}

Rather than bi-infinite arrays, one can restrict the number of rows and columns of the PPs. A plane partitions $\Pi$ with $r$ rows and $c$ columns is an array of non-negative integers $(\Pi_{i,j})_{\substack{i=1,\ldots,r \\ j =1,\ldots, c}}$ such that rows are non-increasing from left to right and columns are non-increasing from bottom to top.

\subsubsection{Reverse Plane Partitions}
Our main object of interest is the reverse plane partition.
\begin{definition}
    Let $\lambda$ be a 1-dimensional partition. A \emph{reverse plane partition} (RPP) of shape $\lambda$ and \emph{volume} $n$ is an assignment of non-negative integers to the cells in the Young diagram of $\lambda$ such that the integers are non-decreasing from left to right in each row and non-decreasing from bottom to top in each column, and the sum of the integers is $n$.
\end{definition}
\begin{example}\label{rpp1}
    If $\lambda=(4,3,2)$, then one possible RPP of shape $\lambda$ and volume $18$ is given by
    \begin{center}
        \begin{ytableau}
            1 & 2\\
            1 & 2 & 4\\
            0 & 0 & 3 & 5
        \end{ytableau}.
    \end{center}
\end{example}
Note that there is a bijection between RPPs of rectangle shape and PPs with a restricted number of rows and columns, given by rotating the Young diagram by 180 degrees. For example
\begin{center}
    \begin{ytableau}
        3 & 3 & 6 & 6\\
        0 & 3 & 3 & 4\\
        0 & 0 & 2 & 2
    \end{ytableau} $\mapsto $
    \begin{ytableau}
        2 & 2 & 0 & 0\\
        4 & 3 & 3 & 0\\
        6 & 6 & 3 & 3
    \end{ytableau}.
\end{center}
Again, we have a bijection between RPPs and lozenge tilings. We can still consider the RPP as stacks of cubes as with PPs, but because the inequalities are switched, we must consider lozenge tilings as a view from the bottom left of the Young diagram instead of from the top right, in order to prevent stacks of cubes from obscuring the view of smaller stacks of cubes.
\begin{example}
    The RPP depicted in Example \ref{rpp1} can be represented as the lozenge tiling
    
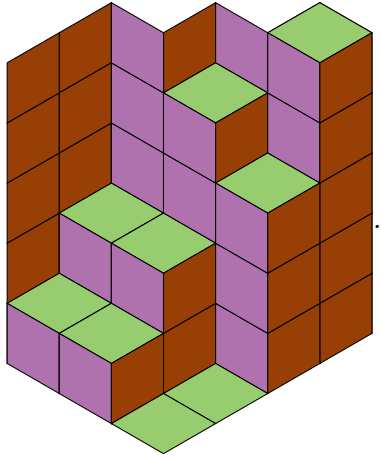
\begin{center}
        \begin{tikzpicture}[baseline=(current bounding box.center),scale=0.8]\reverseplanepartition{5}{{1,2},{1,2,4},{0,0,3,5}}\end{tikzpicture}.
    \end{center}
    We may think of these cubes as being in an infinitely tall room whose back wall is dettermined by the shape $\lambda$.
\end{example}
To fix some notation, denote the set of all plane partitions of shape $\lambda$ by $\operatorname{RPP}(\lambda)$ and for RPP $\Lambda$ denote is volume by $|\Lambda|$. If we fix the shape $\lambda$ the generating function of volume weighted RPP takes a particularly nice form.
\begin{theorem}[\cite{GANSNER198171}] \label{thm:genfunRPP}
If we fix $\lambda$  the the generating function of volume weighted RPP of shape $\lambda$ is given by
$$\sum_{\Lambda \in \operatorname{RPP}(\lambda)} q^{|\Lambda|} = \prod_{c \in \lambda} \frac{1}{1-q^{h_\lambda(c)}}$$
where the product is over all cells of the Young diagram of $\lambda$ and $h_{\lambda}(c)$ is the hook length of the cell $c$ in $\lambda$.
\end{theorem}

\subsection{Five Vertex Models} \label{sec:5V}
Vertex models encompass a large class of models originally arising in statistical physics, but are now known to have deep connection to representation theory and, as we shall see, combinatorics. In broad strokes, consider the vertices $(i, j) \in \mathbb{Z}^2$. Suppose we have certain local configurations we can place over each vertex and that there are restrictions on which local configurations are allowed at neighboring vertices. Furthermore, we assign weights to the different local configurations. The types of local configurations along with the local restrictions and weights define the vertex model. 

While the above description is rather vague, it is in part because there are such a wide variety of vertex models that is hard to give a general definition. We would like to stress that the idea that the building blocks of our vertex model only ``interact" locally in both the constraints and the weights. For our purposes, it will be enough to focus on two types of five-vertex models, where ``five" indicates the number of basic local configurations. These five-vertex models can be seen as degenerations of the well-studied six-vertex model, see \cite{reshetikhin2010lectures} and references therein. Moreover, there are the 1-color version of the multicolored vertex models introduced in \cite{aggarwal2023colored,LLT}. We will use the multi colored generalization in Section \ref{sec:multicolored}.

Our vertex models will consist of up-right paths with certain constraints on their allowed steps.  For the first of our vertex models, the five allowed local configurations, and their corresponding weight, is shown in Figure \ref{fig:whiteweight}. We draw the vertices with a white background (as opposed to the other five-vertex model which we will draw with a gray background) and call these the \emph{white vertices}.
\begin{figure}[h]
\[
\begin{tabular}{cccccc}
 \begin{tikzpicture}[baseline=(current bounding box.center)]\draw[thin] (0,0) rectangle (1,1); \node at (0.5,0.5) {$x$};\end{tikzpicture}: &   \begin{tikzpicture}[baseline=(current bounding box.center)]\draw[thin] (0,0) rectangle (1,1);\end{tikzpicture} & \begin{tikzpicture}[baseline=(current bounding box.center)]\draw[thin] (0,0) rectangle (1,1);\draw[black, very thick] (0.5,0)--(0.5,0.5)--(1,0.5);\end{tikzpicture} & \begin{tikzpicture}[baseline=(current bounding box.center)] \draw[thin] (0,0) rectangle (1,1); \draw[black, very thick] (0,0.5)--(1,0.5); \end{tikzpicture} & \begin{tikzpicture}[baseline=(current bounding box.center)] \draw[thin] (0,0) rectangle (1,1); \draw[black, very thick] (0.5,0)--(0.5,1); \end{tikzpicture} & \begin{tikzpicture}[baseline=(current bounding box.center)] \draw[thin] (0,0) rectangle (1,1); \draw[black, very thick] (0,0.5)--(0.5,0.5)--(0.5,1); \end{tikzpicture} \\
$\text{Weight}:$ & $1$ & $x$ & $x$ & $1$ & $1$ \\
\end{tabular}
\]
\caption{Allowed vertices and their weights for the white vertices.}
\label{fig:whiteweight}
\end{figure}
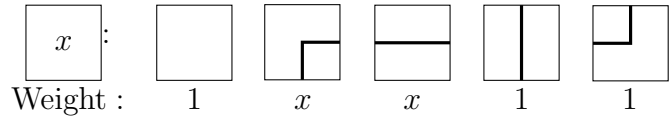\\

We will want to consider all possible configuration $\mathcal{C}$ of the vertex model on some domain with fixed boundary conditions (assignment of paths entering and exiting at each point on the boundary of the domain). The weight of a configuration is 
\[
w(\mathcal{C}) := \prod_{\text{vertices in }\mathcal{C}} (\text{weight of $v$}).
\]
The \emph{partition function} $Z$ of the vertex model is given by
\[
Z = \sum_{\text{all config. } \mathcal{C}} w(\mathcal{C}).
\]

To connect this to the theory of integer partitions, consider an infinite row of the vertex model with fixed boundary conditions on the top and bottom, and no paths entering or exiting from the sides. We may view the locations of where the paths enter the bottom boundary as the location of particles and the locations without paths entering as the locations of holes.
In this way, we may associate the Maya diagram of a partition $\mu$ to the bottom boundary condition. With the same mapping we get a Maya diagram of a partition $\lambda$ associated to the top boundary condition as well. Note that the allowed vertices enforce the constraint that each particle must be connected to the closest particle on top or to the right of itself. One can show that the row has a valid configuration of paths satisfying the boundary condition iff $\mu \preceq \lambda$. Moreover, if there is a valid configuration it is unique. See Example \ref{ex:row1}. 
\begin{example}\label{ex:row1}
Let $\mu=(1,1)$ and $\lambda=(3,1,1)$. Then the row with bottom boundary condition $\mu$ and to boundary condition $\lambda$ (with no paths entering or exiting on the sides) has a unique valid path configuration given by
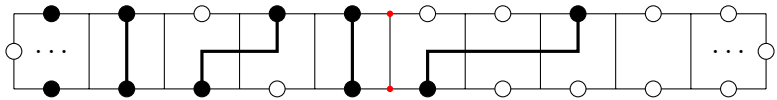
\[
\begin{tikzpicture}[baseline=(current bounding box.center)]
\draw (-1,0) grid (9,1);
\node at (-0.5,0.5) {$\ldots$};
\node at (8.5,0.5) {$\ldots$};

\draw[fill=black] (-0.5,0) circle (3pt);
\draw[fill=black] (0.5,0) circle (3pt);
\draw[fill=black] (1.5,0) circle (3pt);
\draw[fill=white] (2.5,0) circle (3pt);
\draw[fill=black] (3.5,0) circle (3pt);
\draw[fill=black] (4.5,0) circle (3pt);
\draw[fill=white] (5.5,0) circle (3pt);
\draw[fill=white] (6.5,0) circle (3pt);
\draw[fill=white] (7.5,0) circle (3pt);
\draw[fill=white] (8.5,0) circle (3pt);

\draw[fill=black] (-0.5,1) circle (3pt);
\draw[fill=black] (0.5,1) circle (3pt);
\draw[fill=white] (1.5,1) circle (3pt);
\draw[fill=black] (2.5,1) circle (3pt);
\draw[fill=black] (3.5,1) circle (3pt);
\draw[fill=white] (4.5,1) circle (3pt);
\draw[fill=white] (5.5,1) circle (3pt);
\draw[fill=black] (6.5,1) circle (3pt);
\draw[fill=white] (7.5,1) circle (3pt);
\draw[fill=white] (8.5,1) circle (3pt);

\draw[very thick] (0.5,0)--(0.5,1);
\draw[very thick] (1.5,0)--(1.5,0.5)--(2.5,0.5)--(2.5,1);
\draw[very thick] (3.5,0)--(3.5,1);
\draw[very thick] (4.5,0)--(4.5,0.5)--(6.5,0.5)--(6.5,1);

\draw[fill=white] (-1,0.5) circle (3pt);
\draw[fill=white] (9,0.5) circle (3pt);
\draw[red,fill=red] (4,0) circle (1pt);
\draw[red,fill=red] (4,1) circle (1pt);
\end{tikzpicture}
\]
where the red dots, indicate the center of the Maya diagrams corresponding to the top and bottom boundary conditions. The total weight of the row is $x^3$.
\end{example}

\begin{remark}\label{rmk:biinf}
In terms of the partitions, allowing the row to be infinitely long to the right means we are placing no restriction on the size of the parts of the partitions (other than that placed by the interlacing requirement). Allowing the row to be infinitely long to the left means we are placing no restriction on the length of the partition. In the what follows, we will restrict the number of columns to the left of the center of the Maya diagrams, corresponding to restricting the allowed length of the partitions.
\end{remark}

Our second vertex model, has its five allowed vertices and their associated weight given in Figure \ref{fig:grayweight}. Note that to distinguish them from the previous vertex model we shade the gray and refer to them as \emph{gray vertices}.  
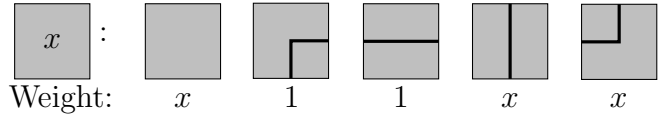
\begin{figure}[h]
\[
\begin{tabular}{cccccc}
  \begin{tikzpicture}[baseline=(current bounding box.center)]\draw[thin,fill=lightgray] (0,0) rectangle (1,1); \node at (0.5,0.5) {$x$}; \end{tikzpicture}\;: & \begin{tikzpicture}[baseline=(current bounding box.center)]\draw[thin,fill=lightgray] (0,0) rectangle (1,1);\end{tikzpicture} & \begin{tikzpicture}[baseline=(current bounding box.center)]\draw[thin,fill=lightgray] (0,0) rectangle (1,1);\draw[very thick] (0.5,0)--(0.5,0.5)--(1,0.5);\end{tikzpicture} & \begin{tikzpicture}[baseline=(current bounding box.center)] \draw[thin,fill=lightgray] (0,0) rectangle (1,1); \draw[very thick] (0,0.5)--(1,0.5); \end{tikzpicture} & \begin{tikzpicture}[baseline=(current bounding box.center)] \draw[thin,fill=lightgray] (0,0) rectangle (1,1); \draw[very thick] (0.5,0)--(0.5,1); \end{tikzpicture} & \begin{tikzpicture}[baseline=(current bounding box.center)] \draw[thin,fill=lightgray] (0,0) rectangle (1,1); \draw[very thick] (0,0.5)--(0.5,0.5)--(0.5,1); \end{tikzpicture} \\
\text{Weight:} &  $x$ & $1$ & $1$ & $x$ & $x$ \\
\end{tabular}
\]
\caption{The allowed vertices and their weights for the gray vertices.}
\label{fig:grayweight}
\end{figure}

These weights can be seen as the complement of the previous vertex model in which vertices previously having weight $1$ are now of weight $x$, and those of weight $x$ are now weight $1$.
Let $v$ be one of the five allowed vertices in either model, and define
\begin{equation} \label{eq:Lmat}
\begin{aligned}
L_x(v) =&\; \text{ the weight of the vertex $v$ in the white vertex model with parameter $x$} \\
L_x'(v) =&\; \text{ the weight of the vertex $v$ in the gray vertex model with parameter $x$}.
\end{aligned}
\end{equation}
We have  $L'_x(v) = x \; L_{1/x}(v)$.

Again we will consider a single row of the vertex model where, as before, we associate a partition $\mu$ through its Maya diagram to the bottom boundary conditions and partition $\lambda$ to the top boundary conditions. Unlike the white vertices, however, a row of the gray vertices will have the following properties:
\begin{itemize}
\item We center the Maya diagram on top one column to the left of the one on the bottom.
\item The row is no longer infinite to the left. Instead, it will have some fixed number of columns $\ell$ to the left of the center of the Maya diagram corresponding to the top boundary condition. This corresponds to restricting $\lambda$ to have length at most $\ell$. The row will still be infinite to the right.
\item As before, no path will enter from the left of a row, but we will force that a path does exit on the right.
\end{itemize}
See Example \ref{ex:row2} for an example. In this case, we have that there is a valid configuration iff $\lambda\preceq\mu$ and this configuration is unique.

\begin{example}\label{ex:row2}
Let $\mu=(3,1,1)$, $\lambda=(2,1)$, and $\ell=4$. Then the row with bottom boundary condition $\mu$ and top boundary condition $\lambda$ has a unique valid path configuration given by
\[
\begin{tikzpicture}[baseline=(current bounding box.center)]
\draw[fill=lightgray] (-1,0) rectangle (9,1);
\draw (-1,0) grid (9,1);

\node at (8.5,0.5) {$\ldots$};

\draw[fill=black] (-0.5,1) circle (3pt);
\draw[fill=black] (0.5,1) circle (3pt);
\draw[fill=white] (1.5,1) circle (3pt);
\draw[fill=black] (2.5,1) circle (3pt);
\draw[fill=white] (3.5,1) circle (3pt);
\draw[fill=black] (4.5,1) circle (3pt);
\draw[fill=white] (5.5,1) circle (3pt);
\draw[fill=white] (6.5,1) circle (3pt);
\draw[fill=white] (7.5,1) circle (3pt);
\draw[fill=white] (8.5,1) circle (3pt);

\draw[fill=black] (-0.5,0) circle (3pt);
\draw[fill=black] (0.5,0) circle (3pt);
\draw[fill=white] (1.5,0) circle (3pt);
\draw[fill=black] (2.5,0) circle (3pt);
\draw[fill=black] (3.5,0) circle (3pt);
\draw[fill=white] (4.5,0) circle (3pt);
\draw[fill=white] (5.5,0) circle (3pt);
\draw[fill=black] (6.5,0) circle (3pt);
\draw[fill=white] (7.5,0) circle (3pt);
\draw[fill=white] (8.5,0) circle (3pt);

\draw[very thick] (-0.5,0)--(-0.5,1);
\draw[very thick] (0.5,0)--(0.5,1);
\draw[very thick] (2.5,0)--(2.5,1);
\draw[very thick] (3.5,0)--(3.5,0.5)--(4.5,0.5)--(4.5,1);
\draw[very thick] (6.5,0)--(6.5,0.5)--(8.2,0.5);
\draw[very thick] (8.75,0.5)--(9,0.5);

\draw[fill=white] (-1,0.5) circle (3pt);
\draw[fill=black] (9,0.5) circle (3pt);
\draw[red,fill=red] (4,0) circle (1pt);
\draw[red,fill=red] (3,1) circle (1pt);
\end{tikzpicture}
\]
The weight of this row is $x^6$.
\end{example}

We summarize the above in the following lemma.
\begin{lemma}\label{lem:rowweights}
For a white row with bottom boundary given by $\mu$, top boundary given by $\lambda \succeq \mu$, and no paths entering or exiting on the sides, the weight of the row is given by
\[
x^{|\lambda|-|\mu|}.
\]
For a gray row with bottom boundary given by $\mu$, top boundary given by $\lambda \preceq \mu$, no paths entering on the left and a path entering on the right, the weight of the row is given by
\[
x^{|\mu|-|\lambda| + \ell}
\]
where $\ell$ is the number of columns to the left of the center of the top Maya diagram or, equivalently, the number of paths exiting the row at the top.
\end{lemma}

\subsubsection{Yang Baxter Equation}
The advantage of using the vertex model framework is that our vertex models are \emph{Yang-Baxter integrable}, that is, the satisfy the Yang-Baxter equation (YBE). To describe the YBE we will also need a new type of vertex we call a \emph{cross vertex}. The allowed local configurations and their weights are given in Figure \ref{fig:YBEWeights}.

\begin{figure}[h!]
\[
  \begin{tabular}{cccccc}
     \begin{tikzpicture}[baseline=(current bounding box.center)] \draw[thin] (0,1)--(1,0);\draw[thin] (0,0)--(1,1); \node[right] at (0.5,0.5) {$z$}; \end{tikzpicture}: &    
     \begin{tikzpicture}[baseline=(current bounding box.center)] \draw[ultra thick] (0,1)--(1,0);\draw[thin] (0,0)--(1,1); \end{tikzpicture} & 
     \begin{tikzpicture}[baseline=(current bounding box.center)] \draw[ultra thick] (0,1)--(0.5,0.5)--(1,1);\draw[thin] (0,0)--(0.5,0.5)--(1,0); \end{tikzpicture} &
      \begin{tikzpicture}[baseline=(current bounding box.center)] \draw[thin] (0,1)--(0.5,0.5)--(1,1);\draw[ultra thick] (0,0)--(0.5,0.5)--(1,0); \end{tikzpicture} & 
      \begin{tikzpicture}[baseline=(current bounding box.center)] \draw[ultra thick] (0,1)--(1,0);\draw[ultra thick] (0,0)--(1,1); \end{tikzpicture} & 
      \begin{tikzpicture}[baseline=(current bounding box.center)] \draw[thin] (0,1)--(1,0);\draw[thin] (0,0)--(1,1); \end{tikzpicture}  \\
         \text{Weight:} & $1-z$ & $z$ & $1$ & $z$ & $1$
    \end{tabular}
\]
\caption{Cross Weights}
\label{fig:YBEWeights}
\end{figure}
We can now state the YBE.
\begin{theorem}
For any choice of  $i_1,i_2,i_3,j_1,j_2,j_3\in\{0,1\}$ with $1$ indicating a path present at that boundary and $0$ indicating no path. We have the following equality of of partition functions
\[
\begin{tikzpicture}[baseline=(current bounding box.center)] 
 \draw (-1,0.5) -- (0,1.5); \draw (-1,1.5) -- (0,0.5); 
 \node[right] at (-0.45,1) {$\frac{y}{x}$};
 \draw[step=1.0,black,thin] (0,0) grid (1,2); 
 \node[left] at (-1,1.5) {$i_1$}; \node[left] at (-1,0.5) {$i_2$}; \node[below] at (0.5,0) {$i_3$};
 \node[right] at (1,1.5) {$j_2$}; \node[right] at (1,0.5) {$j_1$}; \node[above] at (0.5,2) {$j_3$};
 \node at (0.5,0.5) {$x$}; \node at (0.5,1.5) {$y$};
 \end{tikzpicture} 
 =
 \begin{tikzpicture}[baseline=(current bounding box.center)] 
 \draw (2,0.5) -- (1,1.5); \draw (2,1.5) -- (1,0.5); 
  \node[right] at (1.55,1) {$\frac{y}{x}$};
 \draw[step=1.0,black,thin] (0,0) grid (1,2); 
 \node[left] at (0,1.5) {$i_1$}; \node[left] at (0,0.5) {$i_2$}; \node[below] at (0.5,0) {$i_3$};
 \node[right] at (2,0.5) {$j_1$}; \node[right] at (2,1.5) {$j_2$}; \node[above] at (0.5,2) {$j_3$};
 \node at (0.5,0.5) {$y$}; \node at (0.5,1.5) {$x$};
 \end{tikzpicture}.
 \] 
In other words, for any given boundary conditions $i_1,i_2,i_3,j_1,j_2,j_3$, the sum of the weights of all path configurations satisfying the boundary conditions on the LHS above is equal to the sum of weights of all path configurations satisfying the same boundary conditions on the RHS.  
\end{theorem}
 \begin{proof}
 This can be shown by exhaustively checking each case. 
 \end{proof}
 
 \begin{example}
Suppose we set $i_1=j_2=1$ and the rest to zero. Then the YBE says we have the following equality of partition functions:
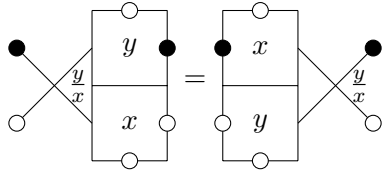
\[
 \begin{tikzpicture}[baseline=(current bounding box.center)] 
 \draw (-1,0.5) -- (0,1.5); \draw (-1,1.5) -- (0,0.5); 
 \node[right] at (-0.45,1) {$\frac{y}{x}$};
 \draw[step=1.0,black,thin] (0,0) grid (1,2); 
\draw[fill=black] (-1,1.5) circle (3pt); 
\draw[fill=white]  (-1,0.5) circle (3pt); 
\draw[fill=white]  (0.5,0) circle (3pt);
\draw[fill=black]  (1,1.5) circle (3pt);
\draw[fill=white]  (1,0.5) circle (3pt);
\draw[fill=white] (0.5,2) circle (3pt);
 \node at (0.5,0.5) {$x$}; \node at (0.5,1.5) {$y$};
 \end{tikzpicture} 
=
 \begin{tikzpicture}[baseline=(current bounding box.center)] 
 \draw (2,0.5) -- (1,1.5); \draw (2,1.5) -- (1,0.5); 
  \node[right] at (1.55,1) {$\frac{y}{x}$};
 \draw[step=1.0,black,thin] (0,0) grid (1,2); 
\draw[fill=black]  (0,1.5) circle (3pt);
\draw[fill=white]  (0,0.5) circle (3pt); 
\draw[fill=white] (0.5,0) circle (3pt);
 \draw[fill=white] (2,0.5) circle (3pt); 
 \draw[fill=black] (2,1.5) circle (3pt); 
 \draw[fill=white] (0.5,2) circle (3pt);
 \node at (0.5,0.5) {$y$}; \node at (0.5,1.5) {$x$};
 \end{tikzpicture}
 \]
 Explicitly writing the sum over path configurations, this becomes
 \[
 w\left(
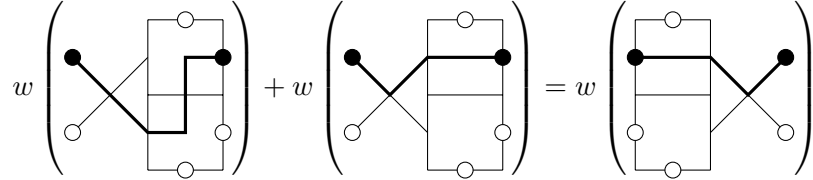
 \begin{tikzpicture}[baseline=(current bounding box.center)] 
 \draw (-1,0.5) -- (0,1.5); \draw (-1,1.5) -- (0,0.5); 
 \draw[step=1.0,black,thin] (0,0) grid (1,2); 
 
 \draw[very thick] (-1,1.5)--(0,0.5)--(0.5,0.5)--(0.5,1.5)--(1,1.5);
 
\draw[fill=black] (-1,1.5) circle (3pt); 
\draw[fill=white]  (-1,0.5) circle (3pt); 
\draw[fill=white]  (0.5,0) circle (3pt);
\draw[fill=black]  (1,1.5) circle (3pt);
\draw[fill=white]  (1,0.5) circle (3pt);
\draw[fill=white] (0.5,2) circle (3pt);
 \end{tikzpicture} 
 \right)
 +
  w\left(
 \begin{tikzpicture}[baseline=(current bounding box.center)] 
 \draw (-1,0.5) -- (0,1.5); \draw (-1,1.5) -- (0,0.5); 
 \draw[step=1.0,black,thin] (0,0) grid (1,2); 
 
  \draw[very thick] (-1,1.5)--(-0.5,1)--(0,1.5)--(1,1.5);
 
\draw[fill=black] (-1,1.5) circle (3pt); 
\draw[fill=white]  (-1,0.5) circle (3pt); 
\draw[fill=white]  (0.5,0) circle (3pt);
\draw[fill=black]  (1,1.5) circle (3pt);
\draw[fill=white]  (1,0.5) circle (3pt);
\draw[fill=white] (0.5,2) circle (3pt);
 \end{tikzpicture} 
 \right)
=
w\left(
 \begin{tikzpicture}[baseline=(current bounding box.center)] 
 \draw (2,0.5) -- (1,1.5); \draw (2,1.5) -- (1,0.5); 
 \draw[step=1.0,black,thin] (0,0) grid (1,2); 
 
 \draw[very thick] (0,1.5)--(1,1.5)--(1.5,1)--(2,1.5);
 
\draw[fill=black]  (0,1.5) circle (3pt);
\draw[fill=white]  (0,0.5) circle (3pt); 
\draw[fill=white] (0.5,0) circle (3pt);
 \draw[fill=white] (2,0.5) circle (3pt); 
 \draw[fill=black] (2,1.5) circle (3pt); 
 \draw[fill=white] (0.5,2) circle (3pt);
 \end{tikzpicture}
 \right)
 \]
 where we see there are two ways to fill in the paths on the left but only one on the right. Computing the weights we have
 \[
 \begin{aligned}
&\text{LHS:} \qquad \left(1-\frac{y}{x}\right)\cdot y + \frac{y}{x}\cdot y = y\\
&\text{RHS:} \qquad x\cdot \frac{y}{x} = y
 \end{aligned}
 \]
 which are in fact equal.
\end{example}

As the gray vertices differ from the white vertices by a change of variables, there is also a YBE between white and gray vertices.

\begin{cor}\label{cor:YBE}
For any choice of  $i_1,i_2,i_3,j_1,j_2,j_3\in\{0,1\}$ with $1$ indicating a path present at that boundary and $0$ indicating no path. We have the following equality of of partition functions
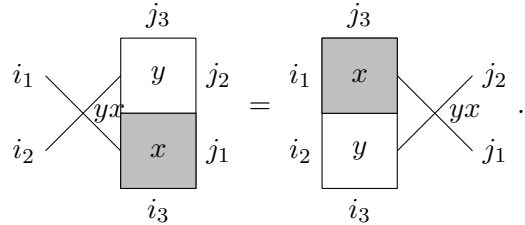
\[
\begin{tikzpicture}[baseline=(current bounding box.center)] 
 \draw (-1,0.5) -- (0,1.5); \draw (-1,1.5) -- (0,0.5); 
 \node[right] at (-0.5,1) {$yx$};
 \draw[step=1.0,black,thin] (0,0) grid (1,2); 
 \draw[fill=lightgray] (0,0) rectangle (1,1);
 \node[left] at (-1,1.5) {$i_1$}; \node[left] at (-1,0.5) {$i_2$}; \node[below] at (0.5,0) {$i_3$};
 \node[right] at (1,1.5) {$j_2$}; \node[right] at (1,0.5) {$j_1$}; \node[above] at (0.5,2) {$j_3$};
 \node at (0.5,0.5) {$x$}; \node at (0.5,1.5) {$y$};
 \end{tikzpicture} 
 =
 \begin{tikzpicture}[baseline=(current bounding box.center)] 
 \draw (2,0.5) -- (1,1.5); \draw (2,1.5) -- (1,0.5); 
  \node[right] at (1.55,1) {$yx$};
 \draw[step=1.0,black,thin] (0,0) grid (1,2); 
 \draw[fill=lightgray] (0,1) rectangle (1,2);
 \node[left] at (0,1.5) {$i_1$}; \node[left] at (0,0.5) {$i_2$}; \node[below] at (0.5,0) {$i_3$};
 \node[right] at (2,0.5) {$j_1$}; \node[right] at (2,1.5) {$j_2$}; \node[above] at (0.5,2) {$j_3$};
 \node at (0.5,0.5) {$y$}; \node at (0.5,1.5) {$x$};
 \end{tikzpicture}.
 \] 
 where the weights for the white vertices are those given in Figure \ref{fig:whiteweight} and the weights of the gray vertices are those given in Figure \ref{fig:grayweight}, and the weights for the cross vertices are given in Figure \ref{fig:YBEWeights} with $z=yx$. 
\end{cor}

\section{A Vertex Model for Reverse Plane Partition } \label{sec:vertexRPP}
By combining several rows of these vertex models, we can create a vertex model representation of the RPP. We will first see how this is constructed by looking at an example.
\begin{example}\label{ex:RPPvertex}
Consider the following RPP of shape $\lambda=(4,3,1)$
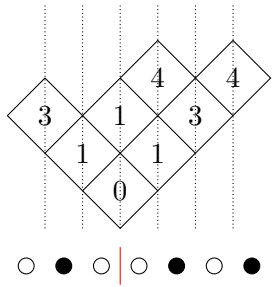
\[
\begin{tikzpicture}[baseline=(current bounding box.center)]
\begin{scope}[scale = 0.707, xscale=-1, rotate=135]
\draw (0,0)--(1,0)--(1,1)--(3,1)--(3,2)--(4,2)--(4,3)--(0,3)--(0,0);
\draw (0,1)--(1,1);
\draw (0,2)--(3,2);
\draw (1,1)--(1,3);
\draw (2,1)--(2,3);
\draw (3,2)--(3,3);
\end{scope}
\node at (2,0.5) {$4$}; \node at (3,0.5) {$4$};
\node at (0.5,0) {$3$}; \node at (1.5,0) {$1$}; \node at (2.5,0) {$3$};
\node at (1,-0.5) {$1$}; \node at (2,-0.5) {$1$};
\node at (1.5,-1) {$0$};
\draw[densely dotted, black] (0.5,-1.5)--(0.5,1.5);
\draw[densely dotted, black] (1,-1.5)--(1,1.5);
\draw[densely dotted, black] (1.5,-1.5)--(1.5,1.5);
\draw[densely dotted, black] (2,-1.5)--(2,1.5);
\draw[densely dotted, black] (2.5,-1.5)--(2.5,1.5);
\draw[densely dotted, black] (3,-1.5)--(3,1.5);
\draw[red] (1.5,-1.75)--(1.5,-2.25);

\draw[] (0.25,-2) circle (3pt);
\draw[fill=black] (0.75,-2) circle (3pt);
\draw[] (1.25,-2) circle (3pt);
\draw[] (1.75,-2) circle (3pt);
\draw[fill=black] (2.25,-2) circle (3pt);
\draw[] (2.75,-2) circle (3pt);
\draw[fill=black] (3.25,-2) circle (3pt);
\end{tikzpicture}
\]
where below we draw the Maya diagram for $\lambda$.

By reading the entries on each dotted line from top to bottom, we get a sequence of interlacing partitions representing our RPP. In particular, we have
\[
\emptyset  \preceq (3) \succeq (1) \preceq (1,0)  \preceq (4,1) \succeq (3)  \preceq (4) \succeq  \emptyset. 
\]
Note that we see a $\preceq$ whenever the Maya diagram has a hole and a $\succeq$ whenever the Maya diagram has a particle.

We can now map this to a path configuration using our vertex model. Starting from the bottom, whenever we see a $\preceq$ we will have a row of white vertices, while whenever we see a $\succeq$ we put a gray row. The specific partitions tells at what positions the paths cross between the rows. In our example, we have
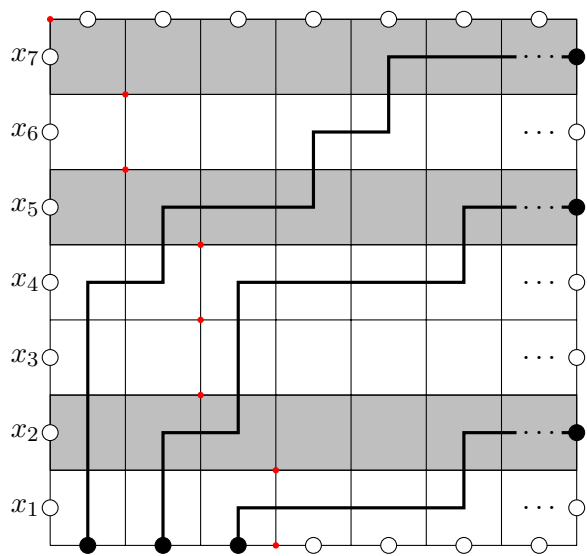
\[
\begin{tikzpicture}[baseline=(current bounding box.center)]
\draw[fill=lightgray] (1,1) rectangle (8,2);
\draw[fill=lightgray] (1,4) rectangle (8,5);
\draw[fill=lightgray] (1,6) rectangle (8,7);
\draw (1,0) grid (8,7);

\draw[fill=white] (1,0.5) circle (3pt);
\draw[fill=white] (1,1.5) circle (3pt);
\draw[fill=white] (1,2.5) circle (3pt);
\draw[fill=white] (1,3.5) circle (3pt);
\draw[fill=white] (1,4.5) circle (3pt);
\draw[fill=white] (1,5.5) circle (3pt);
\draw[fill=white] (1,6.5) circle (3pt);

\draw[fill=white] (8,0.5) circle (3pt);
\draw[fill=black] (8,1.5) circle (3pt);
\draw[fill=white] (8,2.5) circle (3pt);
\draw[fill=white] (8,3.5) circle (3pt);
\draw[fill=black] (8,4.5) circle (3pt);
\draw[fill=white] (8,5.5) circle (3pt);
\draw[fill=black] (8,6.5) circle (3pt);

\draw[fill=black] (1.5,0) circle (3pt);
\draw[fill=black] (2.5,0) circle (3pt);
\draw[fill=black] (3.5,0) circle (3pt);
\draw[fill=white] (4.5,0) circle (3pt);
\draw[fill=white] (5.5,0) circle (3pt);
\draw[fill=white] (6.5,0) circle (3pt);
\draw[fill=white] (7.5,0) circle (3pt);

\draw[fill=white] (1.5,7) circle (3pt);
\draw[fill=white] (2.5,7) circle (3pt);
\draw[fill=white] (3.5,7) circle (3pt);
\draw[fill=white] (4.5,7) circle (3pt);
\draw[fill=white] (5.5,7) circle (3pt);
\draw[fill=white] (6.5,7) circle (3pt);
\draw[fill=white] (7.5,7) circle (3pt);

\draw[red,fill=red] (4,0) circle (1pt);
\draw[red,fill=red] (4,1) circle (1pt);
\draw[red,fill=red] (3,2) circle (1pt);
\draw[red,fill=red] (3,3) circle (1pt);
\draw[red,fill=red] (3,4) circle (1pt);
\draw[red,fill=red] (2,5) circle (1pt);
\draw[red,fill=red] (2,6) circle (1pt);
\draw[red,fill=red] (1,7) circle (1pt);

\draw[very thick] (3.5,0)--(3.5,0.5)--(6.5,0.5)--(6.5,1.5)--(7.2,1.5); \node at (7.5,1.5) {$\ldots$}; \draw[very thick] (7.8,1.5)--(8,1.5);
\draw[very thick] (2.5,0)--(2.5,1.5)--(3.5,1.5)--(3.5,3.5)--(6.5,3.5)--(6.5,4.5)--(7.2,4.5); \node at (7.5,4.5) {$\ldots$}; \draw[very thick] (7.8,4.5)--(8,4.5);
\draw[very thick] (1.5,0)--(1.5,3.5)--(2.5,3.5)--(2.5,4.5)--(4.5,4.5)--(4.5,5.5)--(5.5,5.5)--(5.5,6.5)--(7.2,6.5); \node at (7.5,6.5) {$\ldots$}; \draw[very thick] (7.8,6.5)--(8,6.5);
\node at (7.5,0.5) {$\ldots$};
\node at (7.5,2.5) {$\ldots$};
\node at (7.5,3.5) {$\ldots$};
\node at (7.5,5.5) {$\ldots$};

\node[left] at (1,0.5) {$x_1$};
\node[left] at (1,1.5) {$x_2$};
\node[left] at (1,2.5) {$x_3$};
\node[left] at (1,3.5) {$x_4$};
\node[left] at (1,4.5) {$x_5$};
\node[left] at (1,5.5) {$x_6$};
\node[left] at (1,6.5) {$x_7$};
\end{tikzpicture}
\]
where row $i$ is given weight parameter $x_i$ which we display on the left of the diagram, and the paths configuration drawn is the unique path configuration satisfy the conditions on where the paths enter and exit each row.

The weights of the diagram is the product of the weight of each vertex which are determined from Figure \ref{fig:whiteweight} and \ref{fig:grayweight}. The total weight is $x_1^3x_2^4x_4^3x_5^3x_6x_7^3$.

Note the number of paths is number of non-zero parts of the shape $\lambda$ of the RPP. We could include some number of zero parts in $\lambda$ which would result in extending our domain with gray row on the bottom each with a new path exiting to the right. This would change the weight by an overall factor independent of the specific path configuration.
\end{example}

To make the sequence of Maya diagram explicit in the lozenge tilings we can assign particles and holes to the lozenges according to the rules
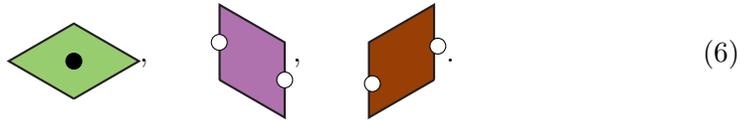
\begin{equation} \label{eq:particlestolozenges}
\begin{tikzpicture}[baseline=(current bounding box.center)]
\fill[YellowGreen,draw=Black,thick] (0,0)--(150:1)--(0,1)--(30:1)--(0,0);
\draw[fill=black] (0,0.5) circle (3pt);
\end{tikzpicture}
, \quad \quad
\begin{tikzpicture}[baseline=(current bounding box.center)]
\fill[Orchid,draw=Black,thick,shift={(150:1)}] (0,0)--(0,1)--(30:1)--(-30:1)--(0,0);
\draw[fill=white] (0,0.5) circle (3pt);
\draw[fill=white] (-0.87,1) circle (3pt);
\end{tikzpicture}
, \quad \quad
\begin{tikzpicture}[baseline=(current bounding box.center)]
\fill[RawSienna,draw=Black,thick,shift={(-0.05,0.05)}] (0,0)--(0,1)--(150:1)--(-150:1)--(0,0);
\draw[fill=white] (0,0.5) circle (3pt);
\draw[fill=white] (-0.87,0) circle (3pt);
\end{tikzpicture}.
\end{equation}

Example \ref{ex:particleholelozenge} shows how this assignment relates to the vertex model. 

\begin{example} \label{ex:particleholelozenge}
Here we have the RPP of shape $\lambda = (4,3,1)$ given in Example \ref{ex:RPPvertex}, shown as both a vertex model and a lozenge tiling.
 
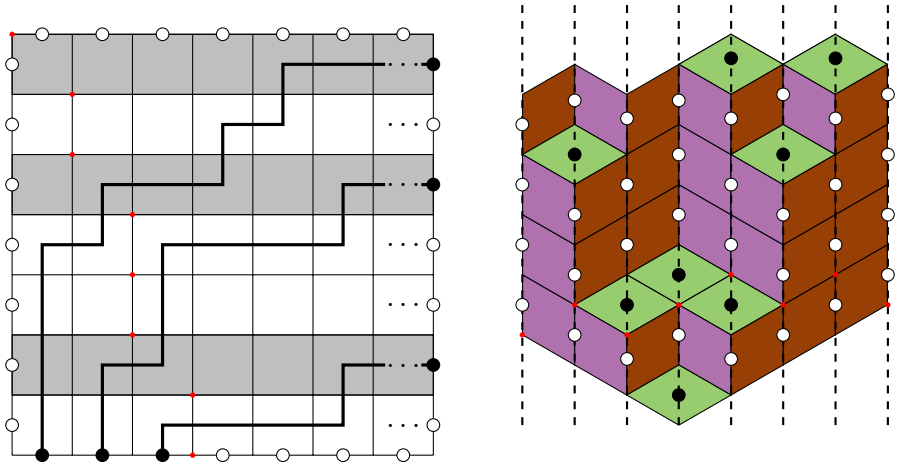
\[
\begin{tikzpicture}[baseline={(0,0)}, scale=0.8]
\draw[fill=lightgray] (1,1) rectangle (8,2);
\draw[fill=lightgray] (1,4) rectangle (8,5);
\draw[fill=lightgray] (1,6) rectangle (8,7);
\draw (1,0) grid (8,7);

\draw[fill=white] (1,0.5) circle (3pt);
\draw[fill=white] (1,1.5) circle (3pt);
\draw[fill=white] (1,2.5) circle (3pt);
\draw[fill=white] (1,3.5) circle (3pt);
\draw[fill=white] (1,4.5) circle (3pt);
\draw[fill=white] (1,5.5) circle (3pt);
\draw[fill=white] (1,6.5) circle (3pt);

\draw[fill=white] (8,0.5) circle (3pt);
\draw[fill=black] (8,1.5) circle (3pt);
\draw[fill=white] (8,2.5) circle (3pt);
\draw[fill=white] (8,3.5) circle (3pt);
\draw[fill=black] (8,4.5) circle (3pt);
\draw[fill=white] (8,5.5) circle (3pt);
\draw[fill=black] (8,6.5) circle (3pt);

\draw[fill=black] (1.5,0) circle (3pt);
\draw[fill=black] (2.5,0) circle (3pt);
\draw[fill=black] (3.5,0) circle (3pt);
\draw[fill=white] (4.5,0) circle (3pt);
\draw[fill=white] (5.5,0) circle (3pt);
\draw[fill=white] (6.5,0) circle (3pt);
\draw[fill=white] (7.5,0) circle (3pt);

\draw[fill=white] (1.5,7) circle (3pt);
\draw[fill=white] (2.5,7) circle (3pt);
\draw[fill=white] (3.5,7) circle (3pt);
\draw[fill=white] (4.5,7) circle (3pt);
\draw[fill=white] (5.5,7) circle (3pt);
\draw[fill=white] (6.5,7) circle (3pt);
\draw[fill=white] (7.5,7) circle (3pt);

\draw[red,fill=red] (4,0) circle (1pt);
\draw[red,fill=red] (4,1) circle (1pt);
\draw[red,fill=red] (3,2) circle (1pt);
\draw[red,fill=red] (3,3) circle (1pt);
\draw[red,fill=red] (3,4) circle (1pt);
\draw[red,fill=red] (2,5) circle (1pt);
\draw[red,fill=red] (2,6) circle (1pt);
\draw[red,fill=red] (1,7) circle (1pt);

\draw[very thick] (3.5,0)--(3.5,0.5)--(6.5,0.5)--(6.5,1.5)--(7.2,1.5); \node at (7.5,1.5) {$\ldots$}; \draw[very thick] (7.8,1.5)--(8,1.5);
\draw[very thick] (2.5,0)--(2.5,1.5)--(3.5,1.5)--(3.5,3.5)--(6.5,3.5)--(6.5,4.5)--(7.2,4.5); \node at (7.5,4.5) {$\ldots$}; \draw[very thick] (7.8,4.5)--(8,4.5);
\draw[very thick] (1.5,0)--(1.5,3.5)--(2.5,3.5)--(2.5,4.5)--(4.5,4.5)--(4.5,5.5)--(5.5,5.5)--(5.5,6.5)--(7.2,6.5); \node at (7.5,6.5) {$\ldots$}; \draw[very thick] (7.8,6.5)--(8,6.5);
\node at (7.5,0.5) {$\ldots$};
\node at (7.5,2.5) {$\ldots$};
\node at (7.5,3.5) {$\ldots$};
\node at (7.5,5.5) {$\ldots$};

\end{tikzpicture}
\hspace{1cm}
\begin{tikzpicture}[baseline = {(0,-2)}, scale=0.8]
  \reverseplanepartition{4}{{3},{1,1,4},{0,1,3,4}}
  \draw[black, thick, dashed] (-0.87,-2)--(-0.87,5);
   \draw[fill=white] (-0.87,0) circle (3pt);
  \draw[fill=white] (-0.87,1) circle (3pt);
  \draw[fill=white] (-0.87,2) circle (3pt);
  \draw[fill=white] (-0.87,3) circle (3pt);
  
  \draw[black, thick, dashed] (0,-2)--(0,5);
  \draw[fill=white] (0,-0.5) circle (3pt);
  \draw[fill=white] (0,0.5) circle (3pt);
  \draw[fill=white] (0,1.5) circle (3pt);
  \draw[fill=black] (0,2.5) circle (3pt);
  \draw[fill=white] (0,3.4) circle (3pt);
  
  \draw[black, thick, dashed] (0.87,-2)--(0.87,5);
  \draw[fill=white] (0.87,-0.9) circle (3pt);
  \draw[fill=black] (0.87,0) circle (3pt);
  \draw[fill=white] (0.87,1) circle (3pt);
  \draw[fill=white] (0.87,2) circle (3pt);
  \draw[fill=white] (0.87,3.1) circle (3pt);

  \draw[black, thick, dashed] (1.73,-2)--(1.73,5);
  \draw[fill=black] (1.73,-1.5) circle (3pt);
  \draw[fill=white] (1.73,-0.5) circle (3pt);
  \draw[fill=black] (1.73,0.5) circle (3pt);
  \draw[fill=white] (1.73,1.5) circle (3pt);
  \draw[fill=white] (1.73,2.5) circle (3pt);
  \draw[fill=white] (1.73,3.4) circle (3pt);

  \draw[black, thick, dashed] (2.6,-2)--(2.6,5);
  \draw[fill=white] (2.6,-0.9) circle (3pt);
  \draw[fill=black] (2.6,0) circle (3pt);
  \draw[fill=white] (2.6,1) circle (3pt);
  \draw[fill=white] (2.6,2) circle (3pt);
  \draw[fill=white] (2.6,3.1) circle (3pt);
  \draw[fill=black] (2.6,4.1) circle (3pt);

  \draw[black, thick, dashed] (3.465,-2)--(3.465,5);
  \draw[fill=white] (3.465,-0.5) circle (3pt);
  \draw[fill=white] (3.465,0.5) circle (3pt);
  \draw[fill=white] (3.465,1.5) circle (3pt);
  \draw[fill=black] (3.465,2.5) circle (3pt);
  \draw[fill=white] (3.465,3.5) circle (3pt);

  \draw[black, thick, dashed] (4.335,-2)--(4.335,5);
  \draw[fill=white] (4.335,0) circle (3pt);
  \draw[fill=white] (4.335,1) circle (3pt);
  \draw[fill=white] (4.335,2) circle (3pt);
  \draw[fill=white] (4.335,3.1) circle (3pt);
  \draw[fill=black] (4.335,4.1) circle (3pt);
  
  \draw[black, thick, dashed] (5.205,-2)--(5.205,5);
  \draw[fill=white] (5.205, 0.5) circle (3pt);
  \draw[fill=white] (5.205, 1.5) circle (3pt);
  \draw[fill=white] (5.205, 2.5) circle (3pt);
  \draw[fill=white] (5.205, 3.5) circle (3pt);
  
  \draw[red,fill=red] (-0.87,-0.5) circle (1pt);
  \draw[red,fill=red] (0,0) circle (1pt);
  \draw[red,fill=red] (0.87,-0.5) circle (1pt);
  \draw[red,fill=red] (1.73,0) circle (1pt);
  \draw[red,fill=red] (2.6,0.5) circle (1pt);
  \draw[red,fill=red] (3.465,0) circle (1pt);
  \draw[red,fill=red] (4.335,0.5) circle (1pt);
  \draw[red,fill=red] (5.205,0) circle (1pt);

\end{tikzpicture}
\]
The sequences of particles read from left-to-right in the lozenge tiling is the same as those in the vertex model (determined by where the paths cross between the rows) read from bottom-to-top. The small red circles in both figures are to mark the center of each Maya diagram. In the lozenge tiling picture, the centers trace the shape of the back wall. If every column had zero height then in each slice the particles would be packed below the centers and there would only be holes above the center.
 \end{example}

The mapping from RPP to vertex model configurations as in Example \ref{ex:RPPvertex} works in general. To begin, we show that there is a bijection from fillings of a RPP to certain sequences of interlacing partitions.
 \begin{theorem}\label{thm:RPPtoInterlace}
 Consider the set of RPP of shape $\lambda$. Let $n=\lambda_1 + \lambda_1'-1$. Suppose $I = (I_0,\ldots, I_n)\in \{\preceq,\succeq\}^n$ is chosen such that $I_k = \preceq$ if there is a hole at $-\lambda_1'+\frac{2k+1}{2}$ in the Maya diagram of $\lambda$ and $I_k = \succeq$ otherwise. Then for any $\Lambda \in \operatorname{RPP}(\lambda)$, the sequence of partitions $ \lambda^{(1)},\ldots, \lambda^{(n)}$ one gets from reading the filling of $\lambda$ (drawn in Russian convention) along vertical slices from left to right, satisfies the interlacing condition
 \[
 \emptyset = \lambda^{(0)} \;I_0\; \lambda^{(1)} \;I_1\; \lambda^{(2)} \;I_3\; \ldots \;I_n\; \lambda^{(n+1)} = \emptyset. 
 \]
 Moreover, this map from fillings to sequences of interlacing partitions is invertible.
 \end{theorem}
\begin{proof}
This bijection is given as an example in Section 2.2 of \cite{betea2014perfect} where the authors study more general sequences of interlacing partitions.
\end{proof}

Given a partition $\lambda$, define the domain for our vertex model
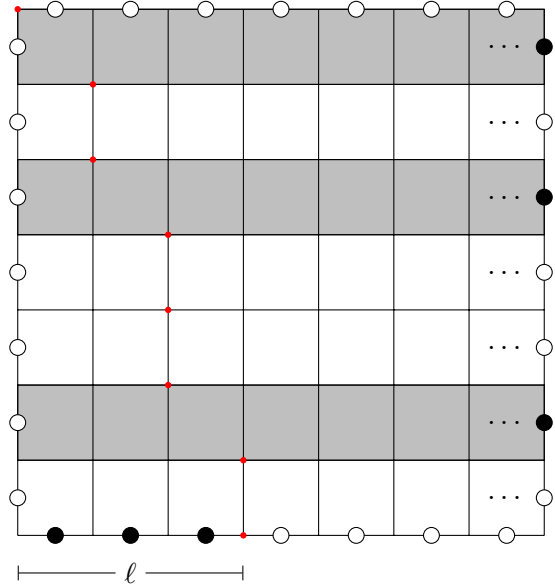
\[
\begin{tikzpicture}[baseline=(current bounding box.center)]
\draw[fill=lightgray] (1,1) rectangle (8,2);
\draw[fill=lightgray] (1,4) rectangle (8,5);
\draw[fill=lightgray] (1,6) rectangle (8,7);
\draw (1,0) grid (8,7);

\draw[fill=white] (1,0.5) circle (3pt);
\draw[fill=white] (1,1.5) circle (3pt);
\draw[fill=white] (1,2.5) circle (3pt);
\draw[fill=white] (1,3.5) circle (3pt);
\draw[fill=white] (1,4.5) circle (3pt);
\draw[fill=white] (1,5.5) circle (3pt);
\draw[fill=white] (1,6.5) circle (3pt);

\draw[fill=white] (8,0.5) circle (3pt);
\draw[fill=black] (8,1.5) circle (3pt);
\draw[fill=white] (8,2.5) circle (3pt);
\draw[fill=white] (8,3.5) circle (3pt);
\draw[fill=black] (8,4.5) circle (3pt);
\draw[fill=white] (8,5.5) circle (3pt);
\draw[fill=black] (8,6.5) circle (3pt);

\draw[fill=black] (1.5,0) circle (3pt);
\draw[fill=black] (2.5,0) circle (3pt);
\draw[fill=black] (3.5,0) circle (3pt);
\draw[fill=white] (4.5,0) circle (3pt);
\draw[fill=white] (5.5,0) circle (3pt);
\draw[fill=white] (6.5,0) circle (3pt);
\draw[fill=white] (7.5,0) circle (3pt);

\draw[fill=white] (1.5,7) circle (3pt);
\draw[fill=white] (2.5,7) circle (3pt);
\draw[fill=white] (3.5,7) circle (3pt);
\draw[fill=white] (4.5,7) circle (3pt);
\draw[fill=white] (5.5,7) circle (3pt);
\draw[fill=white] (6.5,7) circle (3pt);
\draw[fill=white] (7.5,7) circle (3pt);

\draw[red,fill=red] (4,0) circle (1pt);
\draw[red,fill=red] (4,1) circle (1pt);
\draw[red,fill=red] (3,2) circle (1pt);
\draw[red,fill=red] (3,3) circle (1pt);
\draw[red,fill=red] (3,4) circle (1pt);
\draw[red,fill=red] (2,5) circle (1pt);
\draw[red,fill=red] (2,6) circle (1pt);
\draw[red,fill=red] (1,7) circle (1pt);

\node at (7.5,0.5) {$\dots$};
\node at (7.5,1.5) {$\dots$};
\node at (7.5,2.5) {$\dots$};
\node at (7.5,3.5) {$\dots$};
\node at (7.5,4.5) {$\dots$};
\node at (7.5,5.5) {$\dots$};
\node at (7.5,6.5) {$\dots$};

\draw[|-] (1,-0.5)--(2.3,-0.5);
\draw[-|] (2.7,-0.5)--(4,-0.5);
\node at (2.5,-0.5) {$\ell$};
\end{tikzpicture}
\]
where $\ell=\ell(\lambda)$ and row $i$ is a white row if $I_{i-1}= \preceq$ (as defined in Theorem \ref{thm:RPPtoInterlace}) and is a gray row otherwise. We immediately have that there is a bijection between the sequence of interlacing partitions and path configuration in the vertex model since the sequence of partitions determines where paths enter and exit each row and there is a unique way to fill the rows iff the interlacing conditions are met.

Recall that the vertex models comes with weights, in which we row $i$ is given parameter $x_i$. We would like to choose these parameters so that the total weight of the path configuration is equal to $q$ to the volume of the reverse plane partition.

\begin{theorem}\label{thm:bij1color}
Let $\Lambda$ be a reverse plane partition and $\ell(\lambda)$ be the number of non-zero parts of $\lambda$. Let $\mathcal{C}$ be the corresponding path configuration under the bijection described above. Then if one chooses the weight parameters of the vertex model to be $x_i=q^{\pm i}$ where we take $+$ for a gray row and $-$ for a white row, we have
\[
q^{|\Lambda|} = w(\mathcal{C})\cdot A_\lambda(q) 
\]
where $A_\lambda(q)=q^{-\sum_{i=1}^{\ell(\lambda)} (\lambda_i+\ell(\lambda)-i+1)(i-1)}$ is completely determined by the shape $\lambda$ and is independent of the specific configuration. 
\end{theorem}
\begin{proof}
To begin, we have that the gray row corresponding to $\lambda_i$ will have parameter $q^{\lambda_i-i+\ell(\lambda)+1}$. To see this note that if the gray rows were fully packed at the bottom, the row corresponding to $\lambda_i$ would start as the $(\ell-i+1)$th row from the bottom and would need to swap position $\lambda_i$ times with adjacent white rows to get to its end location. Thus it would end as the $(\lambda_i+\ell-i+1)$th row.

As one path exits to the right at each gray row, there are $i-1$ paths exiting the top of the gray row corresponding to $\lambda_i$. If we choose the particular configuration $\mathcal{C}_0$ with each path going vertically upward until it must turn right to exit the domain, then the only non-trivial contribution to the weight is given by paths exiting at the top of a gray row. The total weight is then
\[
w(\mathcal{C}_0) = q^{\sum_{i=1}^{\ell(\lambda)} (\lambda_i-i+\ell(\lambda)+1)(i-1)} =\frac{1}{A_{\lambda}(q)}.
\]
For this configuration, the corresponding RPP has every entry equal to zero giving it zero volume and a weight of 1. We see that the theorem holds for the particular configuration $\mathcal{C}_0$.

Next, we note that the set of all possible path configurations is connected under corner flips of the form
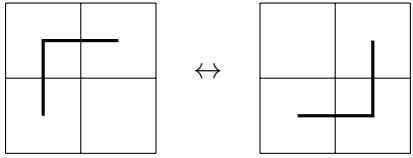
\[
\begin{tikzpicture}[baseline=(current bounding box.center)]
\draw (0,0) grid (2,2);
\draw[very thick] (0.5,0.5)--(0.5,1.5)--(1.5,1.5);
\end{tikzpicture}
\quad \leftrightarrow \quad
\begin{tikzpicture}[baseline=(current bounding box.center)]
\draw (0,0) grid (2,2);
\draw[very thick] (0.5,0.5)--(1.5,0.5)--(1.5,1.5);
\end{tikzpicture}
\]
where each row can be either white or gray. To finish the proof we will show that the weight change in the vertex model after a corner flip is equal to the corresponding weight change for the RPP. We will show this is true for down flips (left to right in the above). The up flips can be work out similarly. 

Recall from our bijection from the vertex model to the RPP that the places the path cross between rows correspond to Maya diagrams of partitions that in turn correspond to the fillings of the RPP read along vertical slices when drawn in the Russian convention. 

Under a down corner flip, the path now crosses between the rows one column further to the right. This corresponds to one part of the corresponding partition increasing by one, which in turn corresponds to one entry in the RPP increasing by 1. We see that the volume of the RPP increases by 1, thus its weight has an additional factor of $q$. 

Finally, we are left with checking that the change in weight in the vertex model agrees. Suppose the rows are rows $i$ and $i+1$. There are four cases to check.

\noindent \textbf{Case 1: top white row - bottom white row}

To see that this agrees with the weight change in the vertex model, note that before the flip there is a factor of $q^{-(i+1)}$ in the weight from the path exiting a vertex to the right in the top row, while after the flip there is a factor of $q^{-i}$ from the path exiting a vertex to the right in the bottom row. We see that the weight gained a factor of $q$ after the corner flip, in agreement with the RPP.

\noindent \textbf{Case 2: top white row - bottom gray row}

In this case, before the flip, there is a factor of $q^{-(i+1)}$ in the weight from the path exiting right in the top row and two factor of $q^i$ from the path exiting upward in the gray row and the empty vertex in the gray row. The net contribution to the weight is $q^{i-1}$. After the flip there is only a factor of $q^i$ from the path exiting upward in the gray row. We see that the weight gained a factor of $q$ after the corner flip.

\noindent \textbf{Case 3: top gray row - bottom white row}

In this case, before the flip, the corner contributes a weight of 1. After the flip there is only a factor of $q^{-i}$ from the path exiting a vertex to the right in the white row and a factor of $q^{i+1}$ from the empty vertex in the gray row. We see that the weight gained a factor of $q$ after the corner flip.

\noindent \textbf{Case 4: top gray row - bottom gray row}

In this case, before the flip, the corner contributes a two factors of $q^{i}$ from the path exiting upward and the empty vertex in the bottom row. After the flip, the corner contributes a factors of $q^{i}$ from the path exiting upward in the bottom row and a factor of $q^{i+1}$ from the empty vertex in the top row. We see that the weight gained a factor of $q$ after the corner flip.

\end{proof}

\subsection{The Generating Function}
We can use this bijection to give a proof of Theorem \ref{thm:genfunRPP}. In Section \ref{sec:overlappingRPPs}, we will use the same technique to compute the generating function for coupled RPPs. To begin we need the following lemma.

\begin{lemma} \label{lem:whitegrayswap}
Choose $x,y$ such that $|xy|<1$. Then the white and gray rows satisfy the commutation relation
\begin{equation}
\begin{aligned}
&\begin{tikzpicture}[baseline=(current bounding box.center)]
\draw[fill=lightgray] (0,0) rectangle (7,1);
\draw (0,0) grid (7,2);
\node at (6.5,0.5) {$\dots$};
\node at (6.5,1.5) {$\dots$};
\node at (0.5,0.5) {$x$};
\node at (1.5,0.5) {$x$};
\node at (2.5,0.5) {$x$};
\node at (3.5,0.5) {$x$};
\node at (4.5,0.5) {$x$};
\node at (5.5,0.5) {$x$};
\node at (0.5,1.5) {$y$};
\node at (1.5,1.5) {$y$};
\node at (2.5,1.5) {$y$};
\node at (3.5,1.5) {$y$};
\node at (4.5,1.5) {$y$};
\node at (5.5,1.5) {$y$};

\node[below] at (3.5,0) {$\mu$};
\node[above] at (3.5,2) {$\lambda$};

\draw[fill=white] (0,0.5) circle (3pt);
\draw[fill=white] (0,1.5) circle (3pt);
\draw[fill=black] (7,0.5) circle (3pt);
\draw[fill=white] (7,1.5) circle (3pt);
\end{tikzpicture}
\\
&\qquad \qquad = (1-xy) 
\begin{tikzpicture}[baseline=(current bounding box.center)]
\draw[fill=lightgray] (0,1) rectangle (7,2);
\draw (0,0) grid (7,2);
\node at (6.5,0.5) {$\dots$};
\node at (6.5,1.5) {$\dots$};
\node at (0.5,0.5) {$y$};
\node at (1.5,0.5) {$y$};
\node at (2.5,0.5) {$y$};
\node at (3.5,0.5) {$y$};
\node at (4.5,0.5) {$y$};
\node at (5.5,0.5) {$y$};
\node at (0.5,1.5) {$x$};
\node at (1.5,1.5) {$x$};
\node at (2.5,1.5) {$x$};
\node at (3.5,1.5) {$x$};
\node at (4.5,1.5) {$x$};
\node at (5.5,1.5) {$x$};

\node[below] at (3.5,0) {$\mu$};
\node[above] at (3.5,2) {$\lambda$};

\draw[fill=white] (0,0.5) circle (3pt);
\draw[fill=white] (0,1.5) circle (3pt);
\draw[fill=white] (7,0.5) circle (3pt);
\draw[fill=black] (7,1.5) circle (3pt);
\end{tikzpicture}
\end{aligned}
\end{equation}
where $\lambda$ and $\mu$ are partitions corresponding to the top and bottom boundary conditions of the pair of rows.
\end{lemma}
\begin{proof}
We start with the LHS and add a cross to the left boundary as shown below
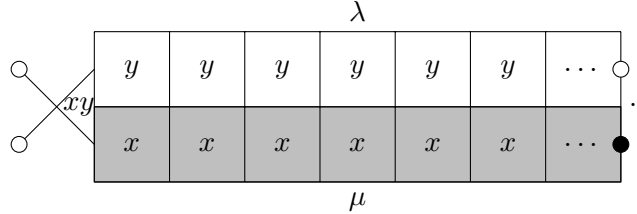
\[
\begin{tikzpicture}[baseline=(current bounding box.center)]
\draw[fill=lightgray] (0,0) rectangle (7,1);
\draw (0,0) grid (7,2);
\draw (-1,1.5)--(0,0.5);
\draw (-1,0.5)--(0,1.5);
\node at (6.5,0.5) {$\dots$};
\node at (6.5,1.5) {$\dots$};
\node at (0.5,0.5) {$x$};
\node at (1.5,0.5) {$x$};
\node at (2.5,0.5) {$x$};
\node at (3.5,0.5) {$x$};
\node at (4.5,0.5) {$x$};
\node at (5.5,0.5) {$x$};
\node at (0.5,1.5) {$y$};
\node at (1.5,1.5) {$y$};
\node at (2.5,1.5) {$y$};
\node at (3.5,1.5) {$y$};
\node at (4.5,1.5) {$y$};
\node at (5.5,1.5) {$y$};

\node[right] at (-0.55,1) {$xy$};

\node[below] at (3.5,0) {$\mu$};
\node[above] at (3.5,2) {$\lambda$};

\draw[fill=white] (-1,0.5) circle (3pt);
\draw[fill=white] (-1,1.5) circle (3pt);
\draw[fill=black] (7,0.5) circle (3pt);
\draw[fill=white] (7,1.5) circle (3pt);
\end{tikzpicture}.
\]
Note that the boundary condition forces the cross to be empty and that the empty cross has weight $1$, so the two row partition function remains unchanged. By repeatedly applying the YBE from Corollary \ref{cor:YBE} we may move the cross all the way to the right without changing the partition function. Doing this we have
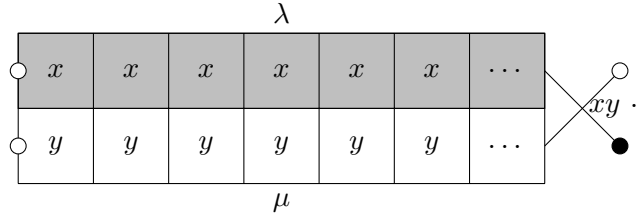
\[
\begin{tikzpicture}[baseline=(current bounding box.center)]
\draw[fill=lightgray] (0,1) rectangle (7,2);
\draw (0,0) grid (7,2);
\draw (7,1.5)--(8,0.5);
\draw (7,0.5)--(8,1.5);

\node at (6.5,0.5) {$\dots$};
\node at (6.5,1.5) {$\dots$};
\node at (0.5,0.5) {$y$};
\node at (1.5,0.5) {$y$};
\node at (2.5,0.5) {$y$};
\node at (3.5,0.5) {$y$};
\node at (4.5,0.5) {$y$};
\node at (5.5,0.5) {$y$};
\node at (0.5,1.5) {$x$};
\node at (1.5,1.5) {$x$};
\node at (2.5,1.5) {$x$};
\node at (3.5,1.5) {$x$};
\node at (4.5,1.5) {$x$};
\node at (5.5,1.5) {$x$};

\node[below] at (3.5,0) {$\mu$};
\node[above] at (3.5,2) {$\lambda$};
\node[right] at (7.45,1) {$xy$};

\draw[fill=white] (0,0.5) circle (3pt);
\draw[fill=white] (0,1.5) circle (3pt);
\draw[fill=black] (8,0.5) circle (3pt);
\draw[fill=white] (8,1.5) circle (3pt);
\end{tikzpicture}.
\]
There are two possibilities for the cross vertex
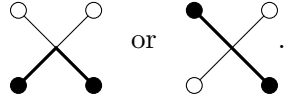
\[
\begin{tikzpicture}[baseline=(current bounding box.center)]
\draw (0,1.5)--(1,0.5);
\draw (0,0.5)--(1,1.5);
\draw[very thick] (0,0.5)--(0.5,1)--(1,0.5);
\draw[fill=black] (0,0.5) circle (3pt);
\draw[fill=white] (0,1.5) circle (3pt);
\draw[fill=black] (1,0.5) circle (3pt);
\draw[fill=white] (1,1.5) circle (3pt);
\end{tikzpicture}
\quad \text{or} \quad
\begin{tikzpicture}[baseline=(current bounding box.center)]
\draw (0,1.5)--(1,0.5);
\draw (0,0.5)--(1,1.5);
\draw[very thick] (0,1.5)--(1,0.5);
\draw[fill=white] (0,0.5) circle (3pt);
\draw[fill=black] (0,1.5) circle (3pt);
\draw[fill=black] (1,0.5) circle (3pt);
\draw[fill=white] (1,1.5) circle (3pt);
\end{tikzpicture}.
\]
In the first case, the path must have exited the right boundary in the white row. There must have been infinitely many columns with the path traveling horizontally in a white vertex while the gray vertex is empty, with each of these columns contributing a weight of $xy$. As $|xy|<1$, the overall weight of such a configuration is zero. 

Thus the second case is the only one with nonzero weight. Removing the cross at a cost of its weight $1-xy$ gives the desired commutation relation.
\end{proof}

\begin{proof}[Proof of Thm. \ref{thm:genfunRPP}]
By repeated applications of Lemma \ref{lem:whitegrayswap} we can push all the gray rows down to get
\begin{equation}\label{eq:alltheswaps}
\begin{tikzpicture}[scale=0.7,baseline=(current bounding box.center)]
\draw[fill=lightgray] (1,1) rectangle (8,2);
\draw[fill=lightgray] (1,4) rectangle (8,5);
\draw[fill=lightgray] (1,6) rectangle (8,7);
\draw (1,0) grid (8,7);

\draw[fill=white] (1,0.5) circle (3pt);
\draw[fill=white] (1,1.5) circle (3pt);
\draw[fill=white] (1,2.5) circle (3pt);
\draw[fill=white] (1,3.5) circle (3pt);
\draw[fill=white] (1,4.5) circle (3pt);
\draw[fill=white] (1,5.5) circle (3pt);
\draw[fill=white] (1,6.5) circle (3pt);

\draw[fill=white] (8,0.5) circle (3pt);
\draw[fill=black] (8,1.5) circle (3pt);
\draw[fill=white] (8,2.5) circle (3pt);
\draw[fill=white] (8,3.5) circle (3pt);
\draw[fill=black] (8,4.5) circle (3pt);
\draw[fill=white] (8,5.5) circle (3pt);
\draw[fill=black] (8,6.5) circle (3pt);

\draw[fill=black] (1.5,0) circle (3pt);
\draw[fill=black] (2.5,0) circle (3pt);
\draw[fill=black] (3.5,0) circle (3pt);
\draw[fill=white] (4.5,0) circle (3pt);
\draw[fill=white] (5.5,0) circle (3pt);
\draw[fill=white] (6.5,0) circle (3pt);
\draw[fill=white] (7.5,0) circle (3pt);

\draw[fill=white] (1.5,7) circle (3pt);
\draw[fill=white] (2.5,7) circle (3pt);
\draw[fill=white] (3.5,7) circle (3pt);
\draw[fill=white] (4.5,7) circle (3pt);
\draw[fill=white] (5.5,7) circle (3pt);
\draw[fill=white] (6.5,7) circle (3pt);
\draw[fill=white] (7.5,7) circle (3pt);

\draw[red,fill=red] (4,0) circle (1pt);
\draw[red,fill=red] (4,1) circle (1pt);
\draw[red,fill=red] (3,2) circle (1pt);
\draw[red,fill=red] (3,3) circle (1pt);
\draw[red,fill=red] (3,4) circle (1pt);
\draw[red,fill=red] (2,5) circle (1pt);
\draw[red,fill=red] (2,6) circle (1pt);
\draw[red,fill=red] (1,7) circle (1pt);

\node at (7.5,0.5) {$\dots$};
\node at (7.5,1.5) {$\dots$};
\node at (7.5,2.5) {$\dots$};
\node at (7.5,3.5) {$\dots$};
\node at (7.5,4.5) {$\dots$};
\node at (7.5,5.5) {$\dots$};
\node at (7.5,6.5) {$\dots$};

\draw[|-] (1,-0.5)--(2.1,-0.5);
\draw[-|] (2.9,-0.5)--(4,-0.5);
\node at (2.5,-0.5) {$\ell(\lambda)$};
\end{tikzpicture}
= 
C(q)
\begin{tikzpicture}[scale=0.7,baseline=(current bounding box.center)]
\draw[fill=lightgray] (1,0) rectangle (8,1);
\draw[fill=lightgray] (1,1) rectangle (8,2);
\draw[fill=lightgray] (1,2) rectangle (8,3);
\draw (1,0) grid (8,7);

\draw[fill=white] (1,0.5) circle (3pt);
\draw[fill=white] (1,1.5) circle (3pt);
\draw[fill=white] (1,2.5) circle (3pt);
\draw[fill=white] (1,3.5) circle (3pt);
\draw[fill=white] (1,4.5) circle (3pt);
\draw[fill=white] (1,5.5) circle (3pt);
\draw[fill=white] (1,6.5) circle (3pt);

\draw[fill=black] (8,0.5) circle (3pt);
\draw[fill=black] (8,1.5) circle (3pt);
\draw[fill=black] (8,2.5) circle (3pt);
\draw[fill=white] (8,3.5) circle (3pt);
\draw[fill=white] (8,4.5) circle (3pt);
\draw[fill=white] (8,5.5) circle (3pt);
\draw[fill=white] (8,6.5) circle (3pt);

\draw[fill=black] (1.5,0) circle (3pt);
\draw[fill=black] (2.5,0) circle (3pt);
\draw[fill=black] (3.5,0) circle (3pt);
\draw[fill=white] (4.5,0) circle (3pt);
\draw[fill=white] (5.5,0) circle (3pt);
\draw[fill=white] (6.5,0) circle (3pt);
\draw[fill=white] (7.5,0) circle (3pt);

\draw[fill=white] (1.5,7) circle (3pt);
\draw[fill=white] (2.5,7) circle (3pt);
\draw[fill=white] (3.5,7) circle (3pt);
\draw[fill=white] (4.5,7) circle (3pt);
\draw[fill=white] (5.5,7) circle (3pt);
\draw[fill=white] (6.5,7) circle (3pt);
\draw[fill=white] (7.5,7) circle (3pt);

\draw[red,fill=red] (4,0) circle (1pt);
\draw[red,fill=red] (4,1) circle (1pt);
\draw[red,fill=red] (3,2) circle (1pt);
\draw[red,fill=red] (3,3) circle (1pt);
\draw[red,fill=red] (3,4) circle (1pt);
\draw[red,fill=red] (2,5) circle (1pt);
\draw[red,fill=red] (2,6) circle (1pt);
\draw[red,fill=red] (1,7) circle (1pt);

\node at (7.5,0.5) {$\dots$};
\node at (7.5,1.5) {$\dots$};
\node at (7.5,2.5) {$\dots$};
\node at (7.5,3.5) {$\dots$};
\node at (7.5,4.5) {$\dots$};
\node at (7.5,5.5) {$\dots$};
\node at (7.5,6.5) {$\dots$};

\draw[|-] (1,-0.5)--(2.1,-0.5);
\draw[-|] (2.9,-0.5)--(4,-0.5);
\node at (2.5,-0.5) {$\ell(\lambda)$};
\end{tikzpicture}
\end{equation}
where the factor $C(q)$ comes from the many row swaps. Note that the RHS has one valid path configuration with weight given by $\frac{1}{A_{\lambda}(q)}$. We claim that 
\begin{equation} \label{eq:C}
C(q) =  \prod_{c \in \lambda} \frac{1}{1-q^{h_\lambda(c)}}.
\end{equation}

Recall from the proof of Theorem \ref{thm:bij1color} that the $i^{th}$ gray row, ordered from topmost to bottommost, will have parameter $q^{\lambda_i+\ell-i+1}$ where $\ell=\ell(\lambda)$. For the white rows, we have that the $j^{th}$ white row, ordered from bottommost to topmost, will have parameter $q^{\lambda_j' - \ell - j}$. To see this, note that if we start with the white rows packed at the top (as in the LHS of Equation \eqref{eq:alltheswaps}) then the row corresponding to $\lambda_j'$ start as row $\ell+j$ from the bottom and must swap $\lambda_j'$ times to get to its final position. Thus it ends up as row $\ell+j-\lambda_j'$ from the bottom. 

We see we can associate a row $\lambda_i$ to each gray row and a column $\lambda_j'$ to each white row. Note that to get from the RHS of Equation \eqref{eq:alltheswaps} to the LHS, the gray row associated to $\lambda_i$ only swaps with white rows associate to $\lambda_j'$ only if $\lambda_i \ge j$.  Thus to each pair of gray and white rows that swap, we can associate a cell in the Young diagram of $\lambda$ that lies in the intersection of the corresponding row and column of $\lambda$. When swapping the rows (from gray row on top to gray row on bottom), we see from Lemma \ref{lem:whitegrayswap} that we pick up a factor of
\[
\left(1-q^{\lambda_i + \ell - i +1}q^{\lambda_j' - \ell - j}\right)^{-1} = \left(1-q^{\lambda_i - i + \lambda_j' - j + 1} \right)^{-1} =  \left(1-q^{h_{\lambda}(c)} \right)^{-1}
\]
where $c$ is the cell in row $i$ and column $j$ of the Young diagram. Collecting these factors for each row swap gives Equation \eqref{eq:C}.

Altogether we have 
\[
\frac{1}{A_{\lambda}(q)}\prod_{c \in \lambda} \frac{1}{1-q^{h_\lambda(c)}} = \sum_{\mathcal{C}} w(\mathcal{C}) = \frac{1}{A_{\lambda}(q)}\sum_{\Lambda} q^{|\Lambda|}
\]
where the last equality follows from Theorem \ref{thm:bij1color}. It follows that
\[
\sum_{\Lambda} q^{|\Lambda|} = \prod_{c \in \lambda} \frac{1}{1-q^{h_\lambda(c)}}
\]
as desired.
\end{proof}

\subsection{An alternate path formulation}
We end this section by describing a mapping between RPPs and certain nonintersecting paths related to the lozenge tilings. This way of viewing the RPPs will be advantageous when we define interacting RPPs in Section \ref{sec:overlappingRPPs}.

We call a skew-partition $\lambda/\mu$ a \emph{border strip} if it does not contain a $2\times2$ box in its Young diagram. Given the Young diagram of a partition $\lambda$, we can recursively construct a sequence of border strips by first finding the largest border strip in $\lambda$, removing it, then repeating. We will indicate these border strip by drawing paths over the cells in each border strip as shown in Example \ref{ex:borderstrip}.

\begin{example}\label{ex:borderstrip}
Let $\lambda=(4,4,3,3,1)$. The border strip decomposition drawn as paths is given below.
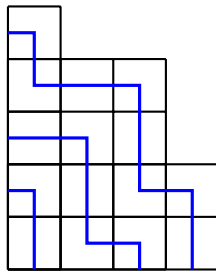
\[
\begin{tikzpicture}[scale=0.7]
    \draw[black,thick] (0,0) grid (3,4);
    \draw[black,thick] (0,0) grid (4,2);
    \draw[black,thick] (0,0) grid (1,5);
    \draw[blue,very thick] (0,4.5)--(0.5,4.5)--(0.5,3.5)--(2.5,3.5)--(2.5,1.5)--(3.5,1.5)--(3.5,0);
    \draw[blue,very thick] (0,2.5)--(1.5,2.5)--(1.5,0.5)--(2.5,0.5)--(2.5,0);
    \draw[blue,very thick] (0,1.5)--(0.5,1.5)--(0.5,0);
\end{tikzpicture}
\]
\end{example}
Given a RPP of shape $\lambda$, we can think of these paths as being drawn on top of the corresponding stacks of cubes. To make each path connected, we assign paths to the other lozenge tilings according to the following rules:
\begin{itemize}
\item In a vertical slice of the lozenge tiling corresponding to hole in the Maya diagram of $\lambda$, we have
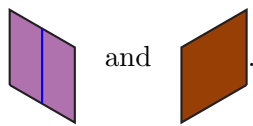
\[
\begin{tikzpicture}[baseline=(current bounding box.center)]
\fill[Orchid,draw=Black,thick,shift={(150:1)}] (0,0)--(0,1)--(30:1)--(-30:1)--(0,0);
\draw[blue,thick] (-0.435,0.25)--(-0.435,1.25);
\end{tikzpicture}
\quad \text{and} \quad
\begin{tikzpicture}[baseline=(current bounding box.center)]
\fill[RawSienna,draw=Black,thick,shift={(-0.05,0.05)}] (0,0)--(0,1)--(150:1)--(-150:1)--(0,0);
\end{tikzpicture}.
\]
\item In a vertical slice of the lozenge tiling corresponding to particle in the Maya diagram of $\lambda$, we have
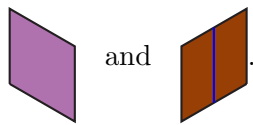
\[ 
\begin{tikzpicture}[baseline=(current bounding box.center)]
\fill[Orchid,draw=Black,thick,shift={(150:1)}] (0,0)--(0,1)--(30:1)--(-30:1)--(0,0);
\end{tikzpicture}
\quad \text{and} \quad
\begin{tikzpicture}[baseline=(current bounding box.center)]
\fill[RawSienna,draw=Black,thick,shift={(-0.05,0.05)}] (0,0)--(0,1)--(150:1)--(-150:1)--(0,0);
\draw[blue,thick] (-0.485,-0.2)--(-0.485,0.8);
\end{tikzpicture}. 
\]
\end{itemize}
In Example of these paths is given in Example \ref{ex:RPPpaths}.

\begin{example} \label{ex:RPPpaths}
Below we show a RPP of shape $\lambda=(4,4,3,3,1)$ and the corresponding paths coming from the border strip decomposition in Example \ref{ex:borderstrip}.
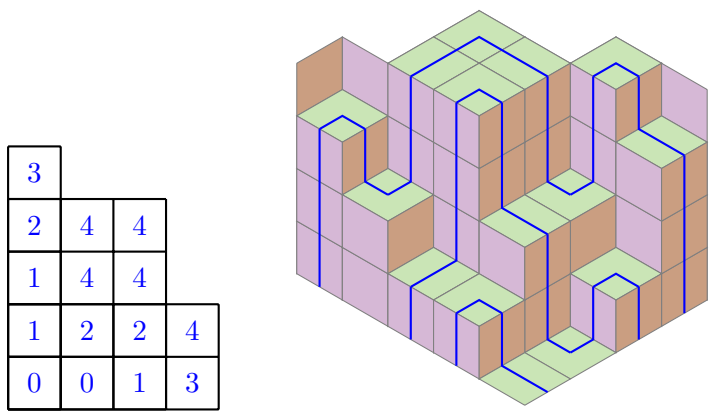
\[
\begin{tikzpicture}[scale=0.7]
    \draw[black,thick] (0,0) grid (3,4);
    \draw[black,thick] (0,0) grid (1,5);
    \draw[black,thick] (0,0) grid (4,2);
    \draw[blue] node at (0.5,0.5) {$0$};
    \draw[blue] node at (0.5,1.5) {$1$};
    \draw[blue] node at (0.5,2.5) {$1$};
    \draw[blue] node at (1.5,0.5) {$0$};
    \draw[blue] node at (2.5,0.5) {$1$};
    \draw[blue] node at (1.5,1.5) {$2$};
    \draw[blue] node at (1.5,2.5) {$4$};
    \draw[blue] node at (0.5,3.5) {$2$};
    \draw[blue] node at (0.5,4.5) {$3$};
    \draw[blue] node at (1.5,3.5) {$4$};
    \draw[blue] node at (2.5,1.5) {$2$};
    \draw[blue] node at (2.5,2.5) {$4$};
    \draw[blue] node at (2.5,3.5) {$4$};
    \draw[blue] node at (3.5,0.5) {$3$};
    \draw[blue] node at (3.5,1.5) {$4$};
\end{tikzpicture}\hspace{1cm}
\begin{tikzpicture}[scale=0.7]
    \reverseplanepartition{4}{{3},{2,4,4},{1,4,4},{1,2,2,4},{0,0,1,3}}
    \fill[White,fill opacity=0.5] (-0.9,-3.1) rectangle (7,4.6);
    \draw[blue,thick,shift={(150:0.5)}] (0,-1)--(0,2);
    \draw[blue,thick,shift={(0,2.5)}] (-150:0.5)--(0,0)--(-30:0.5);
    \draw[blue,thick,shift={(0,1)},shift={(30:0.5)}] (0,1)--(0,0)--(-30:0.5);
    \draw[blue,thick,shift={(30:1.5)},shift={(0,0.5)}] (-150:0.5)--(0,0)--(0,2);
    \draw[blue,thick,shift={(30:3)},shift={(0,2.5)}] (-150:1.5)--(0,0)--(-30:1.5);
    \draw[blue,thick,shift={(30:4.5)},shift={(0,-1)}] (0,2)--(0,0)--(-30:0.5);
    \draw[blue,thick,shift={(30:4)},shift={(-30:1.5)}] (-150:0.5)--(0,0)--(0,2);
    \draw[blue,thick,shift={(30:6)},shift={(0,0.5)}] (-150:0.5)--(0,0)--(-30:0.5);
    \draw[blue,thick,shift={(30:6.5)},shift={(0,-1)}] (0,1)--(0,0)--(-30:1);
    \draw[blue,thick,shift={(30:2.5)},shift={(-30:5)}] (0,3)--(0,0);
    \draw[blue,thick,shift={(-30:2.5)},shift={(0,-1)}] (0,0)--(0,1);
    \draw[blue,thick,shift={(30:3)},shift={(0,-1.5)},shift={(0,-1)}] (-150:0.5)--(0,0)--(-30:0.5);
    \draw[blue,thick,shift={(-30:3)},shift={(30:0.5)},shift={(0,-1)}] (0,1)--(0,0)--(-30:1);
    \draw[blue,thick,shift={(-30:1.5)},shift={(0,-1)}] (0,0)--(0,1);
    \draw[blue,thick,shift={(30:2)},shift={(-30:0.5)},shift={(0,-1)}] (-150:1)--(0,0)--(0,3);
    \draw[blue,thick,shift={(30:3)},shift={(0,2.5)},shift={(0,-1)}] (-150:0.5)--(0,0)--(-30:0.5);
    \draw[blue,thick,shift={(30:3.5)},shift={(0,-1)}] (0,2)--(0,0)--(-30:1);
    \draw[blue,thick,shift={(-30:4.5)},shift={(0,1.5)},shift={(0,-1)}] (0,2)--(0,0)--(-30:0.5);
    \draw[blue,thick,shift={(-30:3.5)},shift={(30:2)},shift={(0,-1)}] (-150:0.5)--(0,0)--(0,1);
    \draw[blue,thick,shift={(-30:2.5)},shift={(30:3.5)},shift={(0,-1)}] (-150:0.5)--(0,0)--(-30:0.5);
    \draw[blue,thick,shift={(-30:4)},shift={(30:2.5)},shift={(0,-1)}] (0,1)--(0,0);
\end{tikzpicture}
\]
\end{example}

One advantage of the path formulation is that it differentiates between the lozenges that correspond to white rows in the vertex model and those corresponding to gray rows. In fact, by following through the string of bijections, one can make the following dictionaries. For white vertices we have
\begin{equation}\label{eq:whitedict}
\begin{tabular}{cccc}
\begin{tikzpicture}[baseline=(current bounding box.center)]\draw[thin] (0,0) rectangle (1,1);
\end{tikzpicture} & 
\begin{tikzpicture}[baseline=(current bounding box.center)]\draw[thin] (0,0) rectangle (1,1);\draw[black, very thick] (0.5,0)--(0.5,0.5);\end{tikzpicture} & 
\begin{tikzpicture}[baseline=(current bounding box.center)] \draw[thin] (0,0) rectangle (2,1); \draw[thin] (0,0) grid (2,1); \draw[black, very thick] (0.5,0.5)--(1.5,0.5); \end{tikzpicture} & 
\begin{tikzpicture}[baseline=(current bounding box.center)] \draw[thin] (0,0) rectangle (1,1); \draw[black, very thick] (0.5,0.5)--(0.5,1); \end{tikzpicture} \\ \\
\begin{tikzpicture}[baseline=(current bounding box.center)]
\fill[RawSienna,draw=Black,thick,shift={(-0.05,0.05)}] (0,0)--(0,1)--(150:1)--(-150:1)--(0,0);
\end{tikzpicture} &
\begin{tikzpicture}[baseline=(current bounding box.center)]
\fill[YellowGreen,draw=Black,thick] (0,0)--(150:1)--(0,1)--(30:1)--(0,0);
\draw[blue,thick] (0,0.5)--(0.435,0.75);
\draw[dashed] (0,0)--(0,1);
\draw[dashed] (0.87,0)--(0.87,1);
\end{tikzpicture} &
\begin{tikzpicture}[baseline=(current bounding box.center)]
\fill[Orchid,draw=Black,thick,shift={(150:1)}] (0,0)--(0,1)--(30:1)--(-30:1)--(0,0);
\draw[blue,thick] (-0.435,0.25)--(-0.435,1.25);
\end{tikzpicture} &
\begin{tikzpicture}[baseline=(current bounding box.center)]
\fill[YellowGreen,draw=Black,thick] (0,0)--(150:1)--(0,1)--(30:1)--(0,0);
\draw[blue,thick] (-0.435,0.25)--(0,0.5);
\draw[dashed] (0,0)--(0,1);
\draw[dashed] (-0.87,0)--(-0.87,1);
\end{tikzpicture} 
\end{tabular}
\end{equation}
and for gray vertices we have
\begin{equation}\label{eq:graydict}
\begin{tabular}{cccc}
\begin{tikzpicture}[baseline=(current bounding box.center)]\draw[thin,fill=lightgray] (0,0) rectangle (1,1);  \end{tikzpicture} & 
\begin{tikzpicture}[baseline=(current bounding box.center)]\draw[thin,fill=lightgray] (0,0) rectangle (1,1);\draw[black, very thick] (0.5,0)--(0.5,0.5);\end{tikzpicture} & 
\begin{tikzpicture}[baseline=(current bounding box.center)] \draw[thin,fill=lightgray] (0,0) rectangle (2,1); \draw[thin] (0,0) grid (2,1); \draw[black, very thick] (0.5,0.5)--(1.5,0.5); \end{tikzpicture} & 
\begin{tikzpicture}[baseline=(current bounding box.center)] \draw[thin,fill=lightgray] (0,0) rectangle (1,1); \draw[black, very thick] (0.5,0.5)--(0.5,1); \end{tikzpicture} \\ \\
\begin{tikzpicture}[baseline=(current bounding box.center)]
\fill[RawSienna,draw=Black,thick,shift={(-0.05,0.05)}] (0,0)--(0,1)--(150:1)--(-150:1)--(0,0);
\draw[blue,thick] (-0.485,-0.2)--(-0.485,0.8);
\end{tikzpicture} &
\begin{tikzpicture}[baseline=(current bounding box.center)]
\fill[YellowGreen,draw=Black,thick] (0,0)--(150:1)--(0,1)--(30:1)--(0,0);
\draw[blue,thick] (0,0.5)--(0.435,0.25);
\draw[dashed] (0,0)--(0,1);
\draw[dashed] (0.87,0)--(0.87,1);
\end{tikzpicture} &
\begin{tikzpicture}[baseline=(current bounding box.center)]
\fill[Orchid,draw=Black,thick,shift={(150:1)}] (0,0)--(0,1)--(30:1)--(-30:1)--(0,0);
\end{tikzpicture} &
\begin{tikzpicture}[baseline=(current bounding box.center)]
\fill[YellowGreen,draw=Black,thick] (0,0)--(150:1)--(0,1)--(30:1)--(0,0);
\draw[blue,thick] (-0.435,0.75)--(0,0.5);
\draw[dashed] (0,0)--(0,1);
\draw[dashed] (-0.87,0)--(-0.87,1);
\end{tikzpicture} 
\end{tabular}
\end{equation}
where the dashed lines in the second and fourth columns are to differentiate lozenges entering a vertical slice from the left versus those exiting to the right.

\section{ Multicolored vertex models and interacting RPP} \label{sec:multicolored}
We have shown that reverse plane partitions can be represented both by a vertex model and a lozenge diagram. We now generalize the vertex models of the previous section by introducing a multicolored vertex model first described in \cite{aggarwal2023colored,LLT}. Rather than consisting of paths of a single color, the vertex model will consist of path of several different colors. Allowed configurations of paths of the same color are exactly the same as those described in Section \ref{sec:5V} but we now include an interaction between paths of different colors in the weights. This interaction creates a coupling between the different RPPs.  As similar coupling for domino tilings of the Aztec diamond was defined in \cite{CGK} and further studied in \cite{Keating2023ShufflingAF}.

While in general the vertex model can be defined for an arbitrary number of colors, we will focus on the 2 color case. We begin by defining the 2-color vertex model generalizing our previous 5-vertex models. We will then use the vertex model to construct the coupled RPPs.

\subsection{White Vertices}
We will first look at the 2-color generalization of the white vertices. Figure \ref{fig:whitevertices2} show the allowed vertices and their weights.
\begin{figure}[h!]
\begin{center}
\begin{tabular}{ccccc}
\begin{tikzpicture}
\draw (-1,-1) -- (1,-1); \draw (-1,-1) -- (-1,1);
\draw (1,1) -- (1,-1); \draw (1,1) -- (-1,1);
\end{tikzpicture}
&
\begin{tikzpicture}
\draw (-1,-1) -- (1,-1); \draw (-1,-1) -- (-1,1);
\draw (1,1) -- (1,-1); \draw (1,1) -- (-1,1);
\draw[red] (0,-1) -- (0,1);
\end{tikzpicture}
&
\begin{tikzpicture}
\draw (-1,-1) -- (1,-1); \draw (-1,-1) -- (-1,1);
\draw (1,1) -- (1,-1); \draw (1,1) -- (-1,1);
\draw[red] (-1,0) -- (1,0);
\end{tikzpicture}
&
\begin{tikzpicture}
\draw (-1,-1) -- (1,-1); \draw (-1,-1) -- (-1,1);
\draw (1,1) -- (1,-1); \draw (1,1) -- (-1,1);
\draw[red] (0,-1) -- (0,0) -- (1,0);
\end{tikzpicture}
&
\begin{tikzpicture}
\draw (-1,-1) -- (1,-1); \draw (-1,-1) -- (-1,1);
\draw (1,1) -- (1,-1); \draw (1,1) -- (-1,1);
\draw[red] (-1,0) -- (0,0) -- (0,1);
\end{tikzpicture}
\\
$1$ & $1 $ & $x$ & $x$ & $1$
\\
\begin{tikzpicture}
\draw (-1,-1) -- (1,-1); \draw (-1,-1) -- (-1,1);
\draw (1,1) -- (1,-1); \draw (1,1) -- (-1,1);
\draw[blue] (0,-1) -- (0,1);
\end{tikzpicture}
&
\begin{tikzpicture}
\draw (-1,-1) -- (1,-1); \draw (-1,-1) -- (-1,1);
\draw (1,1) -- (1,-1); \draw (1,1) -- (-1,1);
\draw[red] (0.1,-1) -- (0.1,1);
\draw[blue] (-0.1,-1) -- (-0.1,1);
\end{tikzpicture}
&
\begin{tikzpicture}
\draw (-1,-1) -- (1,-1); \draw (-1,-1) -- (-1,1);
\draw (1,1) -- (1,-1); \draw (1,1) -- (-1,1);
\draw[red] (-1,0) -- (1,0);
\draw[blue] (0,-1) -- (0,1);
\end{tikzpicture}
&
\begin{tikzpicture}
\draw (-1,-1) -- (1,-1); \draw (-1,-1) -- (-1,1);
\draw (1,1) -- (1,-1); \draw (1,1) -- (-1,1);
\draw[red] (0.1,-1) -- (0.1,0) -- (1,0);
\draw[blue] (-0.1,-1) -- (-0.1,1);
\end{tikzpicture}
&
\begin{tikzpicture}
\draw (-1,-1) -- (1,-1); \draw (-1,-1) -- (-1,1);
\draw (1,1) -- (1,-1); \draw (1,1) -- (-1,1);
\draw[red] (-1,0) -- (0.1,0) -- (0.1,1);
\draw[blue] (-0.1,-1) -- (-0.1,1);
\end{tikzpicture}
\\
 $1$ & $1 $ & $x$ & $x$ & $1$
\\
\begin{tikzpicture}
\draw (-1,-1) -- (1,-1); \draw (-1,-1) -- (-1,1);
\draw (1,1) -- (1,-1); \draw (1,1) -- (-1,1);
\draw[blue] (-1,0) -- (1,0);
\end{tikzpicture}
&
\begin{tikzpicture}
\draw (-1,-1) -- (1,-1); \draw (-1,-1) -- (-1,1);
\draw (1,1) -- (1,-1); \draw (1,1) -- (-1,1);
\draw[red] (0,-1) -- (0,1);
\draw[blue] (-1,0) -- (1,0);
\end{tikzpicture}
&
\begin{tikzpicture}
\draw (-1,-1) -- (1,-1); \draw (-1,-1) -- (-1,1);
\draw (1,1) -- (1,-1); \draw (1,1) -- (-1,1);
\draw[red] (-1,-0.1) -- (1,-0.1);
\draw[blue] (-1,0.1) -- (1,0.1);
\end{tikzpicture}
&
\begin{tikzpicture}
\draw (-1,-1) -- (1,-1); \draw (-1,-1) -- (-1,1);
\draw (1,1) -- (1,-1); \draw (1,1) -- (-1,1);
\draw[red] (0,-1) -- (0,-0.1) -- (1,-0.1);
\draw[blue] (-1,0.1) -- (1,0.1);
\end{tikzpicture}
&
\begin{tikzpicture}
\draw (-1,-1) -- (1,-1); \draw (-1,-1) -- (-1,1);
\draw (1,1) -- (1,-1); \draw (1,1) -- (-1,1);
\draw[red] (-1,-0.1) -- (0,-0.1) -- (0,1);
\draw[blue] (-1,0.1) -- (1,0.1);
\end{tikzpicture}
\\
$x$ & $xt$ & $x^2t$ & $x^2t$ & $xt$
\\
\begin{tikzpicture}
\draw (-1,-1) -- (1,-1); \draw (-1,-1) -- (-1,1);
\draw (1,1) -- (1,-1); \draw (1,1) -- (-1,1);
\draw[blue] (0,-1) -- (0,0) -- (1,0);
\end{tikzpicture}
&
\begin{tikzpicture}
\draw (-1,-1) -- (1,-1); \draw (-1,-1) -- (-1,1);
\draw (1,1) -- (1,-1); \draw (1,1) -- (-1,1);
\draw[red] (0.1,-1) -- (0.1,1);
\draw[blue] (-0.1,-1) -- (-0.1,0) -- (1,0);
\end{tikzpicture}
&
\begin{tikzpicture}
\draw (-1,-1) -- (1,-1); \draw (-1,-1) -- (-1,1);
\draw (1,1) -- (1,-1); \draw (1,1) -- (-1,1);
\draw[red] (-1,-0.1) -- (1,-0.1);
\draw[blue] (0,-1) -- (0,0.1) -- (1,0.1);   
\end{tikzpicture}
&
\begin{tikzpicture}
\draw (-1,-1) -- (1,-1); \draw (-1,-1) -- (-1,1);
\draw (1,1) -- (1,-1); \draw (1,1) -- (-1,1);
\draw[red] (0.1,-1) -- (0.1,-0.1) -- (1,-0.1);
\draw[blue] (-0.1,-1) -- (-0.1,0.1) -- (1,0.1);
\end{tikzpicture}
&
\begin{tikzpicture}
\draw (-1,-1) -- (1,-1); \draw (-1,-1) -- (-1,1);
\draw (1,1) -- (1,-1); \draw (1,1) -- (-1,1);
\draw[red] (-1,-0.1) -- (0.1,-0.1) -- (0.1,1);
\draw[blue] (-0.1,-1) -- (-0.1,0.1) -- (1,0.1);
\end{tikzpicture}
\\
$x$ & $xt$ & $x^2t$ & $x^2t$ & $xt$
\\
\begin{tikzpicture}
\draw (-1,-1) -- (1,-1); \draw (-1,-1) -- (-1,1);
\draw (1,1) -- (1,-1); \draw (1,1) -- (-1,1);
\draw[blue] (-1,0) -- (0,0) -- (0,1);
\end{tikzpicture}
&
\begin{tikzpicture}
\draw (-1,-1) -- (1,-1); \draw (-1,-1) -- (-1,1);
\draw (1,1) -- (1,-1); \draw (1,1) -- (-1,1);
\draw[red] (0.1,-1) -- (0.1,1);
\draw[blue] (-1,0) -- (-0.1,0) -- (-0.1,1);
\end{tikzpicture}
&
\begin{tikzpicture}
\draw (-1,-1) -- (1,-1); \draw (-1,-1) -- (-1,1);
\draw (1,1) -- (1,-1); \draw (1,1) -- (-1,1);
\draw[red] (-1,-0.1) -- (1,-0.1);
\draw[blue] (-1,0.1) -- (0,0.1) -- (0,1);
\end{tikzpicture}
&
\begin{tikzpicture}
\draw (-1,-1) -- (1,-1); \draw (-1,-1) -- (-1,1);
\draw (1,1) -- (1,-1); \draw (1,1) -- (-1,1);
\draw[red] (0.1,-1) -- (0.1,-0.1) -- (1,-0.1);
\draw[blue] (-1,0.1) -- (-0.1,0.1) -- (-0.1,1);
\end{tikzpicture}
&
\begin{tikzpicture}
\draw (-1,-1) -- (1,-1); \draw (-1,-1) -- (-1,1);
\draw (1,1) -- (1,-1); \draw (1,1) -- (-1,1);
\draw[red] (-1,-0.1) -- (0.1,-0.1) -- (0.1,1);
\draw[blue] (-1,0.1) -- (-0.1,0.1) -- (-0.1,1);
\end{tikzpicture}
\\
$1$ & $1$ & $x$ & $x$ & $1$
\end{tabular}
\end{center}
\caption{The allowed local configurations and the corresponding weights for the 2-color white vertices.}
\label{fig:whitevertices2}
\end{figure}
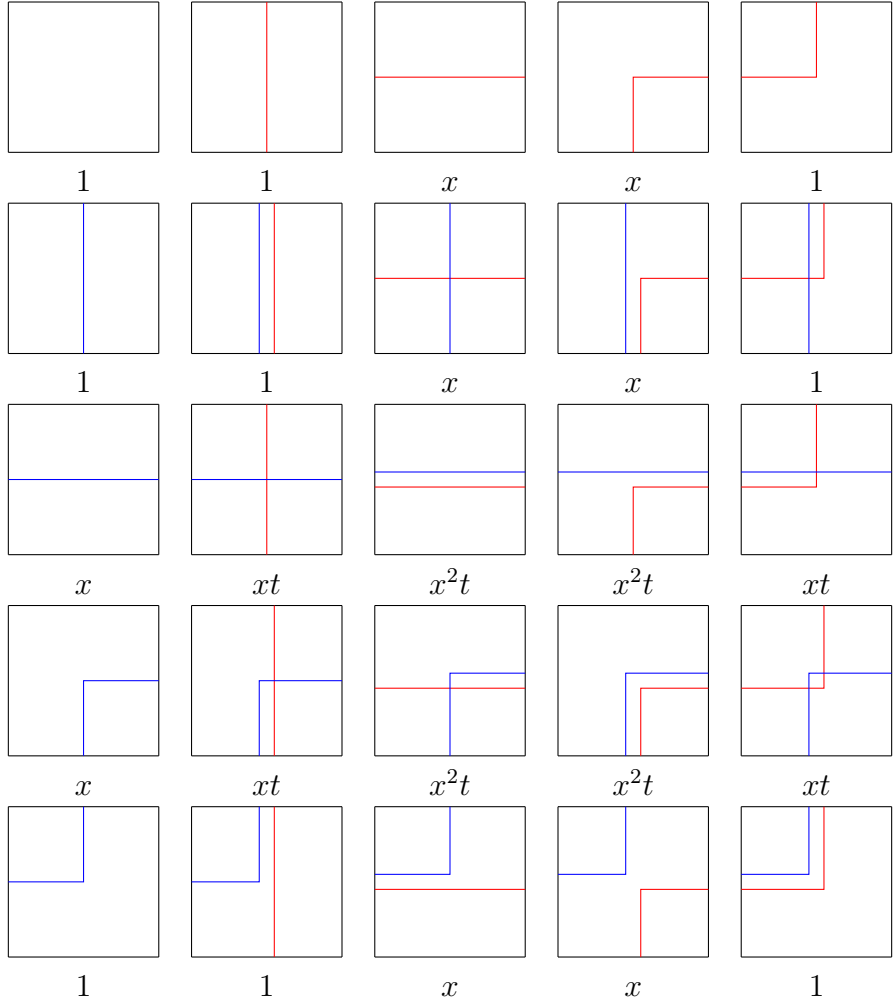
\\
Each of the colors of path is allowed to take one of the five possible vertex configurations from Figure \ref{fig:whiteweight}. The weights contain a power of a parameter $x$ whenever a path exits a vertex to the right, just as before, but now have a new parameter $t$ governing the strength of the interaction between the colors. Let $v_b$ be the vertex configuration for the blue paths and $v_r$ the vertex configuration for the red path.  Then the weights we have above can be written as
\[
L^{(2)}_{x;t}((v_b,v_r)) = L_{xt^{\delta_1}}(v_b)L_x(v_r), \qquad \delta_1 =  \mathbf{1}(\text{red is present})
\]
where $L_x(v)$ is defined in Eqn. \eqref{eq:Lmat}. In words, the interaction can be described by saying that the weight gets a factor of $t$ if a blue path exits right in a vertex in which a red path is present. Note when $t=1$ this reduces to the product of the weights is the product of the weight for each color using the white vertex weights given Figure \ref{fig:whiteweight}.

As in the 1-color case, we will consider a row of white vertices extending infinitely to the right. Now our boundary conditions for the rows will be indexed by a pair of partitions $\vec \lambda = (\lambda^{(1)},\lambda^{(2)})$, one for blue and one for red. See Example \ref{ex:whiterow2}.

\begin{example} \label{ex:whiterow2}
 As an example consider the row with bottom boundary condition given by $\vec \mu = ((2,1),(1,0))$ and top boundary condition given by $\vec \lambda = ((4,2),(4,1))$.
\begin{center}
\begin{tikzpicture}
\draw (0,0) grid (8,1);
\node[left] at (0,0.5) {$x$};

\draw[thick, blue] (3.4,0)--(3.4,0.6)--(5.4,0.6)--(5.4,1);
\draw[thick, blue] (1.4,0)--(1.4,0.6)--(2.4,0.6)--(2.4,1);
\draw[thick, red] (2.6,0)--(2.6,0.4)--(5.6,0.4)--(5.6,1);
\draw[thick, red] (0.6,0)--(0.6,0.4)--(1.6,0.4)--(1.6,1);

\node[below] at (4,0) {$\vec \mu = ((2,1),(1,0))$};
\node[above] at (4,1) {$\vec \lambda = ((4,2),(4,1))$};
\node[] at (7.5,0.5) {$\dots$};
\draw[red,fill=red] (2,0) circle (1pt);
\draw[red,fill=red] (2,1) circle (1pt);

\draw[fill=blue] (1.4,0) circle (0.1);
\draw[fill=blue] (3.4,0) circle (0.1);
\draw[fill=blue] (2.4,1) circle (0.1);
\draw[fill=blue] (5.4,1) circle (0.1);

\draw[fill=red] (0.6,0) circle (0.1);
\draw[fill=red] (2.6,0) circle (0.1);
\draw[fill=red] (1.6,1) circle (0.1);
\draw[fill=red] (5.6,1) circle (0.1);

\draw[fill=white] (4.5,0) circle (0.1);
\draw[fill=white] (5.5,0) circle (0.1);
\draw[fill=white] (6.5,0) circle (0.1);
\draw[fill=white] (7.5,0) circle (0.1);

\draw[fill=white] (0.5,1) circle (0.1);
\draw[fill=white] (3.5,1) circle (0.1);
\draw[fill=white] (4.5,1) circle (0.1);
\draw[fill=white] (6.5,1) circle (0.1);
\draw[fill=white] (7.5,1) circle (0.1);

\draw[fill=white] (0,0.5) circle (0.1);
\draw[fill=white] (8,0.5) circle (0.1);
\end{tikzpicture}
\end{center}
Here a white circle indicates no paths entering or exiting at that boundary. The weight of the row is a product of the weight of each vertex. In this example, the row has weight $x^7 t^3$. Note that the factor of $x$ in the weight is the product of what one would get from each color treating them as a 1-color vertex model. In particular, 
\[
x^7 = x^3 \cdot x^4 = x^{|\lambda^{(1)}/\mu^{(1)}|} \cdot x^{|\lambda^{(2)}/\mu^{(2)}|}.
\]
\end{example}
\noindent Note that we must have $\mu^{(i)}\preceq \lambda^{(i)}$, $i=1,2$ for there to be a valid configuration.

This vertex model still satisfies a version of the Yang-Baxter equation \cite{aggarwal2023colored,LLT}. In the one-color case, recall that the cross weights were given by
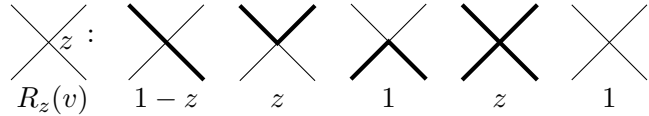
\[
  \begin{tabular}{cccccc}
     \begin{tikzpicture}[baseline=(current bounding box.center)] \draw[thin] (0,1)--(1,0);\draw[thin] (0,0)--(1,1); \node[right] at (0.5,0.5) {$z$}; \end{tikzpicture}: &    
     \begin{tikzpicture}[baseline=(current bounding box.center)] \draw[ultra thick] (0,1)--(1,0);\draw[thin] (0,0)--(1,1); \end{tikzpicture} & 
     \begin{tikzpicture}[baseline=(current bounding box.center)] \draw[ultra thick] (0,1)--(0.5,0.5)--(1,1);\draw[thin] (0,0)--(0.5,0.5)--(1,0); \end{tikzpicture} &
      \begin{tikzpicture}[baseline=(current bounding box.center)] \draw[thin] (0,1)--(0.5,0.5)--(1,1);\draw[ultra thick] (0,0)--(0.5,0.5)--(1,0); \end{tikzpicture} & 
      \begin{tikzpicture}[baseline=(current bounding box.center)] \draw[ultra thick] (0,1)--(1,0);\draw[ultra thick] (0,0)--(1,1); \end{tikzpicture} & 
      \begin{tikzpicture}[baseline=(current bounding box.center)] \draw[thin] (0,1)--(1,0);\draw[thin] (0,0)--(1,1); \end{tikzpicture}  \\
      $R_z(v)$   & $1-z$ & $z$ & $1$ & $z$ & $1$
    \end{tabular}
\]
The cross weights int the 2-color case are given by
\begin{equation} \label{eq:crossweight2}
R^{(2)}_{z;t}((v_b,v_r)) = R_{z/t^{r}}(v_b)R_z(v_r)
\end{equation}
where $v_b$ is the vertex configuration of the blue paths, $v_r$ is the configuration of the red paths, and $r = \mathbf{1}\left(\text{red path looks like } \begin{tikzpicture}[baseline=(current bounding box.center)] \draw[ultra thick, red] (0,1)--(1,0);\draw[thin] (0,0)--(1,1); \end{tikzpicture} \right)$.

For example, we have
\begin{center}
\resizebox{0.8\textwidth}{!}{
\begin{tabular}{cccc}
$w\left(
  \begin{tikzpicture}[baseline=(current bounding box.center)] 
 \draw (-1,0.5)-- (0,1.5); \draw[blue, very thick] (-1,1.5) -- (0,0.5); 
 \draw[step=1.0,black,thin] (0,0) grid (1,2); 
 \draw[red, very thick] (0.5,0)--(0.5,0.4)--(1,0.4);
 \draw[blue, very thick] (0,0.5)--(1,0.5);
 \node[above left] at (0.5,0.5) {$x$}; \node[above left] at (0.5,1.5) {$y$};
 \node[left, scale=0.7] at (-0.6,1) {$\frac{y}{x}$};
 \end{tikzpicture} 
 \right)$
 &
 $=
w\left(
 \begin{tikzpicture}[baseline=(current bounding box.center)] 
 \draw[red, very thick] (2,0.4)--(1.5,0.9)-- (1,0.4); \draw (2,1.4)--(1.55,0.95)--(1,1.5); 
 \draw[blue, very thick] (1,1.5)--(2,0.5);
 \draw[step=1.0,black,thin] (0,0) grid (1,2); 
 \draw[red, very thick] (0.5,0)--(0.5,0.4)--(1,0.4);
 \draw[blue, very thick] (0,1.5)--(1,1.5);
 \node[above left] at (0.5,0.5) {$y$}; \node[above left] at (0.5,1.5) {$x$};
  \node[right, scale=0.7] at (1.6,0.95) {$\frac{y}{x}$};
 \end{tikzpicture}
 \right)$
 &$+$ &
 $w\left(
 \begin{tikzpicture}[baseline=(current bounding box.center)] 
 \draw[red, very thick] (2,0.4) -- (1,1.4); \draw (2,1.5) -- (1,0.5); 
 \draw[blue, very thick] (1,1.5)--(2,0.5);
 \draw[step=1.0,black,thin] (0,0) grid (1,2); 
 \draw[red, very thick] (0.5,0)--(0.5,1.4)--(1,1.4);
 \draw[blue, very thick] (0,1.5)--(1,1.5);
 \node[above left] at (0.5,0.5) {$y$}; \node[above left] at (0.5,1.5) {$x$};
   \node[right, scale=0.7] at (1.6,0.95) {$\frac{y}{x}$};
 \end{tikzpicture}
 \right)$  \\ \\
 $x^2t\left(1-\frac{y}{x}\right)$ & $=  xy\left(1-\frac{y}{x}\right)$ & $+$ & $x^2t\left(1-\frac{y}{x}\right)\left(1-\frac{y}{xt}\right) $
 \end{tabular}
 }
 \end{center}

\subsection{Gray Vertices}
We also generalize the gray vertices to the 2-color case. The allowed vertices and weights are given in Figure \ref{fig:grayvertices2} below.

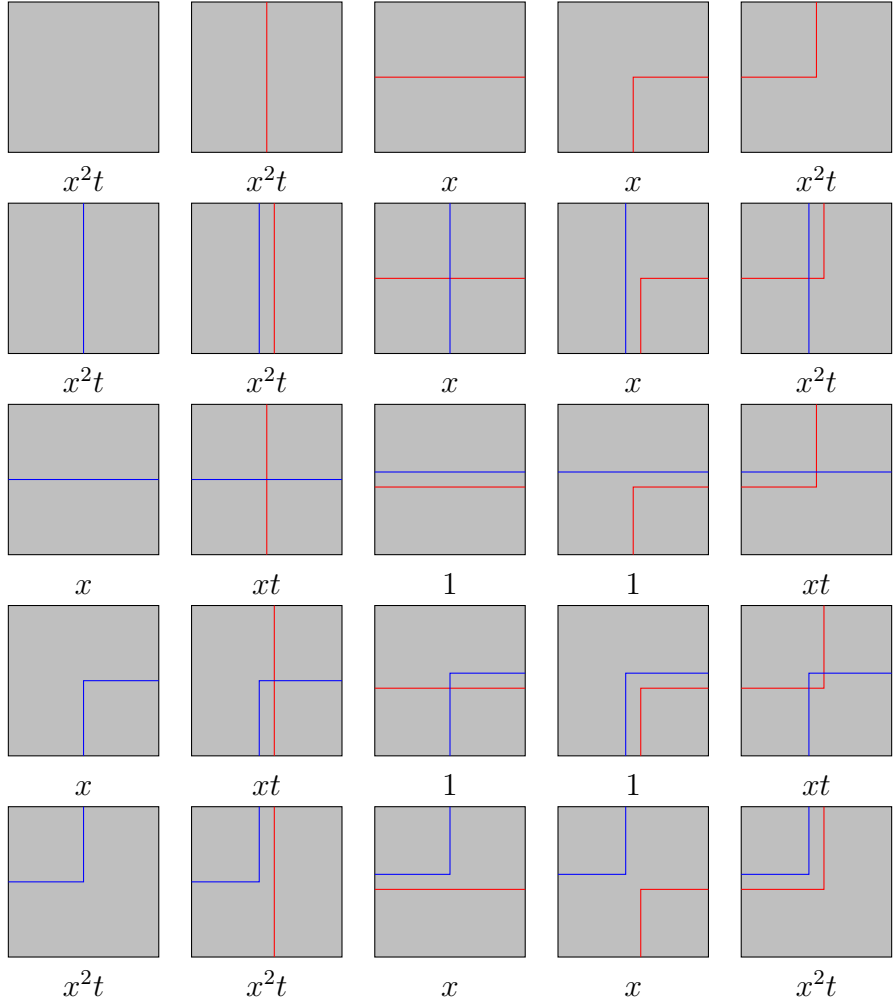
\begin{figure}[h]
\begin{center}
\begin{tabular}{ccccc}
\begin{tikzpicture}
\draw[fill=lightgray] (-1,-1) rectangle (1,1);
\end{tikzpicture}
&
\begin{tikzpicture}
\draw[fill=lightgray] (-1,-1) rectangle (1,1);
\draw[red] (0,-1) -- (0,1);
\end{tikzpicture}
&
\begin{tikzpicture}
\draw[fill=lightgray] (-1,-1) rectangle (1,1);
\draw[red] (-1,0) -- (1,0);
\end{tikzpicture}
&
\begin{tikzpicture}
\draw[fill=lightgray] (-1,-1) rectangle (1,1);
\draw[red] (0,-1) -- (0,0) -- (1,0);
\end{tikzpicture}
&
\begin{tikzpicture}
\draw[fill=lightgray] (-1,-1) rectangle (1,1);
\draw[red] (-1,0) -- (0,0) -- (0,1);
\end{tikzpicture}
\\
$x^2t$ & $x^2t$ & $x$ & $x$ & $x^2t$
\\
\begin{tikzpicture}
\draw[fill=lightgray] (-1,-1) rectangle (1,1);
\draw[blue] (0,-1) -- (0,1);
\end{tikzpicture}
&
\begin{tikzpicture}
\draw[fill=lightgray] (-1,-1) rectangle (1,1);
\draw[red] (0.1,-1) -- (0.1,1);
\draw[blue] (-0.1,-1) -- (-0.1,1);
\end{tikzpicture}
&
\begin{tikzpicture}
\draw[fill=lightgray] (-1,-1) rectangle (1,1);
\draw[red] (-1,0) -- (1,0);
\draw[blue] (0,-1) -- (0,1);
\end{tikzpicture}
&
\begin{tikzpicture}
\draw[fill=lightgray] (-1,-1) rectangle (1,1);
\draw[red] (0.1,-1) -- (0.1,0) -- (1,0);
\draw[blue] (-0.1,-1) -- (-0.1,1);
\end{tikzpicture}
&
\begin{tikzpicture}
\draw[fill=lightgray] (-1,-1) rectangle (1,1);
\draw[red] (-1,0) -- (0.1,0) -- (0.1,1);
\draw[blue] (-0.1,-1) -- (-0.1,1);
\end{tikzpicture}
\\
 $x^2t$ & $x^2t$ & $x$ & $x$ & $x^2t$
\\
\begin{tikzpicture}
\draw[fill=lightgray] (-1,-1) rectangle (1,1);
\draw[blue] (-1,0) -- (1,0);
\end{tikzpicture}
&
\begin{tikzpicture}
\draw[fill=lightgray] (-1,-1) rectangle (1,1);
\draw[red] (0,-1) -- (0,1);
\draw[blue] (-1,0) -- (1,0);
\end{tikzpicture}
&
\begin{tikzpicture}
\draw[fill=lightgray] (-1,-1) rectangle (1,1);
\draw[red] (-1,-0.1) -- (1,-0.1);
\draw[blue] (-1,0.1) -- (1,0.1);
\end{tikzpicture}
&
\begin{tikzpicture}
\draw[fill=lightgray] (-1,-1) rectangle (1,1);
\draw[red] (0,-1) -- (0,-0.1) -- (1,-0.1);
\draw[blue] (-1,0.1) -- (1,0.1);
\end{tikzpicture}
&
\begin{tikzpicture}
\draw[fill=lightgray] (-1,-1) rectangle (1,1);
\draw[red] (-1,-0.1) -- (0,-0.1) -- (0,1);
\draw[blue] (-1,0.1) -- (1,0.1);
\end{tikzpicture}
\\
$x$ & $xt$ & $1$ & $1$ & $xt$
\\
\begin{tikzpicture}
\draw[fill=lightgray] (-1,-1) rectangle (1,1);
\draw[blue] (0,-1) -- (0,0) -- (1,0);
\end{tikzpicture}
&
\begin{tikzpicture}
\draw[fill=lightgray] (-1,-1) rectangle (1,1);
\draw[red] (0.1,-1) -- (0.1,1);
\draw[blue] (-0.1,-1) -- (-0.1,0) -- (1,0);
\end{tikzpicture}
&
\begin{tikzpicture}
\draw[fill=lightgray] (-1,-1) rectangle (1,1);
\draw[red] (-1,-0.1) -- (1,-0.1);
\draw[blue] (0,-1) -- (0,0.1) -- (1,0.1);
\end{tikzpicture}
&
\begin{tikzpicture}
\draw[fill=lightgray] (-1,-1) rectangle (1,1);
\draw[red] (0.1,-1) -- (0.1,-0.1) -- (1,-0.1);
\draw[blue] (-0.1,-1) -- (-0.1,0.1) -- (1,0.1);
\end{tikzpicture}
&
\begin{tikzpicture}
\draw[fill=lightgray] (-1,-1) rectangle (1,1);
\draw[red] (-1,-0.1) -- (0.1,-0.1) -- (0.1,1);
\draw[blue] (-0.1,-1) -- (-0.1,0.1) -- (1,0.1);
\end{tikzpicture}
\\
$x$ & $xt$ & $1$ & $1$ & $xt$
\\
\begin{tikzpicture}
\draw[fill=lightgray] (-1,-1) rectangle (1,1);
\draw[blue] (-1,0) -- (0,0) -- (0,1);
\end{tikzpicture}
&
\begin{tikzpicture}
\draw[fill=lightgray] (-1,-1) rectangle (1,1);
\draw[red] (0.1,-1) -- (0.1,1);
\draw[blue] (-1,0) -- (-0.1,0) -- (-0.1,1);
\end{tikzpicture}
&
\begin{tikzpicture}
\draw[fill=lightgray] (-1,-1) rectangle (1,1);
\draw[red] (-1,-0.1) -- (1,-0.1);
\draw[blue] (-1,0.1) -- (0,0.1) -- (0,1);
\end{tikzpicture}
&
\begin{tikzpicture}
\draw[fill=lightgray] (-1,-1) rectangle (1,1);
\draw[red] (0.1,-1) -- (0.1,-0.1) -- (1,-0.1);
\draw[blue] (-1,0.1) -- (-0.1,0.1) -- (-0.1,1);
\end{tikzpicture}
&
\begin{tikzpicture}
\draw[fill=lightgray] (-1,-1) rectangle (1,1);
\draw[red] (-1,-0.1) -- (0.1,-0.1) -- (0.1,1);
\draw[blue] (-1,0.1) -- (-0.1,0.1) -- (-0.1,1);
\end{tikzpicture}
\\
$x^2t$ & $x^2t$ & $x$ & $x$ & $x^2t$
\end{tabular}
\end{center}
\caption{The allowed local configurations and the corresponding weights for the 2-color gray vertices.}
\label{fig:grayvertices2}
\end{figure}

Recall that in the one-color case the weight of a gray vertex could be as seen as coming from change of variable of the weight of a white vertex. In particular, we had 
\[
L'_x(v) = x \; L_{1/x}(v).
\]
In the 2-color case this becomes
\[
(L')^{(2)}_{x;t}((v_b,v_r)) = x^2 t \; L^{(2)}_{x;t}((v_b,v_r)), \qquad \bar x = \frac{1}{xt}.
\]
It was shown in Appendix B of \cite{CGK} that the contribution of color $i$ (where blue is color 1 and red is color 2) can be written as
\begin{equation}\label{eqn:gray_weights}
\resizebox{0.7\textwidth}{!}{
\begin{tabular}{cccccc}
Type 1 & Type 2 & Type 3 & Type 4 & Type 5 \\
  \begin{tikzpicture}[baseline=(current bounding box.center)]\draw[thin,fill=lightgray] (0,0) rectangle (1,1);\node[left] at (0,0.5) {0}; \node[right] at (1,0.5) {0};\node[above] at (0.5,1) {0};\node[below] at (0.5,0) {0};\end{tikzpicture} & \begin{tikzpicture}[baseline=(current bounding box.center)]\draw[thin,fill=lightgray] (0,0) rectangle (1,1);\draw[red, thick] (0.5,0)--(0.5,0.5)--(1,0.5);\node[left] at (0,0.5) {0}; \node[right] at (1,0.5) {1};\node[above] at (0.5,1) {0};\node[below] at (0.5,0) {1};\end{tikzpicture} & \begin{tikzpicture}[baseline=(current bounding box.center)] \draw[thin,fill=lightgray] (0,0) rectangle (1,1); \draw[red, thick] (0,0.5)--(1,0.5); \node[left] at (0,0.5) {1}; \node[right] at (1,0.5) {1};\node[above] at (0.5,1) {0};\node[below] at (0.5,0) {0};\end{tikzpicture} & \begin{tikzpicture}[baseline=(current bounding box.center)] \draw[thin,fill=lightgray] (0,0) rectangle (1,1); \draw[red, thick] (0.5,0)--(0.5,1); \node[left] at (0,0.5) {0}; \node[right] at (1,0.5) {0};\node[above] at (0.5,1) {1};\node[below] at (0.5,0) {1};\end{tikzpicture} & \begin{tikzpicture}[baseline=(current bounding box.center)] \draw[thin,fill=lightgray] (0,0) rectangle (1,1); \draw[red, thick] (0,0.5)--(0.5,0.5)--(0.5,1); \node[left] at (0,0.5) {1}; \node[right] at (1,0.5) {0};\node[above] at (0.5,1) {1};\node[below] at (0.5,0) {0};\end{tikzpicture} \\
 $xt^{\alpha_i+\beta_i}$ & $t^{\beta_i}$ & $t^{\beta_i}$ & $xt^{\alpha_i+\beta_i}$ & $xt^{\alpha_i+\beta_i}$
\end{tabular}
}
\end{equation}
where
\[ 
\begin{aligned}
&\alpha_i = \textrm{\# colors greater than $i$ of Type 1}, \\
&\beta_i = \textrm{\# colors greater than $i$ of Type 4 or 5}.
\end{aligned}
\]

We will consider rows of gray vertices that extend infinitely to the right whose bottom boundary is given by a pair of partitions $\vec \mu$, top boundary given by a pair of partitions $\vec \lambda$, with no paths entering from the left and a path of each color exiting on the right. See Example \ref{ex:grayrow2}.

\begin{example}\label{ex:grayrow2}
Consider a gray row with top boundary $\vec \lambda=((1,1),(1,1))$ with $\ell= \ell(\lambda^{(1)}) = \ell(\lambda^{(2)})=2$ and bottom boundary $\vec \mu=((3,1,1),(2,1,0))$
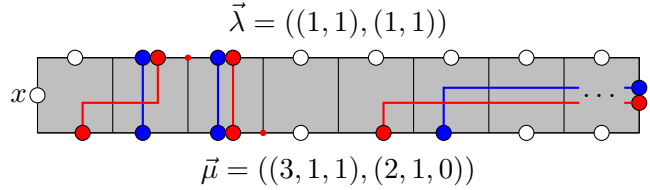
\[
\begin{tikzpicture}[baseline=(current bounding box.center)]
\draw[fill=lightgray] (0,0) rectangle (8,1);
\draw[step=1.0,black,thin, fill=lightgray] (0,0) grid (8,1); 
\node at (4,1.5) {$\vec \lambda=((1,1),(1,1))$};
\node at (4,-0.5) {$\vec \mu=((3,1,1),(2,1,0))$};
\node at (7.5,0.5) {$\dots$};
\draw[red,fill=red] (2,1) circle (1pt);
\draw[red,fill=red] (3,0) circle (1pt);

\draw[blue, thick] (1.4,0)--(1.4,1);
\draw[blue, thick] (2.4,0)--(2.4,1);
\draw[blue, thick] (5.4,0)--(5.4,0.6)--(7.2,0.6);
\draw[red, thick] (4.6,0)--(4.6,0.4)--(7.2,0.4);
\draw[blue, thick] (7.8,0.6)--(8,0.6);
\draw[red, thick] (7.8,0.4)--(8,0.4);
\draw[red, thick] (2.6,0)--(2.6,1);
\draw[red, thick] (0.6,0)--(0.6,0.4)--(1.6,0.4)--(1.6,1);

\draw[fill=white] (0,0.5) circle (0.1);
\draw[fill=blue] (8,0.6) circle (0.1);
\draw[fill=red] (8,0.4) circle (0.1);

\draw[fill=blue] (1.4,0) circle (0.1);
\draw[fill=blue] (2.4,0) circle (0.1);
\draw[fill=blue] (5.4,0) circle (0.1);
\draw[fill=red] (0.6,0) circle (0.1);
\draw[fill=red] (2.6,0) circle (0.1);
\draw[fill=red] (4.6,0) circle (0.1);

\draw[fill=blue] (1.4,1) circle (0.1);
\draw[fill=blue] (2.4,1) circle (0.1);
\draw[fill=red] (1.6,1) circle (0.1);
\draw[fill=red] (2.6,1) circle (0.1);

\draw[fill=white] (3.5,0) circle (0.1);
\draw[fill=white] (6.5,0) circle (0.1);
\draw[fill=white] (7.5,0) circle (0.1);

\draw[fill=white] (0.5,1) circle (0.1);
\draw[fill=white] (3.5,1) circle (0.1);
\draw[fill=white] (4.5,1) circle (0.1);
\draw[fill=white] (5.5,1) circle (0.1);
\draw[fill=white] (6.5,1) circle (0.1);
\draw[fill=white] (7.5,1) circle (0.1);

\node[left] at (0,0.5) {$x$};
\end{tikzpicture}
\]
where no paths enter on the left and a path of each color exits on the right. The row has weight $x^8t^3$. As for the 2-color white row, note that the factor of $x$ in the weight is the product of what one would get from each color treating them as a 1-color vertex model. In particular, 
\[
x^7 = x^{5} \cdot x^{3} = x^{|\mu^{(1)}/\lambda^{(1)}|+\ell} \cdot x^{|\mu^{(2)}/\lambda^{(2)}|+\ell}.
\]
\end{example}
\noindent Note that we must have $\mu^{(i)}\succeq \lambda^{(i)}$, $i=1,2$ for there to be a valid configuration.

As the gray weights are related to white weights by a change of variable, here is still a YBE between the multicolored gray and white vertices. As this is the YBE we will need, we state it precisely below.
\begin{proposition}
For any choice of  $I_1,I_2,I_3,J_1,J_2,J_3\in\{0,1\}^2$ with $1$ indicating a path present at that boundary and $0$ indicating no path. We have the following equality of of partition functions
\[
\begin{tikzpicture}[baseline=(current bounding box.center)] 
 \draw (-1,0.5) -- (0,1.5); \draw (-1,1.5) -- (0,0.5); 
 \node[right] at (-0.5,1) {$yx$};
 \draw[step=1.0,black,thin] (0,0) grid (1,2); 
 \draw[fill=lightgray] (0,0) rectangle (1,1);
 \node[left] at (-1,1.5) {$I_1$}; \node[left] at (-1,0.5) {$I_2$}; \node[below] at (0.5,0) {$I_3$};
 \node[right] at (1,1.5) {$J_2$}; \node[right] at (1,0.5) {$J_1$}; \node[above] at (0.5,2) {$J_3$};
 \node at (0.5,0.5) {$x$}; \node at (0.5,1.5) {$y$};
 \end{tikzpicture} 
 =
 \begin{tikzpicture}[baseline=(current bounding box.center)] 
 \draw (2,0.5) -- (1,1.5); \draw (2,1.5) -- (1,0.5); 
  \node[right] at (1.55,1) {$yx$};
 \draw[step=1.0,black,thin] (0,0) grid (1,2); 
 \draw[fill=lightgray] (0,1) rectangle (1,2);
 \node[left] at (0,1.5) {$I_1$}; \node[left] at (0,0.5) {$I_2$}; \node[below] at (0.5,0) {$I_3$};
 \node[right] at (2,0.5) {$J_1$}; \node[right] at (2,1.5) {$J_2$}; \node[above] at (0.5,2) {$J_3$};
 \node at (0.5,0.5) {$y$}; \node at (0.5,1.5) {$x$};
 \end{tikzpicture}.
 \] 
 where the weights for the white vertices are those given in Figure \ref{fig:whitevertices2} and the weights of the gray vertices are those given in Figure \ref{fig:grayvertices2}, and the weights for the cross vertices are given in Equation \eqref{eq:crossweight2} with $z=yx$.
\end{proposition}
\begin{proof}
See Proposition 6.3 in \cite{LLT}.
\end{proof}

\subsection{Interacting RPPs}
\label{sec:overlappingRPPs}

We can now put together our white rows and gray rows in the same fashion as we did for the 1-color case:
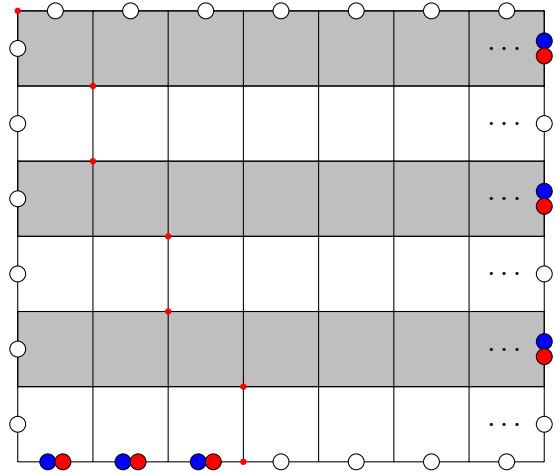
\[
\begin{tikzpicture}[baseline=(current bounding box.center)]
\draw[fill=lightgray] (1,1) rectangle (8,2);
\draw[fill=lightgray] (1,3) rectangle (8,4);
\draw[fill=lightgray] (1,5) rectangle (8,6);
\draw (1,0) grid (8,6);

\draw[fill=white] (1,0.5) circle (3pt);
\draw[fill=white] (1,1.5) circle (3pt);
\draw[fill=white] (1,2.5) circle (3pt);
\draw[fill=white] (1,3.5) circle (3pt);
\draw[fill=white] (1,4.5) circle (3pt);
\draw[fill=white] (1,5.5) circle (3pt);

\draw[fill=white] (8,0.5) circle (3pt);
\draw[fill=blue] (8,1.6) circle (3pt);
\draw[fill=red] (8,1.4) circle (3pt);
\draw[fill=white] (8,2.5) circle (3pt);
\draw[fill=blue] (8,3.6) circle (3pt);
\draw[fill=red] (8,3.4) circle (3pt);
\draw[fill=white] (8,4.5) circle (3pt);
\draw[fill=blue] (8,5.6) circle (3pt);
\draw[fill=red] (8,5.4) circle (3pt);

\draw[fill=blue] (1.4,0) circle (3pt);
\draw[fill=red] (1.6,0) circle (3pt);
\draw[fill=blue] (2.4,0) circle (3pt);
\draw[fill=red] (2.6,0) circle (3pt);
\draw[fill=blue] (3.4,0) circle (3pt);
\draw[fill=red] (3.6,0) circle (3pt);
\draw[fill=white] (4.5,0) circle (3pt);
\draw[fill=white] (5.5,0) circle (3pt);
\draw[fill=white] (6.5,0) circle (3pt);
\draw[fill=white] (7.5,0) circle (3pt);

\draw[fill=white] (1.5,6) circle (3pt);
\draw[fill=white] (2.5,6) circle (3pt);
\draw[fill=white] (3.5,6) circle (3pt);
\draw[fill=white] (4.5,6) circle (3pt);
\draw[fill=white] (5.5,6) circle (3pt);
\draw[fill=white] (6.5,6) circle (3pt);
\draw[fill=white] (7.5,6) circle (3pt);

\draw[red,fill=red] (4,0) circle (1pt);
\draw[red,fill=red] (4,1) circle (1pt);
\draw[red,fill=red] (3,2) circle (1pt);
\draw[red,fill=red] (3,3) circle (1pt);
\draw[red,fill=red] (2,4) circle (1pt);
\draw[red,fill=red] (2,5) circle (1pt);
\draw[red,fill=red] (1,6) circle (1pt);

\node at (7.5,0.5) {$\dots$};
\node at (7.5,1.5) {$\dots$};
\node at (7.5,2.5) {$\dots$};
\node at (7.5,3.5) {$\dots$};
\node at (7.5,4.5) {$\dots$};
\node at (7.5,5.5) {$\dots$};
\end{tikzpicture}
 \]
 where the order of the white and gray rows is determined by the shape of the RPPs as is Section \ref{sec:vertexRPP}. Note that for each color we have a bijection between possible path configurations and RPPs of shape $\lambda$. Thus we immediately get a bijection between configurations of the 2-colored vertex model and pairs of RPPs $(\Lambda,\Lambda')$ of the same shape. Here $\Lambda$ corresponds to the blue paths and $\Lambda'$ corresponds to the red paths. We begin with an example of the construction.
 
\begin{example} \label{ex:2colorRPP}
Below is an example of a pair of RPPs of the same shape (top) followed by their corresponding lozenge tiling (middle) and vertex model configuration (bottom).
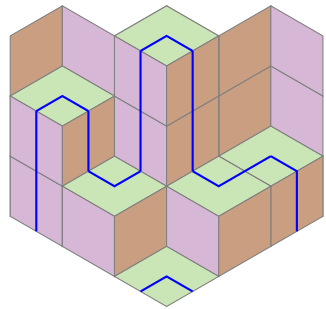
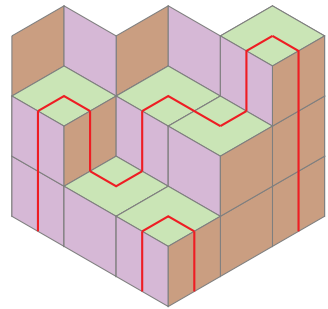
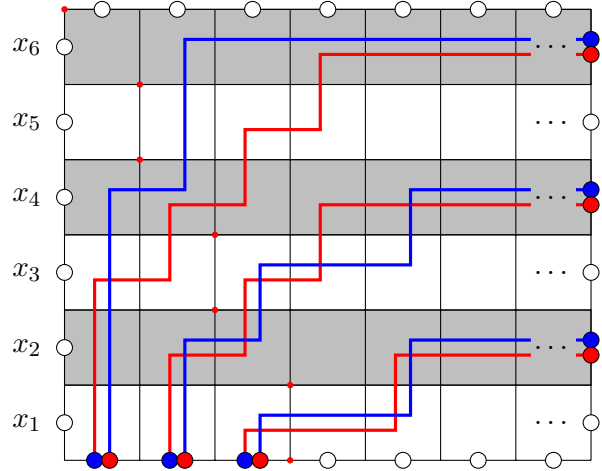
\begin{center}
\[
{
\begin{ytableau}
2 \\
1 & 3 \\
0 & 1 & 1
\end{ytableau}\hspace{90pt}
\begin{ytableau}
2 \\
1 & 2 \\
1 & 2 & 3
\end{ytableau}}
\]
\end{center}

\begin{center}
\[
\begin{tabular}{cc}
\begin{tikzpicture}[scale=0.8]
\reverseplanepartition{3}{{2},{1,3},{0,1,1}}
\fill[white,fill opacity=0.5] (-0.9,-2.1) rectangle (4.4,3.1);
\draw[blue,thick,shift={(0,-1)},shift={(150:0.5)}] (0,0)--(0,2);
\draw[blue,thick,shift={(0,1.5)}] (-30:0.5)--(0,0)--(-150:0.5);
\draw[blue,thick,shift={(30:0.5)}] (0,1)--(0,0);
\draw[blue,thick,shift={(30:1)},shift={(0,-0.5)}] (150:0.5)--(0,0)--(30:0.5);
\draw[blue,thick,shift={(30:1.5)}] (0,-0.5)--(0,1.5);
\draw[blue,thick,shift={(30:2)},shift={(0,1.5)}] (-30:0.5)--(0,0)--(-150:0.5);
\draw[blue,thick,shift={(30:2.5)}] (0,1)--(0,-1);
\draw[blue,thick,shift={(30:3)},shift={(0,-1.5)}] (150:0.5)--(0,0)--(30:1);
\draw[blue,thick,shift={(30:2.5)},shift={(-30:2)}] (150:0.5)--(0,0)--(0,-1);
\draw[blue,thick,shift={(-30:2)},shift={(0,-0.5)}] (-30:0.5)--(0,0)--(-150:0.5);
\end{tikzpicture} & 
\begin{tikzpicture}[scale=0.8]
\reverseplanepartition{3}{{2},{1,2},{1,2,3}}
\fill[white, fill opacity=0.5] (-0.9,-2.1) rectangle (4.4,3.1);
\draw[Red,thick,shift={(0,-1)},shift={(150:0.5)}] (0,0)--(0,2);
\draw[Red,thick,shift={(0,1.5)}] (-30:0.5)--(0,0)--(-150:0.5);
\draw[Red,thick,shift={(30:0.5)}] (0,1)--(0,0);
\draw[Red,thick,shift={(30:1)},shift={(0,-0.5)}] (150:0.5)--(0,0)--(30:0.5);
\draw[Red,thick,shift={(30:1.5)}] (0,-0.5)--(0,0.5);
\draw[Red,thick,shift={(30:2)},shift={(0,0.5)}] (-30:1)--(0,0)--(-150:0.5);
\draw[Red,thick,shift={(30:3)},shift={(0,-0.5)}] (0,0)--(30:0.5);
\draw[Red,thick,shift={(30:3.5)},shift={(0,0.5)}] (0,-1)--(0,0)--(30:0.5);
\draw[Red,thick,shift={(30:4.5)}] (150:0.5)--(0,0)--(0,-3);
\draw[Red,thick,shift={(-30:2)},shift={(0,0.5)}] (-30:0.5)--(0,0)--(-150:0.5);
\draw[Red,thick,shift={(-30:1.5)}] (0,0)--(0,-1);
\draw[Red,thick,shift={(-30:2.5)}] (0,0.5)--(0,-0.5);
\end{tikzpicture}
\end{tabular}
\]
\end{center}

\begin{center}
\[
\begin{tikzpicture}[baseline=(current bounding box.center)]
\draw[fill=lightgray] (1,1) rectangle (8,2);
\draw[fill=lightgray] (1,3) rectangle (8,4);
\draw[fill=lightgray] (1,5) rectangle (8,6);
\draw (1,0) grid (8,6);

\draw[red,very thick] (1.4,0)--(1.4,2.4)--(2.4,2.4)--(2.4,3.4)--(3.4,3.4)--(3.4,4.4)--(4.4,4.4)--(4.4,5.4)--(7.2,5.4);
\draw[red,very thick] (7.8,5.4)--(8,5.4);
\draw[red,very thick] (2.4,0)--(2.4,1.4)--(3.4,1.4)--(3.4,2.4)--(4.4,2.4)--(4.4,3.4)--(7.2,3.4);
\draw[red,very thick] (7.8,3.4)--(8,3.4);
\draw[red,very thick] (3.4,0)--(3.4,0.4)--(5.4,0.4)--(5.4,1.4)--(7.2,1.4);
\draw[red,very thick] (7.8,1.4)--(8,1.4);
\draw[blue,very thick] (1.6,0)--(1.6,3.6)--(2.6,3.6)--(2.6,5.6)--(7.2,5.6);
\draw[blue,very thick] (7.8,5.6)--(8,5.6);
\draw[blue,very thick] (2.6,0)--(2.6,1.6)--(3.6,1.6)--(3.6,2.6)--(5.6,2.6)--(5.6,3.6)--(7.2,3.6);
\draw[blue,very thick] (7.8,3.6)--(8,3.6);
\draw[blue,very thick] (3.6,0)--(3.6,0.6)--(5.6,0.6)--(5.6,1.6)--(7.2,1.6);
\draw[blue,very thick] (7.8,1.6)--(8,1.6);

\draw[fill=white] (1,0.5) circle (3pt);
\draw[fill=white] (1,1.5) circle (3pt);
\draw[fill=white] (1,2.5) circle (3pt);
\draw[fill=white] (1,3.5) circle (3pt);
\draw[fill=white] (1,4.5) circle (3pt);
\draw[fill=white] (1,5.5) circle (3pt);

\draw[fill=white] (8,0.5) circle (3pt);
\draw[fill=blue] (8,1.6) circle (3pt);
\draw[fill=red] (8,1.4) circle (3pt);
\draw[fill=white] (8,2.5) circle (3pt);
\draw[fill=blue] (8,3.6) circle (3pt);
\draw[fill=red] (8,3.4) circle (3pt);
\draw[fill=white] (8,4.5) circle (3pt);
\draw[fill=blue] (8,5.6) circle (3pt);
\draw[fill=red] (8,5.4) circle (3pt);

\draw[fill=blue] (1.4,0) circle (3pt);
\draw[fill=red] (1.6,0) circle (3pt);
\draw[fill=blue] (2.4,0) circle (3pt);
\draw[fill=red] (2.6,0) circle (3pt);
\draw[fill=blue] (3.4,0) circle (3pt);
\draw[fill=red] (3.6,0) circle (3pt);
\draw[fill=white] (4.5,0) circle (3pt);
\draw[fill=white] (5.5,0) circle (3pt);
\draw[fill=white] (6.5,0) circle (3pt);
\draw[fill=white] (7.5,0) circle (3pt);

\draw[fill=white] (1.5,6) circle (3pt);
\draw[fill=white] (2.5,6) circle (3pt);
\draw[fill=white] (3.5,6) circle (3pt);
\draw[fill=white] (4.5,6) circle (3pt);
\draw[fill=white] (5.5,6) circle (3pt);
\draw[fill=white] (6.5,6) circle (3pt);
\draw[fill=white] (7.5,6) circle (3pt);

\draw[red,fill=red] (4,0) circle (1pt);
\draw[red,fill=red] (4,1) circle (1pt);
\draw[red,fill=red] (3,2) circle (1pt);
\draw[red,fill=red] (3,3) circle (1pt);
\draw[red,fill=red] (2,4) circle (1pt);
\draw[red,fill=red] (2,5) circle (1pt);
\draw[red,fill=red] (1,6) circle (1pt);

\draw node at (0.5, 0.5) {$x_1$};
\draw node at (0.5, 1.5) {$x_2$};
\draw node at (0.5, 2.5) {$x_3$};
\draw node at (0.5, 3.5) {$x_4$};
\draw node at (0.5, 4.5) {$x_5$};
\draw node at (0.5, 5.5) {$x_6$};

\node at (7.5,0.5) {$\dots$};
\node at (7.5,1.5) {$\dots$};
\node at (7.5,2.5) {$\dots$};
\node at (7.5,3.5) {$\dots$};
\node at (7.5,4.5) {$\dots$};
\node at (7.5,5.5) {$\dots$};
\end{tikzpicture}
 \]
\end{center}
The fact that we have a bijection between pairs of RPPs of the same shape and 2-color vertex model configurations follows from the 1-color case applied to each color.
\end{example}

\subsubsection{Understanding the weights}
We are left to understand how the vertex model weights translate to the coupled RPPs. The contributions from the $x_i$ behave just as in the 1-color case.
\begin{lemma}
Consider the bijection between pairs of RPPs $(\Lambda,\Lambda')$ of shape $\lambda$ of vertex model configurations where $\Lambda$ corresponds to the blue paths and $\Lambda'$ corresponds to the red paths.
If one chooses the weight parameters of the vertex model to be $x_i=q^{\pm i}$ where we take $+$ for a gray row and $-$ for a white row then the weight of a vertex model configuration is proportional to $q^{|\Lambda|+|\Lambda'|}$.
\end{lemma}
\begin{proof}
This follows from the 1-color case, Theorem \ref{thm:bij1color}, as the weight coming from the $x_i$'s in the 2-color case is the product of the weight coming from the $x_i$'s for blue with those from red where we view each color individually as a 1-color vertex model. Now, however, the constant of proportionality will depend on the interaction strength $t$.
\end{proof}

Understanding the weight coming from the interactions is more difficult.  We break the contribution of the interaction strength $t$ to the weight of the vertex model into two parts: the \emph{trivial contribution} which appears the same in every configuration and does not depend on the resulting RPPs, and the \emph{non-trivial contribution} which does depend on the specifics of the configurations. 

The following lemma gives the trivial contribution.
\begin{lemma}
Suppose we have $\ell$ paths of each color in our vertex model. Then the trivial contribution of the interaction strength $t$ to the weight is given by $t^{\frac{\ell(\ell-1)}{2}}$.
\end{lemma}
\begin{proof}
    At the start of our vertex diagram we will have $\ell$ many paths of each color. At each gray row, one path of each will be exit to the right of the row. This means that we will have $\ell-1$ many paths of each color exiting through the top of the row. By the way the weights are defined in Eqn. \eqref{eqn:gray_weights}, anytime a red path exits through the top of a gray row we will get a factor of a $t$ in the weight, regardless of the behavior of the blue paths. We see that the red paths exiting the top of the gray rows will contribute $(\ell-1)+ (\ell-2) + \ldots +1 = \frac{\ell(\ell-1)}{2}$ powers of $t$ to the weight.
    
    Note that any other powers of $t$ in the weight will depend on the specifics of the local configuration.
\end{proof}

To handle the non-trivial contributions of the interaction we first make some definitions.
 \begin{definition} Let $(\Lambda,\Lambda')$ be a pair of RPPs of the same shape. Suppose we superimpose the corresponding lozenge tilings one on top of the other. We say the a pair of lozenges form a \emph{coupled pair} if it takes the form
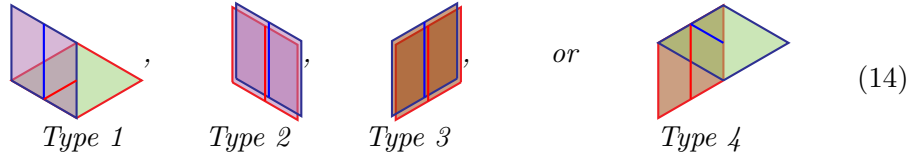
\begin{equation} \label{eq:coupledpair}
\begin{tabular}{ccccc}
\begin{tikzpicture}[fill opacity=0.5, baseline=(current bounding box.center)]
    \fill[YellowGreen,draw=Red,thick] (0,0)--(150:1)--(0,1)--(30:1)--(0,0);
    \fill[Orchid,draw=Blue,thick,shift={(150:1)}] (0,0)--(0,1)--(30:1)--(-30:1)--(0,0);
    \draw[blue,thick] (-0.435,0.25)--(-0.435,1.25);
    \draw[red,thick] (-0.435,0.25)--(0,0.5);
\end{tikzpicture},&\hspace{0.5cm}
\begin{tikzpicture}[fill opacity=0.5, baseline=(current bounding box.center)]
    \fill[Orchid,draw=Red,thick] (0,0)--(0,1)--(30:1)--(-30:1)--(0,0);
    \fill[Orchid,draw=Blue,thick,shift={(0.05,0.05)}] (0,0)--(0,1)--(30:1)--(-30:1)--(0,0);
   \draw[blue,thick,shift={(0.05,0.05)}] (0.435,-0.25)--(0.435,0.75);
   \draw[red,thick] (0.435,-0.25)--(0.435,0.75);
\end{tikzpicture},&\hspace{0.5cm}
\begin{tikzpicture}[fill opacity=0.5, baseline=(current bounding box.center)]
    \fill[RawSienna,draw=Red,thick] (0,0)--(0,1)--(150:1)--(-150:1)--(0,0);
    \fill[RawSienna,draw=Blue,thick,shift={(-0.05,0.05)}] (0,0)--(0,1)--(150:1)--(-150:1)--(0,0);
    \draw[blue,thick,shift={(-0.05,0.05)}] (-0.435,-0.25)--(-0.435,0.75);
    \draw[red,thick] (-0.435,-0.25)--(-0.435,0.75);
\end{tikzpicture}, &\hspace{0.5cm} \text{or} & \hspace{0.5cm}
\begin{tikzpicture}[fill opacity=0.5, baseline=(current bounding box.center)]
    \fill[RawSienna,draw=Red,thick] (0,0)--(0,1)--(150:1)--(-150:1)--(0,0);
    \fill[YellowGreen,draw=Blue,thick] (0,0)--(150:1)--(0,1)--(30:1)--(0,0);
    \draw[red,thick] (-0.435,-0.25)--(-0.435,0.75);
    \draw[blue,thick] (-0.435,0.75)--(0,0.5);
\end{tikzpicture}
\\
Type 1 & Type 2 & Type 3 & & Type 4
\end{tabular}
\end{equation}
when we superimpose them. We define $g(\Lambda,\Lambda')$ to be the number of coupled pairs between the pair of RPPs. 
\end{definition}\label{def:coupledLozenges}

\begin{example}
Continuing with the RPP in Example \ref{ex:2colorRPP}, superimposing the tilings gives
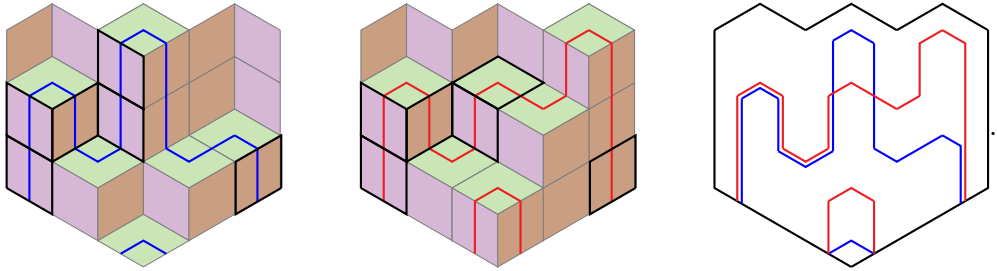
\[
\begin{tikzpicture}[scale=0.7, baseline=(current bounding box.center)]\reverseplanepartition{3}{{2},{1,3},{0,1,1}}
\fill[white,fill opacity=0.5] (-0.9,-2.1) rectangle (4.4,3.1);
\draw[blue,thick,shift={(0,-1)},shift={(150:0.5)}] (0,0)--(0,2);
\draw[blue,thick,shift={(0,1.5)}] (-30:0.5)--(0,0)--(-150:0.5);
\draw[blue,thick,shift={(30:0.5)}] (0,1)--(0,0);
\draw[blue,thick,shift={(30:1)},shift={(0,-0.5)}] (150:0.5)--(0,0)--(30:0.5);
\draw[blue,thick,shift={(30:1.5)}] (0,-0.5)--(0,1.5);
\draw[blue,thick,shift={(30:2)},shift={(0,1.5)}] (-30:0.5)--(0,0)--(-150:0.5);
\draw[blue,thick,shift={(30:2.5)}] (0,1)--(0,-1);
\draw[blue,thick,shift={(30:3)},shift={(0,-1.5)}] (150:0.5)--(0,0)--(30:1);
\draw[blue,thick,shift={(30:2.5)},shift={(-30:2)}] (150:0.5)--(0,0)--(0,-1);
\draw[blue,thick,shift={(-30:2)},shift={(0,-0.5)}] (-30:0.5)--(0,0)--(-150:0.5);
\draw[black,thick,shift={(150:1)}] (0,0)--(0,1)--(30:1)--(-30:1)--(0,0);
\draw[black,thick,shift={(-150:1)}] (0,0)--(0,1)--(30:1)--(-30:1)--(0,0);
\draw[black,thick,shift={(-150:-1)}] (0,0)--(0,1)--(30:1)--(-30:1)--(0,0);
\draw[black,thick,shift={(0.87,1.5)}] (0,0)--(0,1)--(30:1)--(-30:1)--(0,0);
\draw[black,thick,shift={(0.87,0.5)}] (0,0)--(0,1)--(150:1)--(-150:1)--(0,0);
\draw[black,thick,shift={(4.35,-0.5)}] (0,0)--(0,1)--(150:1)--(-150:1)--(0,0);
\end{tikzpicture}
\hspace{1cm}
\begin{tikzpicture}[scale=0.7, baseline=(current bounding box.center)]\reverseplanepartition{3}{{2},{1,2},{1,2,3}}
\fill[white, fill opacity=0.5] (-0.9,-2.1) rectangle (4.4,3.1);
\draw[Red,thick,shift={(0,-1)},shift={(150:0.5)}] (0,0)--(0,2);
\draw[Red,thick,shift={(0,1.5)}] (-30:0.5)--(0,0)--(-150:0.5);
\draw[Red,thick,shift={(30:0.5)}] (0,1)--(0,0);
\draw[Red,thick,shift={(30:1)},shift={(0,-0.5)}] (150:0.5)--(0,0)--(30:0.5);
\draw[Red,thick,shift={(30:1.5)}] (0,-0.5)--(0,0.5);
\draw[Red,thick,shift={(30:2)},shift={(0,0.5)}] (-30:1)--(0,0)--(-150:0.5);
\draw[Red,thick,shift={(30:3)},shift={(0,-0.5)}] (0,0)--(30:0.5);
\draw[Red,thick,shift={(30:3.5)},shift={(0,0.5)}] (0,-1)--(0,0)--(30:0.5);
\draw[Red,thick,shift={(30:4.5)}] (150:0.5)--(0,0)--(0,-3);
\draw[Red,thick,shift={(-30:2)},shift={(0,0.5)}] (-30:0.5)--(0,0)--(-150:0.5);
\draw[Red,thick,shift={(-30:1.5)}] (0,0)--(0,-1);
\draw[Red,thick,shift={(-30:2.5)}] (0,0.5)--(0,-0.5);
\draw[black,thick,shift={(150:1)}] (0,0)--(0,1)--(30:1)--(-30:1)--(0,0);
\draw[black,thick,shift={(-150:1)}] (0,0)--(0,1)--(30:1)--(-30:1)--(0,0);
\draw[black,thick,shift={(-150:-1)}] (0,0)--(0,1)--(30:1)--(-30:1)--(0,0);
\draw[black,thick,shift={(0.87,0.5)}] (0,0)--(0,1)--(150:1)--(-150:1)--(0,0);
\draw[black,thick,shift={(4.35,-0.5)}] (0,0)--(0,1)--(150:1)--(-150:1)--(0,0);
\draw[black,thick,shift={(1.74,1)}] (0,0)--(150:1)--(0,1)--(30:1)--(0,0);
\end{tikzpicture}
\hspace{1cm}
\begin{tikzpicture}[scale=0.7, baseline=(current bounding box.center)]
\draw[black,thick,shift={(-150:1)}] (0,3)--(0,0)--(-30:3);
\draw[black,thick,shift={(0,3)}] (-150:1)--(0,0)--(-30:1);
\draw[black,thick,shift={(30:2)},shift={(0,2)}] (-150:1)--(0,0)--(-30:1);
\draw[black,thick,shift={(30:4)},shift={(0,1)}] (-150:1)--(0,0)--(-30:1);
\draw[black,thick,shift={(-30:3)},shift={(30:2)}] (-150:3)--(0,0)--(0,3);
\draw[blue,thick,shift={(0,-1)},shift={(150:0.4)}] (0,0)--(0,2);
\draw[blue,thick,shift={(0,1.4)}] (-30:0.4)--(0,0)--(-150:0.4);
\draw[blue,thick,shift={(30:0.4)}] (0,1)--(0,0);
\draw[blue,thick,shift={(30:1)},shift={(0,-0.6)}] (150:0.6)--(0,0)--(30:0.6);
\draw[blue,thick,shift={(30:1.6)}] (0,-0.6)--(0,1.5);
\draw[blue,thick,shift={(30:2)},shift={(0,1.5)}] (-30:0.5)--(0,0)--(-150:0.4);
\draw[blue,thick,shift={(30:2.5)}] (0,1)--(0,-1);
\draw[blue,thick,shift={(30:3)},shift={(0,-1.5)}] (150:0.5)--(0,0)--(30:1);
\draw[blue,thick,shift={(30:2.5)},shift={(-30:1.9)}] (150:0.4)--(0,0)--(0,-1.1);
\draw[blue,thick,shift={(-30:2)},shift={(0,-0.5)}] (-30:0.5)--(0,0)--(-150:0.5);
\draw[Red,thick,shift={(0,-1)},shift={(150:0.5)}] (0,0)--(0,2);
\draw[Red,thick,shift={(0,1.5)}] (-30:0.5)--(0,0)--(-150:0.5);
\draw[Red,thick,shift={(30:0.5)}] (0,1)--(0,0);
\draw[Red,thick,shift={(30:1)},shift={(0,-0.5)}] (150:0.5)--(0,0)--(30:0.5);
\draw[Red,thick,shift={(30:1.5)}] (0,-0.5)--(0,0.5);
\draw[Red,thick,shift={(30:2)},shift={(0,0.5)}] (-30:1)--(0,0)--(-150:0.5);
\draw[Red,thick,shift={(30:3)},shift={(0,-0.5)}] (0,0)--(30:0.5);
\draw[Red,thick,shift={(30:3.5)},shift={(0,0.5)}] (0,-1)--(0,0)--(30:0.5);
\draw[Red,thick,shift={(30:4.5)}] (150:0.5)--(0,0)--(0,-3);
\draw[Red,thick,shift={(-30:2)},shift={(0,0.5)}] (-30:0.5)--(0,0)--(-150:0.5);
\draw[Red,thick,shift={(-30:1.5)}] (0,0)--(0,-1);
\draw[Red,thick,shift={(-30:2.5)}] (0,0.5)--(0,-0.5);
\end{tikzpicture}.
\]
The weight of the tiling is $q^8 \cdot q^{11} \cdot t^6$. The lozenges involved in the 6 coupled pairs are outlined in bold.
\end{example}

\begin{theorem}\label{thm:bij2color}
Let $(\Lambda,\Lambda')$ be a pair of RPPs of shape $\lambda$ and $\ell(\lambda)$ be the number of non-zero parts of $\lambda$. Let $\mathcal{C}$ be the corresponding path configuration under the bijection to the 2-color vertex model. Then if one chooses the weight parameters of the vertex model to be $x_i=q^{\pm i}$ where we take $+$ for a gray row and $-$ for a white row, we have
\[
q^{|\Lambda|+|\Lambda'|}t^{g(\Lambda,\Lambda')} = w(\mathcal{C})\cdot A_\lambda(q;t),\quad A_\lambda(q;t)= A_{\lambda}(q)^2 t^{\frac{-\ell(\ell-1)}{2}} 
\]
where $ w(\mathcal{C})$ is the weight of the vertex model configuration and $A_\lambda(q;t)$ is completely determined by the shape $\lambda$ and is independent of the specific configuration. 
\end{theorem}

\begin{proof}
Note that the factors of $q$ for each color behave just as in the 1-color case and so we only need to check that the factors of $t$ agree.

The factors coming from the $t^{\beta_i}$ in the gray weights \eqref{eqn:gray_weights} correspond exactly to the trivial factors of $t$ in the vertex weights which are canceled by the $t^{-\ell(\ell-1)/2}$ factor in $A_\lambda(q;t)$. The factors of $t^{\alpha_i}$ in \eqref{eqn:gray_weights} can be described by saying that we get a factor of $t$ if a blue path does not exit a vertex to the right and a red path is not present in the vertex. The blue path not exiting the vertex to the right means it is either not present or it exits at the top. From the dictionary in \eqref{eq:graydict}, both blue and red not being present in the vertex correspond to the third type of coupled pair in \eqref{eq:coupledpair} while the blue path exiting at the top and the red path not being present corresponds to the last type of coupled pair in \eqref{eq:coupledpair}.

Recall, that for a white row of the vertex model there is a factor of $t$ if a blue path exits right and a red path is present. Note that if a red path is present it must either exit the vertex to the right or up. Using the dictionary in \eqref{eq:whitedict}, we see that a blue path exiting right and a red path exiting up in a vertex corresponds to the first type of coupled pair in \eqref{eq:coupledpair}, while a blue path exiting right and a red path exiting right in a vertex corresponds to the second type of coupled pair in \eqref{eq:coupledpair}.
\end{proof}

\subsection{The Generating Function}
Before we prove our main theorem we need the multicolored analogue of Lemma \ref{lem:whitegrayswap}.
\begin{lemma} \label{lem:whitegrayswap2}
Choose $x,y$ such that $|xy|,|xyt|<1$. Then the white and gray rows satisfy the commutation relation
\begin{equation}
\begin{aligned}
&\begin{tikzpicture}[baseline=(current bounding box.center)]
\draw[fill=lightgray] (0,0) rectangle (7,1);
\draw (0,0) grid (7,2);
\node at (6.5,0.5) {$\dots$};
\node at (6.5,1.5) {$\dots$};
\node at (0.5,0.5) {$x$};
\node at (1.5,0.5) {$x$};
\node at (2.5,0.5) {$x$};
\node at (3.5,0.5) {$x$};
\node at (4.5,0.5) {$x$};
\node at (5.5,0.5) {$x$};
\node at (0.5,1.5) {$y$};
\node at (1.5,1.5) {$y$};
\node at (2.5,1.5) {$y$};
\node at (3.5,1.5) {$y$};
\node at (4.5,1.5) {$y$};
\node at (5.5,1.5) {$y$};

\node[below] at (3.5,0) {$\vec\mu$};
\node[above] at (3.5,2) {$\vec\lambda$};

\draw[fill=white] (0,0.5) circle (3pt);
\draw[fill=white] (0,1.5) circle (3pt);
\draw[fill=black] (7,0.5) circle (3pt);
\draw[fill=white] (7,1.5) circle (3pt);
\end{tikzpicture}
\\
&\qquad \qquad = (1-xy)(1-xyt) 
\begin{tikzpicture}[baseline=(current bounding box.center)]
\draw[fill=lightgray] (0,1) rectangle (7,2);
\draw (0,0) grid (7,2);
\node at (6.5,0.5) {$\dots$};
\node at (6.5,1.5) {$\dots$};
\node at (0.5,0.5) {$y$};
\node at (1.5,0.5) {$y$};
\node at (2.5,0.5) {$y$};
\node at (3.5,0.5) {$y$};
\node at (4.5,0.5) {$y$};
\node at (5.5,0.5) {$y$};
\node at (0.5,1.5) {$x$};
\node at (1.5,1.5) {$x$};
\node at (2.5,1.5) {$x$};
\node at (3.5,1.5) {$x$};
\node at (4.5,1.5) {$x$};
\node at (5.5,1.5) {$x$};

\node[below] at (3.5,0) {$\vec\mu$};
\node[above] at (3.5,2) {$\vec\lambda$};

\draw[fill=white] (0,0.5) circle (3pt);
\draw[fill=white] (0,1.5) circle (3pt);
\draw[fill=white] (7,0.5) circle (3pt);
\draw[fill=black] (7,1.5) circle (3pt);
\end{tikzpicture}
\end{aligned}
\end{equation}
where $\vec\lambda$ and $\vec\mu$ are pairs of partitions corresponding to the top and bottom boundary conditions of the rows.
\end{lemma}
\begin{proof}
The proof follows analogously to the 1-color case. This result can also be found in \cite{aggarwal2023colored,LLT}.
\end{proof}

\begin{theorem}\label{mainthm}
For any fixed $\lambda$, the generating function for volume weighted pairs of interacting RPPs is given by
\begin{equation} \label{eq:genfun2}
\sum_{\Lambda, \Lambda^\prime \in RPP(\lambda)} q^{|\Lambda| + |\Lambda'|}t^{g(\Lambda,\Lambda^\prime)} = \prod_{c \in \lambda} \frac{1}{(1-q^{h_{\lambda}(c)})(1-q^{h_{\lambda}(c)}t)}.
\end{equation}
\end{theorem}
\begin{proof}
The proof follows exactly as in the 1-color case where the two factors in the denominators of \eqref{eq:genfun2} come from the two factor in the commutation relation, Lemma \ref{lem:whitegrayswap2}.
\end{proof}

\section{The $t \rightarrow 0$ Limit}\label{sec:tto0}
We can now more easily prove the following theorem.
\begin{theorem}\label{thm:t0}
     For any integer $n \ge 0$, the number of pairs of RPPs $\Lambda, \Lambda' \in RPP(\lambda)$ satisfying $g(\Lambda,\Lambda') = 0$ and having total volume $|\Lambda| + |\Lambda'|=n$ is equal to the number of RPPs of shape $\lambda$ and volume $n$.
\end{theorem}
\begin{proof}
    We start with the result from Theorem \ref{mainthm}:
    \[
    \sum_{\Lambda, \Lambda' \in RPP(\lambda)} q^{|\Lambda| + |\Lambda'|}t^{g(\Lambda,\Lambda')} = \prod_{c \in \lambda} \frac{1}{(1-q^{h_{\lambda}(c)})(1-q^{h_{\lambda}(c)}t)}
    \]
    Now we take the limit as $t$ approaches 0 on both sides. On the LHS, all terms containing a positive power of $t$ vanish. On the RHS, the second factor in the denominator approaches 1. Thus, we obtain
    \[
    \sum_{\Lambda, \Lambda' \in RPP(\lambda), g(\Lambda, \Lambda') = 0} q^{|\Lambda| + |\Lambda'|} = \prod_{c \in \lambda} \frac{1}{1-q^{h_{\lambda}(c)}}
    \]
    The LHS is now the generating function for pairs of reverse plane partitions of shape $\lambda$ satisfying $g(\Lambda,\Lambda') = 0$. The RHS is now the generating function for reverse plane partitions of shape $\lambda$. Matching coefficients of $q^n$ gives the theorem.
\end{proof}
The rest of this section is devoted to giving a combinatorial proof of the above theorem.

\subsection{Bijective Proof}
The bijection we construct will be based on sliding the paths coming from the border strip decomposition of the RPP. A similar bijection for coupled domino tilings of the Aztec diamond was given in \cite{CGK}. To begin we prove a lemma describing the restriction having $t=0$ places on the paths. It will be useful for us to order the paths on the tilings. For each color, we will say the upper most past is first, followed by the second highest path, and so on. Equivalently, we order the border strip in the border strip decomposition from outermost to innermost. Furthermore, we will say that the paths start on the left and end on the right.

\begin{lemma} \label{lem:t0contraints1}
When the interaction strength $t=0$:
\begin{itemize}
\item The $i^{th}$ blue path must be weakly below the $i^{th}$ red path.
\item The $i^{th}$ blue path must be strictly above the $(i+1)^{th}$ red path. 
\end{itemize}
\end{lemma}
\begin{proof}
Consider the $i^{th}$ path of each color. Suppose after some step the paths are at the some point. The next step of both paths cannot be vertical otherwise there is a coupled pair of type two or three from \eqref{eq:coupledpair}. If the next step put the blue path above the red path it would introduce a coupled pair of either type one or four. We see that  $i^{th}$ blue path must be weakly below the $i^{th}$ red path.

Now consider the $i^{th}$ blue path and $(i+1)^{th}$ red path. Note that the blue path starts and ends strictly above the red path. Suppose after some step the paths are a the same point. As the blue paths ends above the red path, there is some step where the paths separate with the blue path becoming above the red path. This would introduce a coupled pair  of either type one or four.  Thus the $i^{th}$ blue path must be strictly above the $(i+1)^{th}$ red path.
\end{proof}
Note that it is also true that if paths satisfy the constraints in the above lemma then there are no coupled pairs of lozenges.

Using the above lemma, we can show that setting the interaction strength to zero force main of the entries in the RPPs to be zero.
\begin{lemma}\label{lem:t0contraints2}
When the interaction strength $t=0$, all entries in the first $i$ columns and first $i$ rows are zero of the $i^{th}$ blue border strip and $(i+1)^{th}$ red border strip are zero.
\end{lemma}
\begin{proof}
Consider the first path of each color. The leftmost step of the blue path cannot be vertical as there would be a coupled pair of either the first or second type in \eqref{eq:coupledpair}. The rightmost step of the blue path also cannot be vertical. To see this note that if both blue and red were vertical then we would have a coupled pair of the third type. If the blue was vertical and red was not, then this would violate the constraint from Lemma \ref{lem:t0contraints1} that  blue path is weakly below the red path.

In the RPP, these constraints means that the last entry in the first column and the last entry in the first row are both zero. Note that in order to obey the constraints of the RPP, this forces every entry in the first column and first row of the blue RPP to be zero. Thus the lemma holds for the first blue path.

Now consider the second red path. By Lemma \ref{lem:t0contraints1} and the fact the all entries of the first blue border strip in the first column are zero, we see that the red path cannot take a vertical step upward as it would meet the first blue path. Similarly, the last step could not be vertical as it would mean the red path was previously at the same point as the as the first blue path. We see that the first and last entry along the second red border strip are zero. The constraints defining the RPP imply the entries of every subsequent red border strip in the first row or column are also zero.

Now consider the $j^{th}$ blue border strip and $(j+1)^{th}$ red border strip and suppose the result holds for all previous border strips. Note that his implies that the entries in the first $j-1$ rows and columns are zero. Consider the pair of RPP we get by removing the first $j-1$ border strips and the first $j-1$ rows and columns. It is enough to show that that the result holds for the first row and column of the new RPPs. But this follows from exactly the same argument as above. 

By induction we get the lemma.
\end{proof}

\begin{proposition}\label{prop:sliding}
For each $i$, shift the $i^{th}$ blue path and the $(i-1)^{th}$ red path vertically down $i$ steps and remove the segments forced to be at height zero. The resulting configuration maps gives a valid RPP of the same shape.
\end{proposition}
\begin{proof}
In terms of filling in the RPP, shifting a paths down one step corresponds to shifting a border strip diagonally down and left one step. By Lemma \ref{lem:t0contraints1}, we have that all the entries of the RPP that are shifted off of the Young diagram were in fact forced to be zero due to the interaction strength being set to zero.

Note that if we take the $i^{th}$ border strip of $\lambda$, shift it down and left one step, and remove the cells that leave the Young diagram, the result is that it transforms into the $(i+1)^{th}$ border strip. Thus the tableaux we get by filling the $(2i-1)^{th}$ border strip with the shifted $i^{th}$ red border strip and the $2i^{th}$ border strip with the shifted $i^{th}$ blue border strip is at least the correct shape. 

We are left to check that the entries satisfy the inequalities defining a RPP, namely, that the are weakly increasing moving right along rows and up along columns. To do this, we will show that any triple of entries of the form
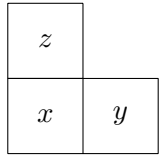
\[
\begin{tikzpicture}
\draw (0,0) grid (2,1);
\draw (0,1) grid (1,2);
\node at (0.5,0.5) {$x$};
\node at (1.5,0.5) {$y$};
\node at (0.5,1.5) {$z$};
\end{tikzpicture}
\]
satisfies $x\le y$ and $x\le z$. First, we will show  $x\le y$. There are three cases we must consider.

\noindent \textbf{Case 1: $x$ and $y$ are part of the same border strip.} In this case, since the border strip was part of one of the colored RPP, we get the inequality $x\le y$ automatically. 

\noindent \textbf{Case 2: $x$ is part of the $i^{th}$ blue border strip and $y$ is part of the $i^{th}$ red border strip.} In this case, we must have a cell above $y$ that is also part of the $i^{th}$ red border strip. Call this entry $w$. By Lemma \ref{lem:t0contraints1}, we must have $x\le w$. Suppose $x>y$. Then the resulting lozenges, before shifting, superimposed on top of one another would take the form
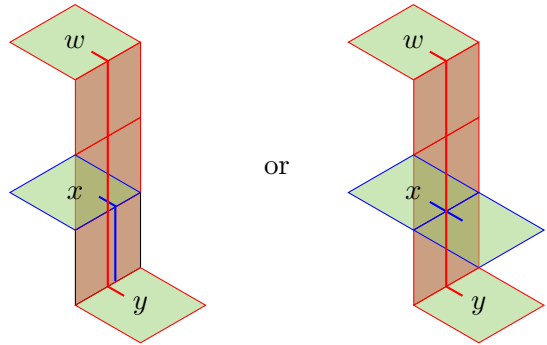
\[
\begin{tikzpicture}[fill opacity=0.5, baseline=(current bounding box.center)]
\fill[RawSienna,draw=black,shift={(0.87,-0.5)}] (0,0)--(0,1)--(150:1)--(-150:1)--(0,0);
\fill[RawSienna,draw=red,shift={(0.87,0.5)}] (0,0)--(0,1)--(150:1)--(-150:1)--(0,0);
\fill[RawSienna,draw=red,shift={(0.87,1.5)}] (0,0)--(0,1)--(150:1)--(-150:1)--(0,0);
\fill[YellowGreen,draw=blue] (0,0)--(150:1)--(0,1)--(30:1)--(0,0); \node[fill opacity=1.0] at (0,0.5) {$x$};
\fill[YellowGreen,draw=red,shift={(0,2)}] (0,0)--(150:1)--(0,1)--(30:1)--(0,0); \node[fill opacity=1.0] at (0,2.5) {$w$};
 \fill[YellowGreen,draw=red,shift={(0.87,-1.5)}] (0,0)--(150:1)--(0,1)--(30:1)--(0,0); \node[fill opacity=1.0] at (0.87,-1) {$y$};
\draw[blue,thick,shift={(0.1,0.07)}] (0.2175,0.375)--(0.435,0.25)--(0.435,-0.75);
\draw[red,thick] (0.2175,2.375)--(0.435,2.25)--(0.435,-0.75)--(0.6525,-0.875);
\end{tikzpicture}
\qquad \text{or} \qquad
\begin{tikzpicture}[fill opacity=0.5, baseline=(current bounding box.center)]
\fill[RawSienna,draw=red,shift={(0.87,-0.5)}] (0,0)--(0,1)--(150:1)--(-150:1)--(0,0);
\fill[RawSienna,draw=red,shift={(0.87,0.5)}] (0,0)--(0,1)--(150:1)--(-150:1)--(0,0);
\fill[RawSienna,draw=red,shift={(0.87,1.5)}] (0,0)--(0,1)--(150:1)--(-150:1)--(0,0);
\fill[YellowGreen,draw=blue] (0,0)--(150:1)--(0,1)--(30:1)--(0,0); \node[fill opacity=1.0] at (0,0.5) {$x$};
\fill[YellowGreen,draw=blue,shift={(0.87,-0.5)}] (0,0)--(150:1)--(0,1)--(30:1)--(0,0); 
\fill[YellowGreen,draw=red,shift={(0,2)}] (0,0)--(150:1)--(0,1)--(30:1)--(0,0); \node[fill opacity=1.0] at (0,2.5) {$w$};
 \fill[YellowGreen,draw=red,shift={(0.87,-1.5)}] (0,0)--(150:1)--(0,1)--(30:1)--(0,0); \node[fill opacity=1.0] at (0.87,-1) {$y$};
\draw[blue,thick] (0.2175,0.375)--(0.435,0.25)--(0.6525,0.125);
\draw[red,thick] (0.2175,2.375)--(0.435,2.25)--(0.435,-0.75)--(0.6525,-0.875);
\end{tikzpicture}
\]
Note that there is a coupled pair of type 3 in the first configuration and a coupled pair of type 4 int the second configuration. This is a contradiction, so we must have $x\le y$.

\noindent \textbf{Case 3: $x$ is part of the $(i+1)^{th}$ red border strip and $y$ is part of the $i^{th}$ blue border strip.} Just as in the previous case, we must have a the cell above $y$ that is also part of the $i^{th}$ blue border strip with entry $w$. By Lemma \ref{lem:t0contraints1}, we know $x\le w$. Again, suppose $x>y$. Looking at the lozenge configurations, we have
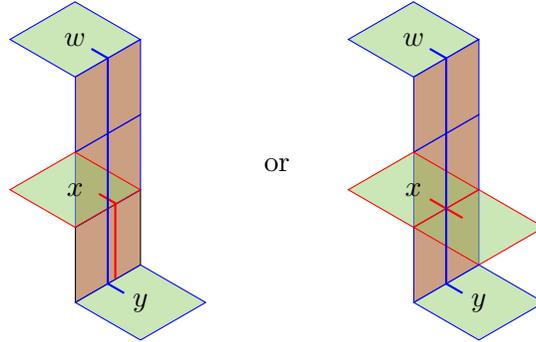
\[
\begin{tikzpicture}[fill opacity=0.5, baseline=(current bounding box.center)]
\fill[RawSienna,draw=black,shift={(0.87,-0.5)}] (0,0)--(0,1)--(150:1)--(-150:1)--(0,0);
\fill[RawSienna,draw=blue,shift={(0.87,0.5)}] (0,0)--(0,1)--(150:1)--(-150:1)--(0,0);
\fill[RawSienna,draw=blue,shift={(0.87,1.5)}] (0,0)--(0,1)--(150:1)--(-150:1)--(0,0);
\fill[YellowGreen,draw=red] (0,0)--(150:1)--(0,1)--(30:1)--(0,0); \node[fill opacity=1.0] at (0,0.5) {$x$};
\fill[YellowGreen,draw=blue,shift={(0,2)}] (0,0)--(150:1)--(0,1)--(30:1)--(0,0); \node[fill opacity=1.0] at (0,2.5) {$w$};
 \fill[YellowGreen,draw=blue,shift={(0.87,-1.5)}] (0,0)--(150:1)--(0,1)--(30:1)--(0,0); \node[fill opacity=1.0] at (0.87,-1) {$y$};
\draw[red,thick,shift={(0.1,0.07)}] (0.2175,0.375)--(0.435,0.25)--(0.435,-0.75);
\draw[blue,thick] (0.2175,2.375)--(0.435,2.25)--(0.435,-0.75)--(0.6525,-0.875);
\end{tikzpicture}
\qquad \text{or} \qquad
\begin{tikzpicture}[fill opacity=0.5, baseline=(current bounding box.center)]
\fill[RawSienna,draw=blue,shift={(0.87,-0.5)}] (0,0)--(0,1)--(150:1)--(-150:1)--(0,0);
\fill[RawSienna,draw=blue,shift={(0.87,0.5)}] (0,0)--(0,1)--(150:1)--(-150:1)--(0,0);
\fill[RawSienna,draw=blue,shift={(0.87,1.5)}] (0,0)--(0,1)--(150:1)--(-150:1)--(0,0);
\fill[YellowGreen,draw=red] (0,0)--(150:1)--(0,1)--(30:1)--(0,0); \node[fill opacity=1.0] at (0,0.5) {$x$};
\fill[YellowGreen,draw=red,shift={(0.87,-0.5)}] (0,0)--(150:1)--(0,1)--(30:1)--(0,0); 
\fill[YellowGreen,draw=blue,shift={(0,2)}] (0,0)--(150:1)--(0,1)--(30:1)--(0,0); \node[fill opacity=1.0] at (0,2.5) {$w$};
 \fill[YellowGreen,draw=blue,shift={(0.87,-1.5)}] (0,0)--(150:1)--(0,1)--(30:1)--(0,0); \node[fill opacity=1.0] at (0.87,-1) {$y$};
\draw[red,thick] (0.2175,0.375)--(0.435,0.25)--(0.6525,0.125);
\draw[blue,thick] (0.2175,2.375)--(0.435,2.25)--(0.435,-0.75)--(0.6525,-0.875);
\end{tikzpicture}
\]
We see that the red path meets the blue path contradicting the second point of Lemma \ref{lem:t0contraints1}. Again, we must have $x\le y$.

Now, we will show $x\le z$. We have the same three cases for $z$.

\noindent \textbf{Case 1: $x$ and $z$ are part of the same border strip.} In this case, since the border strip was part of one of the colored RPP, we get the inequality $x\le z$ automatically just as in the case for $y$. 

\noindent \textbf{Case 2: $x$ is part of the $i^{th}$ blue border strip and $z$ is part of the $i^{th}$ red border strip.} In this case, we must have a cell to the right of $z$ that is also part of the $i^{th}$ red border strip. Call the entry in this cell $w$. By Lemma \ref{lem:t0contraints1}, we must have $x\le w$. Suppose $x>z$. The resulting lozenges before shifting are given by
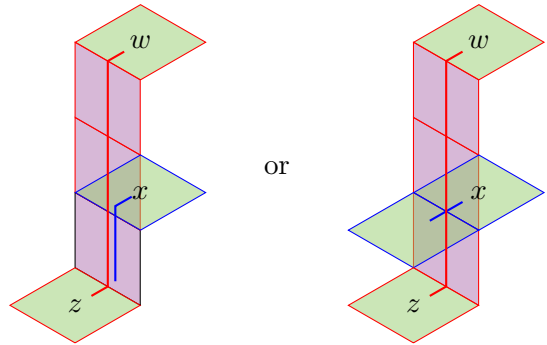
\[
\begin{tikzpicture}[fill opacity=0.5, baseline=(current bounding box.center)]
\fill[Orchid,draw=black,shift={(-0.87,-0.5)}] (0,0)--(0,1)--(30:1)--(-30:1)--(0,0);
\fill[Orchid,draw=red,shift={(-0.87,0.5)}] (0,0)--(0,1)--(30:1)--(-30:1)--(0,0);
\fill[Orchid,draw=red,shift={(-0.87,1.5)}] (0,0)--(0,1)--(30:1)--(-30:1)--(0,0);
\fill[YellowGreen,draw=blue] (0,0)--(150:1)--(0,1)--(30:1)--(0,0); \node[fill opacity=1.0] at (0,0.5) {$x$};
\fill[YellowGreen,draw=red,shift={(0,2)}] (0,0)--(150:1)--(0,1)--(30:1)--(0,0); \node[fill opacity=1.0] at (0,2.5) {$w$};
\fill[YellowGreen,draw=red,shift={(-0.87,-1.5)}] (0,0)--(150:1)--(0,1)--(30:1)--(0,0); \node[fill opacity=1.0] at (-0.87,-1) {$z$};
\draw[blue,thick,shift={(0.1,0.07)}] (-0.2175,0.375)--(-0.435,0.25)--(-0.435,-0.75);
\draw[red,thick] (-0.2175,2.375)--(-0.435,2.25)--(-0.435,-0.75)--(-0.6525,-0.875);
\end{tikzpicture}
\qquad \text{or} \qquad
\begin{tikzpicture}[fill opacity=0.5, baseline=(current bounding box.center)]
\fill[Orchid,draw=red,shift={(-0.87,-0.5)}] (0,0)--(0,1)--(30:1)--(-30:1)--(0,0);
\fill[Orchid,draw=red,shift={(-0.87,0.5)}] (0,0)--(0,1)--(30:1)--(-30:1)--(0,0);
\fill[Orchid,draw=red,shift={(-0.87,1.5)}] (0,0)--(0,1)--(30:1)--(-30:1)--(0,0);
\fill[YellowGreen,draw=blue] (0,0)--(150:1)--(0,1)--(30:1)--(0,0); \node[fill opacity=1.0] at (0,0.5) {$x$};
\fill[YellowGreen,draw=blue,shift={(-0.87,-0.5)}] (0,0)--(150:1)--(0,1)--(30:1)--(0,0); 
\fill[YellowGreen,draw=red,shift={(0,2)}] (0,0)--(150:1)--(0,1)--(30:1)--(0,0); \node[fill opacity=1.0] at (0,2.5) {$w$};
\fill[YellowGreen,draw=red,shift={(-0.87,-1.5)}] (0,0)--(150:1)--(0,1)--(30:1)--(0,0); \node[fill opacity=1.0] at (-0.87,-1) {$z$};
\draw[blue,thick] (-0.2175,0.375)--(-0.435,0.25)--(-0.6525,0.125);
\draw[red,thick] (-0.2175,2.375)--(-0.435,2.25)--(-0.435,-0.75)--(-0.6525,-0.875);
\end{tikzpicture}
\]
The first configuration gives a coupled pair of type 2 and the second configuration violates the first point of Lemma \ref{lem:t0contraints1}.

\noindent \textbf{Case 3: $x$ is part of the $(i+1)^{th}$ red border strip and $z$ is part of the $i^{th}$ blue border strip.} We must have a the cell above $z$ that is also part of the $i^{th}$ blue border strip with entry we will denote by $w$. By Lemma \ref{lem:t0contraints1}, we know $x\le w$. Again, suppose $x>y$. Looking at the lozenge configurations, we have
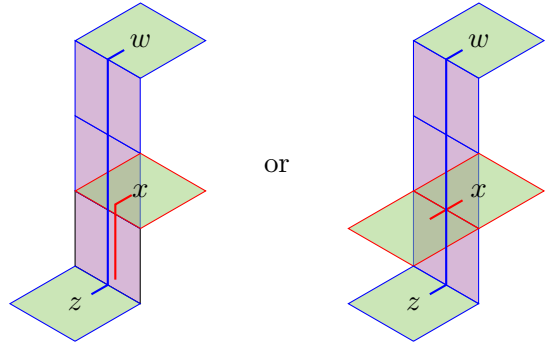
\[
\begin{tikzpicture}[fill opacity=0.5, baseline=(current bounding box.center)]
\fill[Orchid,draw=black,shift={(-0.87,-0.5)}] (0,0)--(0,1)--(30:1)--(-30:1)--(0,0);
\fill[Orchid,draw=blue,shift={(-0.87,0.5)}] (0,0)--(0,1)--(30:1)--(-30:1)--(0,0);
\fill[Orchid,draw=blue,shift={(-0.87,1.5)}] (0,0)--(0,1)--(30:1)--(-30:1)--(0,0);
\fill[YellowGreen,draw=red] (0,0)--(150:1)--(0,1)--(30:1)--(0,0); \node[fill opacity=1.0] at (0,0.5) {$x$};
\fill[YellowGreen,draw=blue,shift={(0,2)}] (0,0)--(150:1)--(0,1)--(30:1)--(0,0); \node[fill opacity=1.0] at (0,2.5) {$w$};
 \fill[YellowGreen,draw=blue,shift={(-0.87,-1.5)}] (0,0)--(150:1)--(0,1)--(30:1)--(0,0); \node[fill opacity=1.0] at (-0.87,-1) {$z$};
\draw[red,thick,shift={(0.1,0.07)}] (-0.2175,0.375)--(-0.435,0.25)--(-0.435,-0.75);
\draw[blue,thick] (-0.2175,2.375)--(-0.435,2.25)--(-0.435,-0.75)--(-0.6525,-0.875);
\end{tikzpicture}
\qquad \text{or} \qquad
\begin{tikzpicture}[fill opacity=0.5, baseline=(current bounding box.center)]
\fill[Orchid,draw=blue,shift={(-0.87,-0.5)}] (0,0)--(0,1)--(30:1)--(-30:1)--(0,0);
\fill[Orchid,draw=blue,shift={(-0.87,0.5)}] (0,0)--(0,1)--(30:1)--(-30:1)--(0,0);
\fill[Orchid,draw=blue,shift={(-0.87,1.5)}] (0,0)--(0,1)--(30:1)--(-30:1)--(0,0);
\fill[YellowGreen,draw=red] (0,0)--(150:1)--(0,1)--(30:1)--(0,0); \node[fill opacity=1.0] at (0,0.5) {$x$};
\fill[YellowGreen,draw=red,shift={(-0.87,-0.5)}] (0,0)--(150:1)--(0,1)--(30:1)--(0,0); 
\fill[YellowGreen,draw=blue,shift={(0,2)}] (0,0)--(150:1)--(0,1)--(30:1)--(0,0); \node[fill opacity=1.0] at (0,2.5) {$w$};
 \fill[YellowGreen,draw=blue,shift={(-0.87,-1.5)}] (0,0)--(150:1)--(0,1)--(30:1)--(0,0); \node[fill opacity=1.0] at (-0.87,-1) {$z$};
\draw[red,thick] (-0.2175,0.375)--(-0.435,0.25)--(-0.6525,0.125);
\draw[blue,thick] (-0.2175,2.375)--(-0.435,2.25)--(-0.435,-0.75)--(-0.6525,-0.875);
\end{tikzpicture}
\]
As with $y$, these violate the second point of Lemma \ref{lem:t0contraints1}.

In any case we must have $x\le z$. Altogether this shows that the inequalities of the RPP are satisfied, so the resulting filling gives a valid RPP.
\end{proof}

Example \ref{ex:shifted} gives an example of this construction. Note that this is reversible. We can take the paths of a single RPP and alternately color the border strips red and blue. Then we may reverse the slides and add or extend any border strips with zeros where necessary. The same case work as in the proof of Proposition \ref{prop:sliding} shows that there will be no coupled pairs of lozenges. Altogether this gives a bijective proof of Theorem \ref{thm:t0}.

\begin{example} \label{ex:shifted}
Consider the pair of RPPs of shape $\lambda=(4,4,3,3,1)$ shown below 
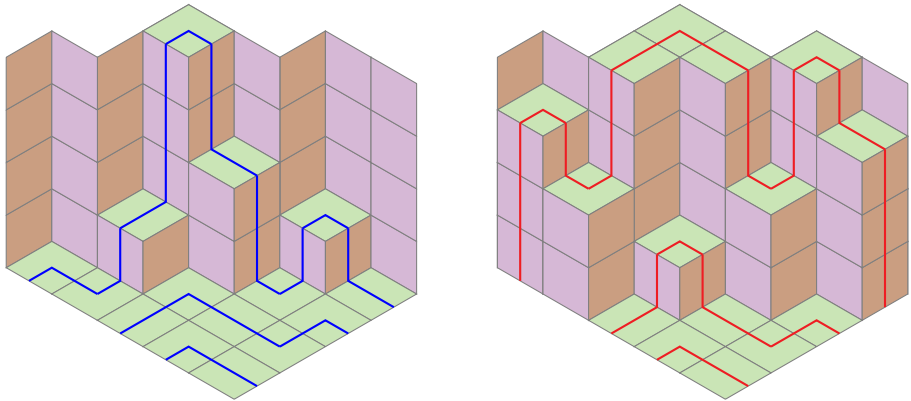
\[
\begin{tikzpicture}[scale=0.7]
    \reverseplanepartition{4}{{0},{0,1,4},{0,0,2},{0,0,0,1},{0,0,0,0}}
    \fill[White,fill opacity=0.5] (-0.9,-3.1) rectangle (7,4.6);
    \draw[blue,thick,shift={(0,-0.5)}] (-150:0.5)--(0,0)--(-30:1);
    \draw[blue,thick,shift={(-30:1.5)}] (-150:0.5)--(0,0)--(0,1);
    \draw[blue,thick,shift={(30:2)},shift={(-30:0.5)}] (-150:1)--(0,0)--(0,3);
    \draw[blue,thick,shift={(30:3)},shift={(0,2.5)}] (-150:0.5)--(0,0)--(-30:0.5);
    \draw[blue,thick,shift={(30:3.5)}] (0,2)--(0,0)--(-30:1);
    \draw[blue,thick,shift={(-30:4.5)},shift={(0,1.5)}] (0,2)--(0,0)--(-30:0.5);
    \draw[blue,thick,shift={(-30:3.5)},shift={(30:2)}] (-150:0.5)--(0,0)--(0,1);
    \draw[blue,thick,shift={(-30:2.5)},shift={(30:3.5)}] (-150:0.5)--(0,0)--(-30:0.5);
    \draw[blue,thick,shift={(-30:4)},shift={(30:2.5)}] (0,1)--(0,0)--(-30:1);
    \draw[blue,thick,shift={(-30:2.5)},shift={(30:0.5)}] (-150:1.5)--(0,0)--(-30:2);
    \draw[blue,thick,shift={(-30:4.5)},shift={(30:1.5)}] (-150:1)--(0,0)--(-30:0.5);
    \draw[blue,thick,shift={(-30:3.5)},shift={(-150:0.5)}] (-150:0.5)--(0,0)--(-30:1.5);
\end{tikzpicture}\hspace{1cm}
\begin{tikzpicture}[scale=0.7]
    \reverseplanepartition{4}{{3},{2,4,4},{0,1,4},{0,0,2,4},{0,0,0,3}}
    \fill[White,fill opacity=0.5] (-0.9,-3.1) rectangle (7,4.6);
    \draw[Red,thick,shift={(150:0.5)}] (0,-1)--(0,2);
    \draw[Red,thick,shift={(0,2.5)}] (-150:0.5)--(0,0)--(-30:0.5);
    \draw[Red,thick,shift={(0,1)},shift={(30:0.5)}] (0,1)--(0,0)--(-30:0.5);
    \draw[Red,thick,shift={(30:1.5)},shift={(0,0.5)}] (-150:0.5)--(0,0)--(0,2);
    \draw[Red,thick,shift={(30:3)},shift={(0,2.5)}] (-150:1.5)--(0,0)--(-30:1.5);
    \draw[Red,thick,shift={(30:4.5)},shift={(0,-1)}] (0,2)--(0,0)--(-30:0.5);
    \draw[Red,thick,shift={(30:4)},shift={(-30:1.5)}] (-150:0.5)--(0,0)--(0,2);
    \draw[Red,thick,shift={(30:6)},shift={(0,0.5)}] (-150:0.5)--(0,0)--(-30:0.5);
    \draw[Red,thick,shift={(30:6.5)},shift={(0,-1)}] (0,1)--(0,0)--(-30:1);
    \draw[Red,thick,shift={(30:2.5)},shift={(-30:5)}] (0,3)--(0,0);
    \draw[Red,thick,shift={(-30:2.5)}] (-150:1)--(0,0)--(0,1);
    \draw[Red,thick,shift={(30:3)},shift={(0,-1.5)}] (-150:0.5)--(0,0)--(-30:0.5);
    \draw[Red,thick,shift={(-30:3)},shift={(30:0.5)}] (0,1)--(0,0)--(-30:1.5);
    \draw[Red,thick,shift={(-30:4.5)},shift={(30:1.5)}] (-150:1)--(0,0)--(-30:0.5);
    \draw[Red,thick,shift={(-30:3.5)},shift={(-150:0.5)}] (-150:0.5)--(0,0)--(-30:1.5);
\end{tikzpicture}
\]
where one can check that there are no coupled pairs of lozenges. When we superimpose the path on top of one another we get the figure on the LHS below. After removing the portions of path corresponding to the forced zero height and shifting the path according to Proposition \ref{prop:sliding} we get the figure on the RHS below.
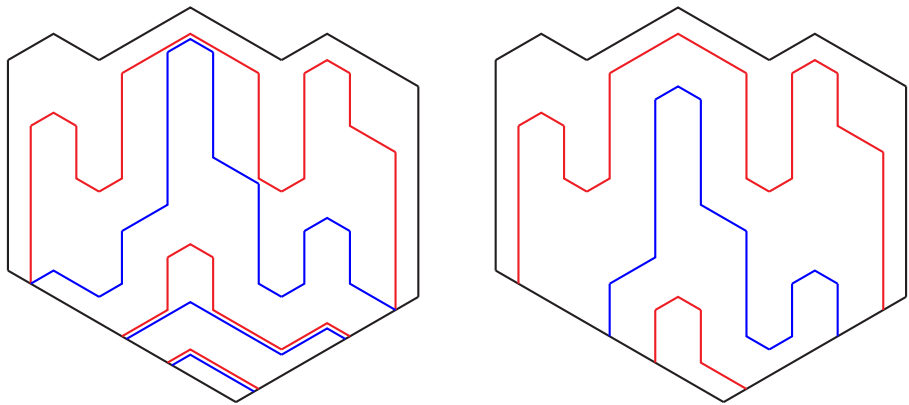
\[
\begin{tikzpicture}[scale=0.7]
    \draw[Black,thick,shift={(0,-1)}] (150:1)--(-30:4);
    \draw[Black,thick,shift={(-30:5)}] (-150:1)--(30:3);
    \draw[Black,thick,shift={(150:1)}] (0,-1)--(0,3);
    \draw[Black,thick,shift={(0,4)}] (-150:1)--(0,0)--(-30:1);
    \draw[Black,thick,shift={(30:3)},shift={(0,3)}] (-150:2)--(0,0)--(-30:2);
    \draw[Black,thick,shift={(30:6)},shift={(0,1)}] (-150:1)--(0,0)--(-30:2);
    \draw[Black,thick,shift={(30:8)}] (0,-1)--(0,-5);
    \draw[Red,thick,shift={(150:0.5)}] (0,-1)--(0,2);
    \draw[Red,thick,shift={(0,2.5)}] (-150:0.5)--(0,0)--(-30:0.5);
    \draw[Red,thick,shift={(0,1)},shift={(30:0.5)}] (0,1)--(0,0)--(-30:0.5);
    \draw[Red,thick,shift={(30:1.5)},shift={(0,0.5)}] (-150:0.5)--(0,0)--(0,2);
    \draw[Red,thick,shift={(30:3)},shift={(0,2.5)}] (-150:1.5)--(0,0)--(-30:1.5);
    \draw[Red,thick,shift={(30:4.5)},shift={(0,-1)}] (0,2)--(0,0)--(-30:0.5);
    \draw[Red,thick,shift={(30:4)},shift={(-30:1.5)}] (-150:0.5)--(0,0)--(0,2);
    \draw[Red,thick,shift={(30:6)},shift={(0,0.5)}] (-150:0.5)--(0,0)--(-30:0.5);
    \draw[Red,thick,shift={(30:6.5)},shift={(0,-1)}] (0,1)--(0,0)--(-30:1);
    \draw[Red,thick,shift={(30:2.5)},shift={(-30:5)}] (0,3)--(0,0);
    \draw[Red,thick,shift={(-30:2.5)}] (-150:1)--(0,0)--(0,1);
    \draw[Red,thick,shift={(30:3)},shift={(0,-1.5)}] (-150:0.5)--(0,0)--(-30:0.5);
    \draw[Red,thick,shift={(-30:3)},shift={(30:0.5)}] (0,1)--(0,0)--(-30:1.5);
    \draw[Red,thick,shift={(-30:4.5)},shift={(30:1.5)}] (-150:1)--(0,0)--(-30:0.5);
    \draw[Red,thick,shift={(-30:3.5)},shift={(-150:0.5)}] (-150:0.5)--(0,0)--(-30:1.5);
    \draw[blue,thick,shift={(0,-0.5)}] (-150:0.5)--(0,0)--(-30:1);
    \draw[blue,thick,shift={(-30:1.5)}] (-150:0.5)--(0,0)--(0,1);
    \draw[blue,thick,shift={(30:2)},shift={(-30:0.5)}] (-150:1)--(0,0)--(0,2.9);
    \draw[blue,thick,shift={(30:3)},shift={(0,2.4)}] (-150:0.5)--(0,0)--(-30:0.5);
    \draw[blue,thick,shift={(30:3.5)},shift={(0,-0.1)}] (0,2)--(0,0)--(-30:1);
    \draw[blue,thick,shift={(-30:4.5)},shift={(0,1.5)}] (0,1.9)--(0,0)--(-30:0.5);
    \draw[blue,thick,shift={(-30:3.5)},shift={(30:2)}] (-150:0.5)--(0,0)--(0,1);
    \draw[blue,thick,shift={(-30:2.5)},shift={(30:3.5)}] (-150:0.5)--(0,0)--(-30:0.5);
    \draw[blue,thick,shift={(-30:4)},shift={(30:2.5)}] (0,1)--(0,0)--(-30:1);
    \draw[blue,thick,shift={(-30:2.5)},shift={(30:0.5)},shift={(0,-0.1)}] (-150:1.4)--(0,0)--(-30:2);
    \draw[blue,thick,shift={(-30:4.5)},shift={(30:1.5)},shift={(0,-0.1)}] (-150:1)--(0,0)--(-30:0.4);
    \draw[blue,thick,shift={(-30:3.5)},shift={(-150:0.5)},shift={(0,-0.1)}] (-150:0.4)--(0,0)--(-30:1.4);
\end{tikzpicture}\hspace{1cm}
\begin{tikzpicture}[scale=0.7]
    \draw[Black,thick,shift={(0,-1)}] (150:1)--(-30:4);
    \draw[Black,thick,shift={(-30:5)}] (-150:1)--(30:3);
    \draw[Black,thick,shift={(150:1)}] (0,-1)--(0,3);
    \draw[Black,thick,shift={(0,4)}] (-150:1)--(0,0)--(-30:1);
    \draw[Black,thick,shift={(30:3)},shift={(0,3)}] (-150:2)--(0,0)--(-30:2);
    \draw[Black,thick,shift={(30:6)},shift={(0,1)}] (-150:1)--(0,0)--(-30:2);
    \draw[Black,thick,shift={(30:8)}] (0,-1)--(0,-5);
    \draw[Red,thick,shift={(150:0.5)}] (0,-1)--(0,2);
    \draw[Red,thick,shift={(0,2.5)}] (-150:0.5)--(0,0)--(-30:0.5);
    \draw[Red,thick,shift={(0,1)},shift={(30:0.5)}] (0,1)--(0,0)--(-30:0.5);
    \draw[Red,thick,shift={(30:1.5)},shift={(0,0.5)}] (-150:0.5)--(0,0)--(0,2);
    \draw[Red,thick,shift={(30:3)},shift={(0,2.5)}] (-150:1.5)--(0,0)--(-30:1.5);
    \draw[Red,thick,shift={(30:4.5)},shift={(0,-1)}] (0,2)--(0,0)--(-30:0.5);
    \draw[Red,thick,shift={(30:4)},shift={(-30:1.5)}] (-150:0.5)--(0,0)--(0,2);
    \draw[Red,thick,shift={(30:6)},shift={(0,0.5)}] (-150:0.5)--(0,0)--(-30:0.5);
    \draw[Red,thick,shift={(30:6.5)},shift={(0,-1)}] (0,1)--(0,0)--(-30:1);
    \draw[Red,thick,shift={(30:2.5)},shift={(-30:5)}] (0,3)--(0,0);
    \draw[Red,thick,shift={(-30:2.5)},shift={(0,-1)}] (0,0)--(0,1);
    \draw[Red,thick,shift={(30:3)},shift={(0,-1.5)},shift={(0,-1)}] (-150:0.5)--(0,0)--(-30:0.5);
    \draw[Red,thick,shift={(-30:3)},shift={(30:0.5)},shift={(0,-1)}] (0,1)--(0,0)--(-30:1);
    \draw[blue,thick,shift={(-30:1.5)},shift={(0,-1)}] (0,0)--(0,1);
    \draw[blue,thick,shift={(30:2)},shift={(-30:0.5)},shift={(0,-1)}] (-150:1)--(0,0)--(0,3);
    \draw[blue,thick,shift={(30:3)},shift={(0,2.5)},shift={(0,-1)}] (-150:0.5)--(0,0)--(-30:0.5);
    \draw[blue,thick,shift={(30:3.5)},shift={(0,-1)}] (0,2)--(0,0)--(-30:1);
    \draw[blue,thick,shift={(-30:4.5)},shift={(0,1.5)},shift={(0,-1)}] (0,2)--(0,0)--(-30:0.5);
    \draw[blue,thick,shift={(-30:3.5)},shift={(30:2)},shift={(0,-1)}] (-150:0.5)--(0,0)--(0,1);
    \draw[blue,thick,shift={(-30:2.5)},shift={(30:3.5)},shift={(0,-1)}] (-150:0.5)--(0,0)--(-30:0.5);
    \draw[blue,thick,shift={(-30:4)},shift={(30:2.5)},shift={(0,-1)}] (0,1)--(0,0);
\end{tikzpicture}
\]
The newly constructed path configuration corresponds to the lozenge tiling:
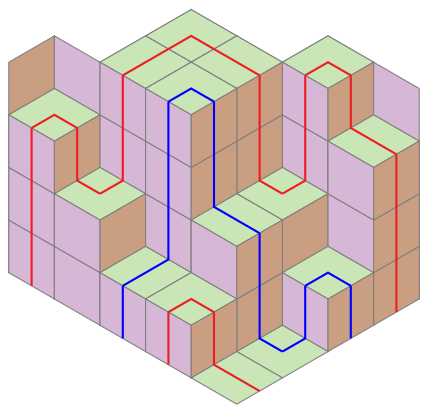
\[
\begin{tikzpicture}[scale=0.7]
    \reverseplanepartition{4}{{3},{2,4,4},{1,4,4},{1,2,2,4},{0,0,1,3}}
    \fill[White,fill opacity=0.5] (-0.9,-3.1) rectangle (7,4.6);
    \draw[Red,thick,shift={(150:0.5)}] (0,-1)--(0,2);
    \draw[Red,thick,shift={(0,2.5)}] (-150:0.5)--(0,0)--(-30:0.5);
    \draw[Red,thick,shift={(0,1)},shift={(30:0.5)}] (0,1)--(0,0)--(-30:0.5);
    \draw[Red,thick,shift={(30:1.5)},shift={(0,0.5)}] (-150:0.5)--(0,0)--(0,2);
    \draw[Red,thick,shift={(30:3)},shift={(0,2.5)}] (-150:1.5)--(0,0)--(-30:1.5);
    \draw[Red,thick,shift={(30:4.5)},shift={(0,-1)}] (0,2)--(0,0)--(-30:0.5);
    \draw[Red,thick,shift={(30:4)},shift={(-30:1.5)}] (-150:0.5)--(0,0)--(0,2);
    \draw[Red,thick,shift={(30:6)},shift={(0,0.5)}] (-150:0.5)--(0,0)--(-30:0.5);
    \draw[Red,thick,shift={(30:6.5)},shift={(0,-1)}] (0,1)--(0,0)--(-30:1);
    \draw[Red,thick,shift={(30:2.5)},shift={(-30:5)}] (0,3)--(0,0);
    \draw[Red,thick,shift={(-30:2.5)},shift={(0,-1)}] (0,0)--(0,1);
    \draw[Red,thick,shift={(30:3)},shift={(0,-1.5)},shift={(0,-1)}] (-150:0.5)--(0,0)--(-30:0.5);
    \draw[Red,thick,shift={(-30:3)},shift={(30:0.5)},shift={(0,-1)}] (0,1)--(0,0)--(-30:1);
    \draw[blue,thick,shift={(-30:1.5)},shift={(0,-1)}] (0,0)--(0,1);
    \draw[blue,thick,shift={(30:2)},shift={(-30:0.5)},shift={(0,-1)}] (-150:1)--(0,0)--(0,3);
    \draw[blue,thick,shift={(30:3)},shift={(0,2.5)},shift={(0,-1)}] (-150:0.5)--(0,0)--(-30:0.5);
    \draw[blue,thick,shift={(30:3.5)},shift={(0,-1)}] (0,2)--(0,0)--(-30:1);
    \draw[blue,thick,shift={(-30:4.5)},shift={(0,1.5)},shift={(0,-1)}] (0,2)--(0,0)--(-30:0.5);
    \draw[blue,thick,shift={(-30:3.5)},shift={(30:2)},shift={(0,-1)}] (-150:0.5)--(0,0)--(0,1);
    \draw[blue,thick,shift={(-30:2.5)},shift={(30:3.5)},shift={(0,-1)}] (-150:0.5)--(0,0)--(-30:0.5);
    \draw[blue,thick,shift={(-30:4)},shift={(30:2.5)},shift={(0,-1)}] (0,1)--(0,0);
\end{tikzpicture}
\]
In terms of the border strip decomposition, we have
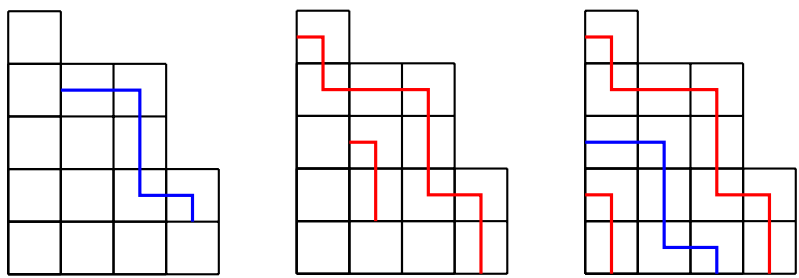
\[
\begin{tikzpicture}[scale=0.7]
    \draw[black,thick] (0,0) grid (3,4);
    \draw[black,thick] (0,0) grid (4,2);
    \draw[black,thick] (0,0) grid (1,5);
    \draw[blue,very thick] (1,3.5)--(2.5,3.5)--(2.5,1.5)--(3.5,1.5)--(3.5,1);
\end{tikzpicture}\hspace{1cm}
\begin{tikzpicture}[scale=0.7]
    \draw[black,thick] (0,0) grid (3,4);
    \draw[black,thick] (0,0) grid (4,2);
    \draw[black,thick] (0,0) grid (1,5);
    \draw[red,very thick] (0,4.5)--(0.5,4.5)--(0.5,3.5)--(2.5,3.5)--(2.5,1.5)--(3.5,1.5)--(3.5,0);
    \draw[red,very thick] (1,2.5)--(1.5,2.5)--(1.5,1);
\end{tikzpicture}\hspace{1cm}
\begin{tikzpicture}[scale=0.7]
    \draw[black,thick] (0,0) grid (3,4);
    \draw[black,thick] (0,0) grid (4,2);
    \draw[black,thick] (0,0) grid (1,5);
    \draw[red,very thick] (0,4.5)--(0.5,4.5)--(0.5,3.5)--(2.5,3.5)--(2.5,1.5)--(3.5,1.5)--(3.5,0);
    \draw[blue,very thick] (0,2.5)--(1.5,2.5)--(1.5,0.5)--(2.5,0.5)--(2.5,0);
    \draw[red,very thick] (0,1.5)--(0.5,1.5)--(0.5,0);
\end{tikzpicture}
\]
where the first two figures show the parts of the border strips not forced to be zero and the rightmost image is the result after sliding.
\end{example}

\bibliographystyle{abbrv}
\bibliography{InteractingRPP}

\end{document}